\documentclass{blank_nmeauth}
\usepackage{amssymb,amsmath,mathrsfs,amsthm}
\usepackage{algorithm}
\usepackage{algpseudocode}
\usepackage{array}
\usepackage{cancel}
\usepackage{enumerate}
\usepackage{booktabs}
\usepackage{multirow}
\usepackage{threeparttable}
\usepackage{graphicx}
\usepackage{stmaryrd}
\usepackage{float}
\usepackage[numbers]{natbib}
\usepackage{diagbox}

\renewcommand{\citet}[1]{\cite{#1}}
\usepackage{tikz}
\usepackage{subfig}
\usepackage{pgfplots}
\usepackage{pgfplotstable}
\pgfplotsset{compat=1.5}
\usepackage[plainpages=false,pdfpagelabels,bookmarks = true, bookmarksopen = false,breaklinks, colorlinks=true, citecolor = black, urlcolor = black, filecolor = black, linkcolor = black, pdftitle = {Computing stress intensity factors for curvilinear fractures}, pdfauthor = {M.M. Chiaramonte and Y. Shen and and L.M. Keer and A.J. Lew}, pdfsubject = {IJNME},  pdfproducer = {Stanford University}]{hyperref}

\graphicspath{{./Figures/}}
\theoremstyle{plain}
\newtheorem*{prop*}{Proposition}
\usepackage{soul,color,ifthen}
\usepackage{todonotes}

\newcounter{mycomment} 
\newcommand{\mycomment}[2][]{
\refstepcounter{mycomment}
	{
	\todo[color={blue!100!green!33},size=\small,inline]{%
		\textbf{Comment [\uppercase{#1}\themycomment]:}~#2
	}
	}
}

\tikzset{
    state/.style={
           rectangle,
           rounded corners,
           draw=black, very thick,
           minimum height=2em,
           inner sep=2pt,
           text centered,
           },
}

\newcommand{\eq}[1]{\begin{equation} #1 \end{equation}}
\newcommand{\eqnonumber}[1]{\begin{equation*} #1 \end{equation*}}

\newcommand{\Cbb}{\mathbb{C}}

\newcommand{\Ibb}{\mathbb{I}}
\newcommand{\Rbb}{\mathbb{R}}

\newcommand{\Csc}{\mathscr{C}}
\newcommand{\Bsc}{\mathscr{B}}

\newcommand{\1}{{\bf 1}}
\newcommand{\0}{{\bf 0}}

\newcommand{\bb}{{\boldsymbol b}}

\newcommand{\ee}{{\boldsymbol e}}
\newcommand{\gb}{{\boldsymbol g}}
\newcommand{\nb}{{\boldsymbol n}}
\newcommand{\tb}{{\boldsymbol t}}
\newcommand{\ub}{{\boldsymbol u}}

\newcommand{\wb}{{\boldsymbol w}}
\newcommand{\x}{\boldsymbol x}

\newcommand{\Bc}{\mathcal{B}}

\newcommand{\Gc}{\mathcal{G}}

\newcommand{\Ic}{\mathcal{I}}

\newcommand{\Oc}{\mathcal{O}}

\newcommand{\Mc}{\mathcal{M}}
\newcommand{\Vc}{\mathcal{V}}
\newcommand{\Wc}{\mathcal{W}}

\newcommand{\Sc}{\mathcal{S}}

\newcommand{\betab}{\boldsymbol \beta}
\newcommand{\cs}{\boldsymbol \sigma}

\newcommand{\gammab}{\boldsymbol \gamma}
\newcommand{\Phib}{\boldsymbol \Phi}

\newcommand{\taub}{\boldsymbol \tau}

\newcommand{\lambdab}{\boldsymbol \lambda}

\newcommand{\Sigmab}{\boldsymbol \Sigma}

\DeclareMathOperator{\dive}{div}
\DeclareMathOperator{\grad}{grad}

\DeclareMathOperator{\vectorspan}{span}

\DeclareMathOperator{\area}{area}
\DeclareMathOperator{\dist}{dist}

\newcommand{\msp}[1]{\hbox{\hskip #1in}}
\newcommand{\mmsp}{\hbox{\hskip 0.2in}}

\newcommand{\au}{^\text{aux}}
\newcommand{\modes}{{I,II}}

\newcommand{\xt}{\x_{t}}

\newenvironment{rmk}[1][]{\par\vskip0.5\baselineskip\noindent\textit{Remark} (#1). \it}{\hfill\vskip0.5\baselineskip}
\newenvironment{justification}[1][]{\par\vskip0.5\baselineskip\noindent\textit{Justification} (#1).  }{\qed\vskip0.5\baselineskip  }

\usetikzlibrary{matrix,arrows}

\renewcommand{\todo}[1]{{\color{red}\vspace{5 mm}\par \noindent
  \marginpar{\textsc{ToDo}} \framebox{\begin{minipage}[c]{0.95
        \textwidth} \tt #1 \end{minipage}}\vspace{5 mm}\par}}
\newcommand{\comment}[1]{{\color{blue}\vspace{5 mm}\par \noindent
  \marginpar{\textsc{Comment}} \framebox{\begin{minipage}[c]{0.95
        \textwidth} \tt #1 \end{minipage}}\vspace{5 mm}\par}}
\renewcommand{\comment}[1]{}
\renewcommand{\grad}{\nabla}
\renewcommand{\dive}{\nabla\cdot}

\pdfoutput=1

\renewcommand{\grad}{\nabla}
\renewcommand{\dive}{\nabla\cdot}

\newcommand{\ic}[1]{\iffalse \fi}

\renewcommand{\todo}[1]{}
\newcommand{\deleted}[2][]{}
\renewcommand{\comment}[1]{}
\newcommand{\added}[2][]{#2}
\newcommand{\replaced}[3][]{#2}
\newcommand{\sorted}[1]{}
\begin{document}
\title{
Computing stress intensity factors for curvilinear cracks
}

\author{Maurizio M. Chiaramonte\affil{1}, Yongxing Shen\affil{2}\corrauth, Leon M. Keer\affil{3}, \\ and Adrian J. Lew\affil{1}\corrauth}  
\runningheads{M. M. Chiaramonte,  Y. Shen, L. M. Keer, and A. J. Lew}{Computing stress intensity factors for curvilinear cracks} 
\address{\centering
\affilnum{1}Department of Mechanical Engineering, Stanford University, 496 Lomita Mall, Stanford CA 94305, USA\\
\href{mailto:mchiaram@stanford.edu,lewa@stanford.edu}{\normalfont\ttfamily\{mchiaram,lewa\}@stanford.edu } 
\vskip 0.5\baselineskip
\affilnum{2}University of Michigan-Shanghai Jiao Tong University Joint Institute, Shanghai Jiao Tong University, 800 Dongchuan Road, Shanghai, 200240, China\\ 
\href{mailto:yongxing.shen@sjtu.edu.cn}{\normalfont\ttfamily yongxing.shen@sjtu.edu.cn}\\
\vskip 0.55\baselineskip
\affilnum{3}Department of Mechanical Engineering, Northwestern University, 2145 Sheridan Rd., Evanston IL 60208, USA\\
\href{mailto:l-keer@northwestern.edu}{\normalfont\ttfamily l-keer@northwestern.edu } \\
}
\corraddr{yongxing.shen@sjtu.edu.cn, lewa@stanford.edu}

\footnotetext[0]{Contract/grant sponsor: Y. Shen (while at Universitat Polit\`ecnica de Catalunya): Marie Curie Career Integration Grant (European Commission); contract/grant
number: FP7-PEOPLE-2011-CIG-CompHydraulFrac. Y. Shen (current affiliation): National Science Foundation of China; contract/grant number: 11402146.
A. Lew,  National
  Science Foundation; contract/grant number CMMI-1301396. M.M. Chiaramonte, Office of Technology Licensing, Stanford Graduate Fellowship.}

\def\feinexamples{1}
\newcommand{\testfunctionname}{material variation }
\begin{abstract}
  The use of the interaction integral to compute stress intensity
  factors around a crack tip requires selecting an auxiliary field and
  a material variation field.  We formulate a family of these fields accounting for the curvilinear
  nature of cracks that, in conjunction with a discrete
  formulation of the interaction integral, yield optimally convergent
  stress intensity factors. We formulate three
  pairs of auxiliary and material variation
  fields chosen to yield a simple
   expression of the interaction integral for different
  classes of problems. The formulation accounts for  crack
  face tractions and body forces. Distinct features of the fields are their
  ease of construction and implementation. The resulting
  stress intensity factors are observed converging at a rate that doubles the
  one of the stress field. We provide a sketch of the theoretical justification for the
  observed convergence rates, and discuss issues such as quadratures
  and domain approximations needed to attain such  convergent behavior.
    Through two representative examples, a
  circular arc crack and a loaded power function crack, we illustrate the convergence rates of
  the computed stress intensity factors. The numerical results
  also show the independence of the method on
  the size of the domain
  of integration.

\end{abstract}

\maketitle

\section{Introduction}
\label{sxn:introduction}

The stress field near the tip of a loaded crack is singular under the
assumption of linear elastic fracture mechanics.  The coefficients of the
asymptotic stress field, known as the \emph{stress intensity factors}, play a key role in
characterizing the magnitude of the load applied to the crack and predicting
its propagation.

Given the stress singularity and the poor accuracy in pointwise
evaluation of the stress field, it is often impossible to extract the
stress intensity factors directly from numerical solutions, unless a
higher-order method to compute the elastic field is adopted, such as those
proposed in Liu \emph{et al.}~\cite{liu2004}, Shen and Lew \cite{shen2009}, and Chiaramonte \emph{et al.}~\cite{chiaramonte2014b}. 


As a result, path and domain integral methods to extract the stress intensity
factors have been created precisely to circumvent this limitation. A method
of this kind typically formulates the expression of the stress intensity
factors as functionals of the solution, thus enjoying a higher order of
convergence than the \replaced[id=mc]{one of}{convergence order for} the elastic field itself.
Predominant methods of this kind have been constructed based on the $J$-integral \cite{rice1968} and the interaction integral \cite{stern1975}. 

In the context of linear elastic fracture mechanics, the $J$-integral is identified with the system's  elastic energy release rate\replaced[id=mc]{;}{,} the elastic energy that would be released per unit length of crack extension in the tangential direction. 
This integral and related ones for the computation of fracture-mechanics-related
quantities have been elaborated by Eshelby \citet{eshelby1951},
Rice \cite{rice1968}, Freund \cite{freund1977}, and many others. 
Shih \emph{et al.}~\cite{shih1986} derived the expression for the energy release rate of a thermally stressed body in the presence of crack face traction and body force.
A general treatment of such conservation integrals, including those expressed in a form of domain integrals, can be found in Moran and Shih \citet{moran1987a, moran1987b}. The domain form is better suited and more accurate for numerical computation.
Nevertheless, in the case of mixed-mode loading, $J$ is a quadratic function of all three stress intensity factors \cite{irwin57}; therefore, additional integrals are needed to determine the three quantities individually, e.g., as done by Chang and Wu \cite{chang2007} for non-planar curved cracks.

In contrast, the interaction integral, or the interaction energy integral, is able to yield the three stress intensity factors separately.
This method is based on the $J$-integral by superposing the 
elastic field of the loaded body and an {\it auxiliary}  field with known stress
intensity factors. The auxiliary field does not need to satisfy the
elasticity equations but must resemble the asymptotic solution of a
cracked elastic body corresponding to one of the three loading modes
(e.g., plane-strain mode I or mode II, or anti-plane mode
III). Therefore the auxiliary field, for straight cracks, is normally
chosen to be the asymptotic solution, as found in \cite{williams1952,
  williams1957}. Doing so readily yields the stress intensity factor
of the actual field for the chosen mode. Along with the auxiliary
field, the interaction integral requires the construction of a vector
field, named the {\it \testfunctionname}field, which indicates the
velocity (variation) of points in the reference configuration as it is
deformed into  a domain with a longer crack. Under mild conditions, the value of the
interaction integral does not depend on this choice, but a good choice
of \testfunctionname can simplify computations. For example, the
interaction integral is computed by integrating over the support of
the \testfunctionname field, so it is convenient to choose \testfunctionname
fields with small and compact support. While developing auxiliary and \testfunctionname fields for straight cracks (planar cracks in three dimension) is an amenable task, doing so for curvilinear cracks (non-planar cracks in three dimension) poses several challenges. In the following paragraphs we provide a short review of
the effort  related to the computation of stress intensity factors
with the use of the interaction integral.


Earlier methods to compute the stress intensity factors with the
interaction integral involved path integrals, such as those in Stern \emph{et al.}~\cite{stern1975} and Yau \emph{et al.}~\cite{yau1980}.
The method was later generalized to a straight-front crack in
three-dimensions \replaced[id=mc]{by}{in} Nakamura and Parks \cite{Nakamura1989} and Nakamura \cite{Nakamura1991}. A curved crack front introduces additional terms, since the popular plane-strain modes I and II auxiliary fields no longer satisfy the compatibility and the equilibrium conditions. These additional terms were accounted for in Nahta and Moran \cite{Nahta1993} for the axisymmetric case, and in Gosz \emph{et al.}~\citet{gosz1998} for general planar cracks in three dimensions. 
Kim \emph{et al.}~\cite{Kim2001} adopted auxiliary fields corresponding to penny-shaped cracks, as well as conventional plane-strain and anti-plane ones.
An alternative approach was given by Daimon and Okada \cite{Daimon2014} who adopted a compatible auxiliary field and accounted for its lack of equilibrium by superposing a numerically computed displacement field with the finite element method.
 A study of the effect of omitting some terms accounting for the curved front is given by Walters \emph{et al.}~\cite{walters2005}.  

More recently the method of
\citet{gosz1998} was adapted for non-planar cracks in Gosz and Moran
\cite{gosz2002}, the latter of which is arguably a milestone in the development of domain integral methods to extract stress intensity factors from curvilinear cracks in 2D and non-planar cracks in 3D. In
\cite{gosz2002}, the method constructs the auxiliary fields through the use of
curvilinear coordinates and their corresponding  covariant basis to account for the
crack curvature. This procedure \replaced[id=mc]{constructs}{allows for a simple construction of} 
the
auxiliary fields by juxtaposing the components of the stress fields of 
\citet{williams1952} with the described basis. 
In \cite{sukumar2008}, Sukumar \emph{et al.}~implemented the method of \cite{gosz2002} in combination with an extended finite element method in a three-dimensional setting \cite{Moes2002} and a fast-marching method. The main drawback of the curvilinear coordinates of \cite{gosz2002} is the need to  perform  boundary integrals over  both the real crack surfaces and a pair of fictitious crack surfaces. 
In \cite{gonzalez2013b}, Gonz\'alez-Albuixech \emph{et al.} proposed
another curvilinear coordinate system  that can eliminate the
integration on the fictitious crack surfaces, and in this way facilitate the computation.


In \cite{gonzalez2013, gonzalez2013b}, Gonz\'alez-Albuixech \emph{et al.}~studied the properties of the aforementioned
methods in two- and three-dimensions, respectively, but not with all terms arising from the derivation of the interaction integral. Such omission of terms may have contributed to the observed slow convergence and occasional divergence.
This observation confirms the statement of \cite{gosz2002} that all terms arising from the lack of compatibility and equilibrium of the auxiliary field have to be taken into account.

The domain version of the interaction integral has also been generalized to functionally graded materials \cite{moghaddam2011, ghajar2011}.

\ifnum1=0
\mycomment{MC: Novelty 
\begin{itemize} 
\item Cast all definitions of \testfunctionname fields and auxiliary fields \item
Simple construction of auxiliary fields easily implementable \item Account for
body forces and crack face tractions \item Does not rely on level set functions
which are costly to evaluate \item Provide method that requires boundary
integration but does not require evaluation of gradient of gradient of
auxiliary fields (divergence free) \item Provide method that does not requires
evaluation of boundary integration but requires computation of  gradient of
gradient of aux \item Rapidly convergent computation \item Provide convergence
curve 
\end{itemize}
}
\fi

In this paper we present a suite of auxiliary fields and material
variations fields. By pairing two constructs of \testfunctionname
fields and two constructs of auxiliary fields, we create two kinds of
interaction integral suitable for curvilinear cracks and for
situations in which body forces and crack face tractions are
present. One kind of interaction integral is suited for applications
where crack faces are loaded (e.g. hydraulic fracturing), and the
other one is best suited for applications where body forces are
non-zero (e.g. thermally loaded materials). Moreover, no fictitious
crack face is needed, a major simplification to the predominant method
in the literature.
 
%
%

One of the two choices of the \testfunctionname fields has a constant
direction pointing to the direction of the crack growth, a
straightforward choice adopted by most authors, and the other has a
direction that is tangential to the crack near the crack tip, similar
to that proposed in \cite{Moes2002}. For a curved crack, this second
choice necessarily coincides with the first one only at the crack
tip. A key advantage of the \testfunctionname fields we introduce here \replaced[id=mc]{is their ease of construction which is reflected in their straightforward implementation in computer codes}{
 is that they are simple to construct (or implement) in the
computer}. Moreover, \deleted[id=mc]{and} in contrast to many of the existing
constructions, the magnitude of the \testfunctionname fields  is 
mesh independent. This mesh independence contributes to the observed
optimal rate of convergence.

The two auxiliary fields are constructed from the well-known
asymptotic solutions of a straight crack. Both fields respect the
discontinuity introduced by the crack, thus avoiding the evaluation of
integrals over fictitious crack faces as the method in \cite{gosz2002}
does. 
One of the auxiliary fields is obtained by ``extending'' the
asymptotic solutions past the range $[-\pi,\pi]$, and hence satisfies
equilibrium, compatibility, and the constitutive relation. The
resulting interaction integral expression then yields a term on the
crack faces, even in the absence of crack face traction. The second
auxiliary field is  an  incompatible strain field. It is
obtained by first mapping a straight crack to the curved crack near
the crack tip, and using this map to push forward the strain field of
the straight-crack asymptotic solution. \replaced[id=mc]{Then, by suitably rotating the
strain tensor at each point,}{By then suitable rotating the
strain tensor at each point,} we obtain an auxiliary strain field that
is traction-free at the curved crack faces.  This is useful for
problems in which  crack faces are traction-free\deleted[id=mc]{:}\added[id=mc]{. In fact,} if this
auxiliary field is used in combination with the tangential
\testfunctionname field, the crack face integral vanishes, resulting
in a significantly simplified expression for the interaction integral.


We showcase the convergence \deleted[id=mc]{behavior} of the stress intensity factors
obtained with the proposed fields for a set of representative examples
computed with two different finite element methods. In all cases, the
stress intensity factors converge with a rate that doubles the rate of
convergence of the strains. We also numerically demonstrate the
independence of the computed stress intensity factors from the chosen
support for the \testfunctionname field.  Although the numerical
examples adopt finite element methods to obtain an approximate
solution to the elasticity problem, the numerical implementation of
the interaction integral with the new fields is general, and can be
used in conjunction with any numerical method for the solution of the
governing equations\added[id=mc]{ (e.g. finite difference, finite volume,
boundary integral equations, isogeometric analysis, and meshless
methods)}.

The paper is organized as follows. We first state the problem that we
seek to solve in \S \ref{sxn:problem_statement}. We then proceed in \S
\ref{sxn:interint} to present the interaction integral with the
description of the new \testfunctionname and auxiliary fields. In the
same section we justify that the proposed forms of the interaction
integral are well-defined.  A numerical approximation of the
interaction integral is presented in \S \ref{sxn:computation} with
remarks on its expected convergence. \replaced[id=mc]{The}{Te} last \replaced[id=mc]{part}{section} of
\S\ref{sxn:computation} provides a step-by-step recapitulation of the
method\deleted[id=mc]{,} suited for the reader interested in a concise presentation.
In \S \ref{sxn:numerical_examples} we verify the computation of the
stress intensity factors against analytical solutions for two
problems\replaced[id=mc]{:}{,} a circular arc crack and a power function crack\deleted[id=mc]{,for which
we provide analytical solutions}. \ic{ We then conclude the manuscript
  in \S \ref{sxn:concluding_remarks}.}  Throughout the paper we
included sections titled ``{\it Justification}\deleted[id=mc]{,}'' which contain
sketches of proofs for some of the assertions we make, and they are not
essential for the description of the methods in this paper.

\section{Problem Statement}
\label{sxn:problem_statement}

We present next the problem statement\deleted[id=mc]{,} which consists \replaced[id=mc]{of the evaluation of the stress intensity factors following the solution of the elasticity fields for a cracked solid.}{ of the elasticity problem
of a cracked solid and  the computation of the stress
intensity factors.
}

\subsection{Elasticity Problem} \label{subsxn:elasticity_problem}

We consider a body $\Bc \subset \Rbb^2$ undergoing a deformation
defined by the displacement field $\ub$. We assume $\Bc$ to be an open
and connected domain with a (piecewise) smooth boundary $\partial \Bc$.
We represent the crack with a twice differentiable, simple and rectifiable curve $\Csc \subset \Bc $ and denote its crack faces with $\Csc_\pm$. The cracked domain is given by $\Bc_\Csc = \Bc\backslash\Csc$. The boundary of $\Bc_\Csc$ is the union of the crack faces and the boundary of $\Bc$, \added[id=mc]{namely} $\partial \Bc_\Csc = \partial\Bc \cup \Csc_\pm$. Let $\partial
\Bc_\Csc$  be decomposed into $ \partial_{\tau} \Bc_\Csc$ and
$\partial_d \Bc_\Csc$ such that $\partial_{\tau}
\Bc_\Csc\supseteq\Csc_\pm$, $\partial_{\tau}  \Bc_\Csc \cup \partial_d
\Bc_\Csc =\partial  \Bc_\Csc$, and $\partial_{\tau}  \Bc_\Csc
\cap \partial_d  \Bc_\Csc = \emptyset$. Tractions $\overline \tb$ and
displacements $\overline \ub$ are prescribed over $\partial_{\tau}
\Bc_\Csc$ and $\partial_d \Bc_\Csc $, respectively, while a body force
field $\bb$ is applied over $ \Bc_\Csc$. Let $\xt $ denote any one of
the two crack tips. We denote by $\nb$ the unit external normal to
${\cal B}$, as well as the unit external normal to each one of the two
faces of the crack. 
Figure
\ref{fig:problem} shows the schematic of the problem configuration.
\begin{figure}[H] \centering
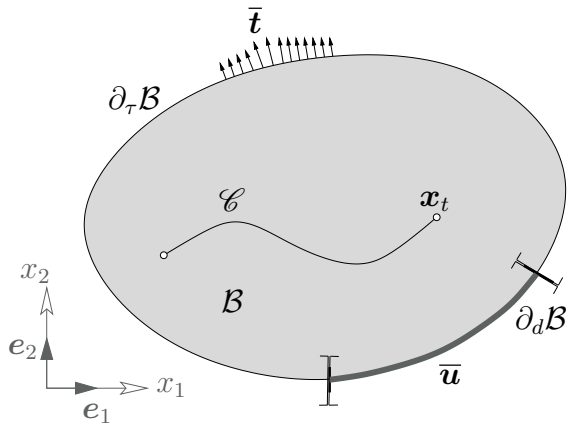 \caption{The configuration of the problem}
\label{fig:problem} 
\end{figure}
We confine ourselves to planar linear elasticity in the context of infinitesimal
deformations. The elasticity problem statement reads: Given $\bb$ , $\overline \ub$, and $\overline \tb $  find $\ub:\Bc_\Csc \to
\Rbb^2$ such that 
\begin{subequations}
\begin{alignat}{2}
\dive \cs(\grad \ub) + \bb &= 0,\quad&&  \text{ in } \Bc_\Csc,\label{eq:equilibrium}\\ 
\ub &= \overline \ub,\quad&&\text{ on } \partial_d\Bc_\Csc,\\
 \cs(\grad \ub)\nb &= \overline \tb,\quad&&  \text{ on } \partial_\tau \Bc_\Csc,
\label{eq:traction}
\end{alignat}
\label{eq:elasticity}
\end{subequations}
where   $\cs$ is the stress tensor. This is  given by
\begin{equation}
\label{eq:constitutiverelation}
\cs(\grad\ub) = \Cbb:\grad\ub,
\end{equation} 
and 
\begin{equation*}
\Cbb = \hat\lambda  \1 \otimes \1 + 2 \mu \Ibb, \mmsp \hat \lambda = 
\begin{cases} \lambda, &
\text{ for plane strain},\\ \dfrac{2 \lambda  \mu }{\lambda +2 \mu }, &
\text{ for plane stress}.
\end{cases}
\end{equation*}
The constants $\lambda$ and $\mu$ are Lam\'e's first and second parameters,
respectively,  $\1$ is the identity second-order tensor, 
and $\Ibb$ is the fourth-order symmetric
identity operator given by
 \[ \Ibb =
 \frac12(\delta_{ik} \delta_{jl} + \delta_{il}\delta_{jk} ) \ee_i \otimes \ee_j
 \otimes \ee_k \otimes \ee_l,
 \] 
where $\{\ee_1,\ee_2\}$ is a  Cartesian basis, and an index repeated twice indicates sum from 1 to 2 in such
index.

\subsection{Crack Tip Coordinates and Stress Intensity Factors}
\label{subsxn:fracture_mechanics_problem}

To aid the definition of the stress intensity factors we 
first introduce a system of coordinates and a family of vector bases.  
\begin{figure}[htbp]
\centering
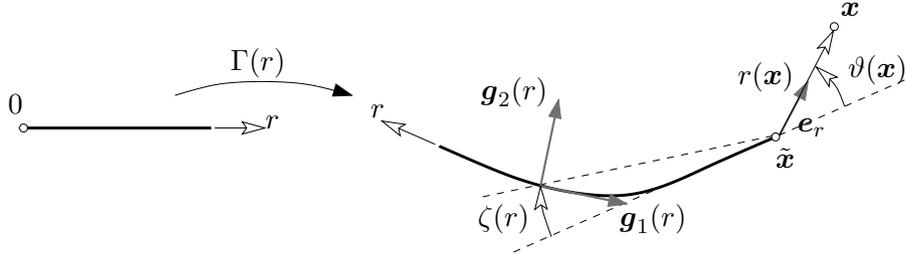
\caption{Description of local basis and coordinates}
\label{fig:loccrds}
\end{figure}

Let $(r,\vartheta)$ be crack tip polar coordinates  as shown in
Fig. \ref{fig:loccrds}, and let $B_\rho( \hat \x) =
\left\{ \x\in\mathbb{R}^2 \middle| | \x - \hat \x | < \rho \right\}$
be the open ball of radius $\rho>0$ centered at $\hat \x$.  The radial coordinate
is defined as $r(\x):= | \x - \xt |$. Let $\Gamma:
 r\mapsto\{\x\in\Csc|\;|\x-\xt|=r\}$ be a description of part of
 the crack parametrized by the distance to the
 crack tip. We set the domain of $\Gamma$ to be $[0,\rho]$ with $\rho
 >0$ such that $B_\rho(\xt)\subset\Bc$ and that $\Gamma'\neq0$ over
 its domain of definition. As a consequence, $\Gamma$ is bijective for $r\in [0,\rho]$. We re-iterate that the crack is assumed to
 be a twice continuously differentiable curve such that $\Gamma \in
 C\added[id=mc]{^2}(\Rbb_0^+;\Rbb^2)$, where $\Rbb_0^+=\{0\}\cup\Rbb^+$. By convention,
 the possible values of the $(r,\vartheta)$ coordinates for points in
 $B_\rho(\xt)$ that we will use belong to
$$D_\rho =
\{(r,\vartheta)\in \mathbb R_0^+\times \mathbb R\mid -\pi - \zeta(r) \le \vartheta\le\pi
-\zeta(r)  \},$$
where $\zeta(r) $ is the angle between the vector $\Gamma(r) - \xt$ and $\Gamma'(0)$. In other words, $\zeta(r)$ is the angle subdued by (a) the
tangent at the crack tip and (b) the secant line passing through the crack tip
and $\Gamma(r)$. Figure
\ref{fig:thetaext} shows the values of the coordinate $\vartheta$ for the particular case of a
circular arc crack. 
 Lastly let $\gb_i(r), r\in[0,\rho]$ be the
right-handed orthonormal bases induced from the mapping $\Gamma$ such that
$\gb_1(r) = -\Gamma'(r) / |\Gamma'(r) |$, see Fig. \ref{fig:loccrds}. 

\begin{figure}[htbp] 
\centering 
\begin{tikzpicture}[x=2.5in,y=2.5in]
 \node at (0,0) { \includegraphics[trim=1.1in 1.in 0.75in
0.5in,clip,height=0.3\textheight]{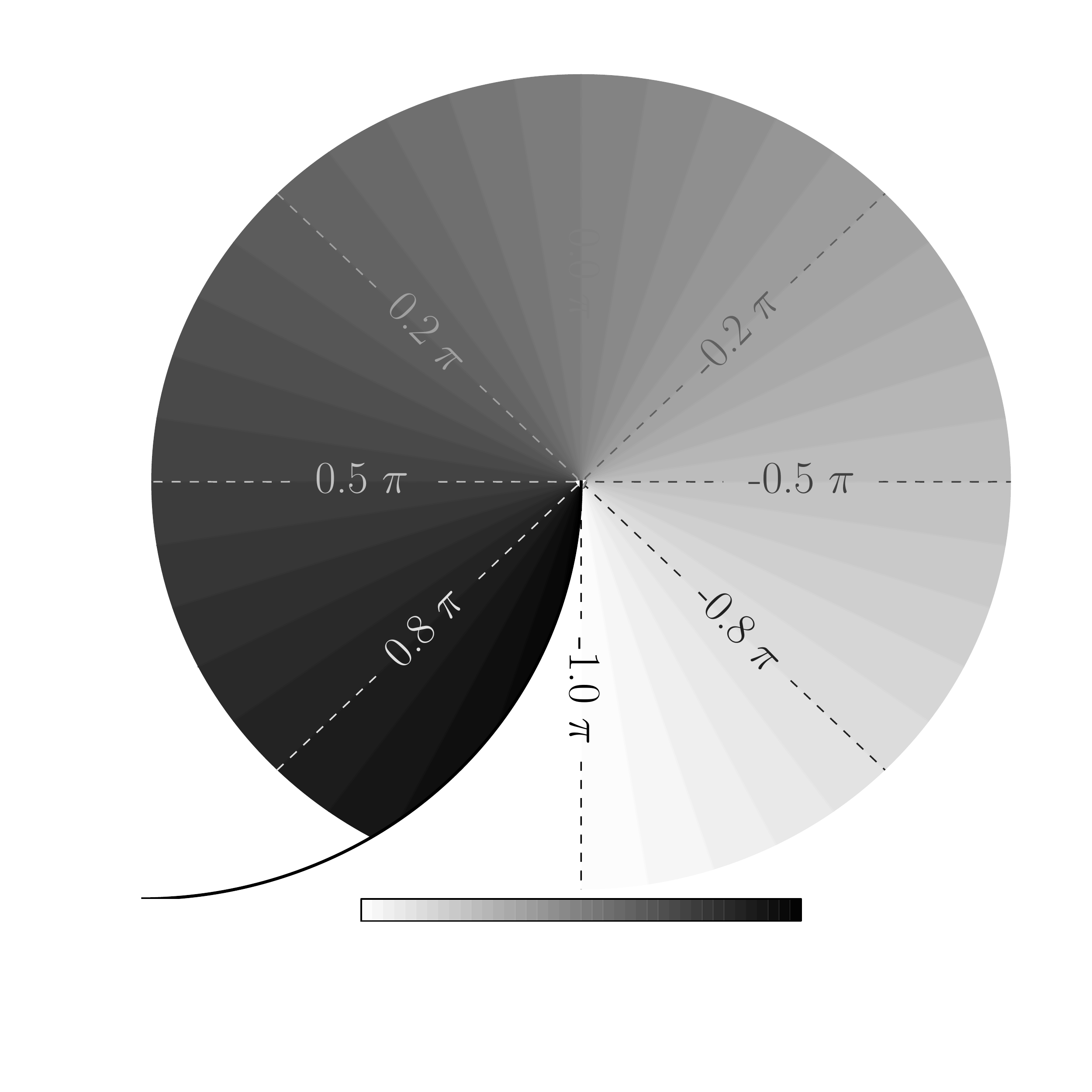} }; 
\node at (-0.35,-0.525) [right] { $-\pi + \max{ \zeta } \msp{0.25} \vartheta   \msp{0.25} \pi 
+ \max{\zeta}$};
\node at (-0.65,-0.425) [right] { Crack}; \end{tikzpicture} \captionof{figure}{
The $\vartheta$-coordinate with the branch cut along the crack.}
\label{fig:thetaext} 
\end{figure}
\sorted{
\todo{AL: it seems that the locations of the legends are off-center}
\todo{MC: done} }
Throughout this manuscript,  we assume that there exist unique real numbers $K_I, K_{II}$ such that 
\begin{equation}
\label{eq:regularity} 
\ub = K_I \ub^I + K_{II} \ub^{II} +
\ub_S, \quad \ub_S\in H^2\left(\Bc_\Csc; \Rbb^2\right),
\end{equation}
where $K_I$ and $K_{II}$ are \added[id=mc]{the} modes $I$ and $II$ stress intensity
factors, respectively, and $\ub^I, \ub^{II}\in H^1(\Bc_\Csc;
\Rbb^2)\setminus H^2(\Bc_\Csc; \Rbb^2)$ are the asymptotic
displacement solutions of these modes.  


The problem of evaluating the stress intensity factors is: Given \deleted[id=mc]{$\ub$} a solution \added[id=mc]{$\ub$} of Problem
\eqref{eq:elasticity}, compute $K_I$ and $K_{II}$ as defined in
(\ref{eq:regularity}). 

\begin{rmk}[Explicit evaluation of the behavior at the crack tip]
Whenever
$\betab=\grad\ub\in C^0(\Bc_\Csc; \Rbb^{2\times2})$, the stress
intensity factors can be alternatively defined as:
\begin{equation}
\label{eq:stress_intensity_factors}
\begin{aligned} 
K_I = \lim_{r\to 0} \sqrt{2\pi r} \cs(\betab )|_{\vartheta = 0} :
 \gb_2\otimes\gb_2,  \\
K_{II} = \lim_{r\to 0 } \sqrt{2\pi r} \cs(\betab )|_{\vartheta = 0} :
\gb_1\otimes\gb_2.  
\end{aligned} 
\end{equation}

The calculation of the stress intensity factors by
evaluating the limits in \eqref{eq:stress_intensity_factors} with a numerical
solution often leads to poor results. 
In fact,
few methods are capable of accurately resolving the singularity in the stress
field.
Therefore the predominant methods
to compute the stress intensity factors are based on the approximation
of the \emph{interaction integral}, a formulation that avoids pointwise evaluation of the stress field in the region with
the singularity.  We proceed to introduce \deleted[id=mc]{alternative formulas for} the
interaction integral in \S\ref{sxn:interint}.
\end{rmk}

\section{Interaction Integral} 
\label{sxn:interint}

\sorted{
\todo{YS: I think (a stress field) ``to yield traction-free crack
  faces'' should be changed to ``to be traction-free on the crack
  faces.''

AL: I like this, but cannot find where we do this.

}
}

In the sequel we first define the interaction integral functional
between any two admissible fields alongside a concise justification of
this definition (see \S\ref{subsxn:interaction_integral} ).  In
\S\ref{subsxn:prob-dep-int} we specialize it to the case in which one
of the fields is the solution to the elasticity problem of interest. This specialization
 results in a formulation that possesses the following properties:
(a) it does not involve second derivatives of the discrete
approximation to the exact displacement field; (b) it is a
problem-dependent functional\deleted[id=mc]{,} because it uses the prescribed tractions
and body forces; and (c) it can be further simplified, depending
on the problem, by carefully choosing the so-called material variation
and auxiliary fields, to be described later. We perform the
simplification in
\S\ref{subsxn:the_fields} and \S\ref{subsxn:interaction_integrals}, and
obtain three different formulas. In \S\ref{subsxn:concl_rem} we provide guidelines on which
specific formula for the interaction integral to choose for a given
application.


\subsection{Definition of the Interaction Integral Functional}
\label{subsxn:interaction_integral} 
The interaction integral involves two elasticity fields, one with
known stress intensity factors, and the other whose stress intensity
factors we are interested in evaluating.


Let us introduce the \emph{interaction energy momentum tensor} $\overline\Sigmab:\Rbb^{2\times 2}\times
\Rbb^{2\times 2} \to \Rbb^{2\times 2}$ defined as
\begin{equation*}
\overline\Sigmab\left(\betab^a,\betab^b\right) = \overline
w\left(\betab^a,\betab^b\right) \1- \betab^a \,^\top\cs\left(\betab^b\right) -
\betab^b\,^\top\cs(\betab^a),\\ 
\end{equation*}
where $\overline w : \Rbb^{2\times 2}\times \Rbb^{2\times 2} \to \Rbb $ is given as
\begin{equation}
\overline w\left(\betab^a,\betab^b\right) = \frac{1}{2} \left[ \cs(\betab^a) 
 :\betab^b +
\cs \left(\betab^b\right) :\betab^a \right].\label{eq:reciprocal_energy} 
\end{equation}
Assuming the same constitutive relation (\ref{eq:constitutiverelation}) for both fields, 
\eqref{eq:reciprocal_energy} simplifies to \newcommand{\testset}{\Mc}
\begin{equation*}
\overline
w\left(\betab^a,\betab^b\right) = \cs(\betab^a ): \betab^b =
\cs\left(\betab^b\right): \betab^a.  
\end{equation*} Additionally, let  the
set of \emph{material variations} be defined as 
\eq{
  \testset = \left\{\delta \gammab \in{C^1\left(\Bc_\Csc; \Rbb^2
      \right)} \, \middle| \delta \gammab = \0 \text{ in } \Bc_\Csc
    \setminus B_{\rho} (\xt ), \delta \gammab(\xt) = \gb_1(0)
		\right\}\replaced[id=mc]{.}{,}  \label{eq:setv} }\deleted[id=mc]{where $\nb(\xt)$ is normal to any of the two crack faces evaluated at the
crack tip.}Finally, let 
\[
\Bsc^b=\vectorspan\left\{\nabla
\ub^I,\nabla
\ub^{II}\right\} \oplus H^1\left(\Bc_\Csc; \Rbb^{2\times2}\right),
\]
and for any tensor field $\betab\in \Bsc^b$ define $K_I\replaced[id=mc]{[\betab]}{(\betab)}$ and
$K_{II}\replaced[id=mc]{[\betab]}{(\betab)}$ such that
\begin{equation}
  \label{eq:1}
  \betab =   K_I\replaced[id=mc]{[\betab]}{(\betab)} \nabla \ub^I+K_{II}\replaced[id=mc]{[\betab]}{(\betab)}\nabla \ub^{II} + \betab_S
\end{equation}
with $\betab_S\in H^1\left(\Bc_\Csc ; \Rbb^{2\times2}\right)$. In
particular, if $\betab=\nabla \ub$, for $\ub$ \added[id=mc]{being} a solution of Problem
\eqref{eq:elasticity}, then $K_I\replaced[id=mc]{[\betab]}{(\betab)}$ and $K_{II}\replaced[id=mc]{[\betab]}{(\betab)}$ are the
stress intensity factors of $\ub$. However, $K_I\replaced[id=mc]{[\betab]}{(\betab)}$ and
$K_{II}\replaced[id=mc]{[\betab]}{(\betab)}$ are still defined for any $\betab\in \Bsc^b$ that  is not the gradient of
a displacement field. For convenience, regardless of whether $\betab$
is or is not the gradient of a displacement field, we will refer to  $K_I\replaced[id=mc]{[\betab]}{(\betab)}$ and
$K_{II}\replaced[id=mc]{[\betab]}{(\betab)}$ as the stress intensity factors of $\betab$. 
\sorted{
\todo{YS: I have modified the definition of $\testset$ since we
  have already defined the tangent direction of the crack at the
  tip. Moreover, the previous definition did not distinguish the two
  possible senses of $\delta\gammab(\xt)$, but in fact we only allow
  unit advancement, not unit healing.

  AL: Very nice.
}
}

We define the {\it interaction integral functional} $\hat\Ic: \Bsc^b \times
\Bsc^b \times \testset \to \Rbb $
as 
\begin{equation} \label{eq:interinta}
\begin{aligned} \hat\Ic\left[ \betab^a,\betab^b,\delta \gammab \right]  =&
\int_{\Csc^\pm_\rho  }  \delta \gammab \cdot
\overline\Sigmab\left(\betab^a,\betab^b\right) \nb \;dS \\
-&  \int_{ B_\rho(\xt)\setminus \Csc } \left[ 
\overline\Sigmab\left(\betab^a,\betab^b\right):\grad\delta \gammab +
\dive\overline\Sigmab\left(\betab^a,\betab^b\right)\cdot \delta \gammab
\right]\; dV, \end{aligned} 
\end{equation}
where, for convenience, we let $\Csc^\pm_\rho $ denote $\Csc_\pm
\cap B_\rho(\xt)$. The value $\hat\Ic\left[
  \betab^a,\betab^b,\delta \gammab \right]$ is the
{\it interaction integral} between $\betab^a$ and $\betab^b$. 
The relation between the interaction integral and the stress intensity
factors of  $\betab^a,\betab^b \in
\Bsc^b$ is
\eq{ \hat\Ic\left[ \betab^a,\betab^b,
\delta \gammab\right] = \eta \left(\replaced[id=mc]{K_I[\betab^a] K_I\left[\betab^b\right] + K_
{II}[\betab^a] K_{II}\left[\betab^b\right]}{K_I(\betab^a) K_I\left(\betab^b\right) + K_
{II}(\betab^a) K_{II}\left(\betab^b\right) } \right)\label{eq:I_in_terms_of_Ks} }
for any $\delta \gammab \in \testset$,   where $\eta$ is a material constant defined as \begin{equation*} 
\eta = \left\{\begin{array}{ l l } \dfrac{\lambda +2 \mu }{2 \mu  (\lambda +\mu)}, &
\text{ for plane strain},\\ \\ \dfrac{2 (\lambda +\mu )}{\mu  (3 \lambda +2 \mu
)}, & \text{ for plane stress}.  
\end{array} \right.
\end{equation*}

It follows from (\ref{eq:I_in_terms_of_Ks}) that, if we are interested in finding $K_{I}\replaced[id=mc]{[\betab]}{(\betab)}$ (or
$K_ {II}\replaced[id=mc]{[\betab]}{(\betab)} $), we must generate an {\it auxiliary tensor field} $\betab\au_{I}$  (or
$\betab\au_{II} $) $\in \Bsc^b$ satisfying $K_I\replaced[id=mc]{[ \betab\au_{I}]}{( \betab\au_{I})} = 1$, $K_
{II}\replaced[id=mc]{[ \betab\au_{I}]}{( \betab\au_{I})} = 0 $ (or $K_I\replaced[id=mc]{[ \betab\au_{II}]}{( \betab\au_{II})}  = 0$, $K_{II}\replaced[id=mc]{[ \betab\au_{II}]}{( \betab\au_{II})} = 1 $). In this case, (\ref{eq:I_in_terms_of_Ks})  implies that
 \begin{equation*} K_{I,II}\replaced[id=mc]{[ \betab]}{( \betab)} = \frac{
\hat\Ic[ \betab , \betab\au_{I,II}, \delta \gammab ] } {\eta }. 
 \end{equation*}

Notice that the interaction integral can be regarded as a tool to {\it
  extract} the singular parts of fields $\betab^a,\betab^b$, as it
follows from (\ref{eq:interinta}). The regular part $\betab_S$ of
either field (\replaced[id=mc]{cf.}{c.f.} \eqref{eq:1}) does not contribute to the value of the interaction integral.

  \begin{justification}[Equation \eqref{eq:I_in_terms_of_Ks}] 
		Consider
    $\betab^a,\betab^b\in \Bsc^b$, and $r>0$. Notice that the two
    terms in the volume integral of (\ref{eq:interinta})  form
    an exact divergence. Applying the divergence theorem on $
    (B_\rho(\xt)\setminus\Csc)\setminus B_r(\xt)$ reveals that the
    integration in (\ref{eq:interinta}) over $
    (B_\rho(\xt)\setminus \Csc)\setminus B_r(\xt)$ and $\Csc^\pm_\rho \setminus
    B_r(\xt)$ add up to an integral over $\partial B_r(\xt)$. It then
    follows that
\begin{equation}
\begin{aligned} \hat\Ic\left[ \betab^a,\betab^b,\delta \gammab \right] = \lim_{r\to
0} \int_{\partial B_r(\xt)} \delta \gammab \cdot \overline
\Sigmab\left(\betab^a,\betab^b\right) \nb\;dS,
\end{aligned}\label{eq:int_integral_limit}
\end{equation} 
\added[id=mc]{with $\nb$ here is also used to denote the outward unit normal to $\partial B_{r}(\xt) $,} since the rest of the terms vanish as $r\to 0$. 

To proceed \replaced[id=mc]{in showing}{to prove} that \eqref{eq:int_integral_limit} implies \eqref{eq:I_in_terms_of_Ks}, we write $\betab^a=\betab_T^a+\betab_S^a$ and $\betab^b=\betab_T^b+\betab_S^b$, where $\betab_T^a, \betab_T^b\in \vectorspan\{\nabla \ub^I,\nabla \ub^{II}\}$ and $\betab_S^a, \betab_S^b\in H^1(\Bc_\Csc; \Rbb^{2\times2})]$. It is straightforward to show that
\[
\hat{\Ic}\left[ \betab_T^a,\betab_T^b,\delta \gammab \right] = \lim_{r\to
0} \int_{\partial B_r(\xt)} \delta \gammab \cdot \overline
\Sigmab\left(\betab_T^a,\betab_T^b\right) \nb\;dS = \eta \left( \replaced[id=mc]{K_I[\betab^a] K_I\left[\betab^b\right] + K_
{II}[\betab^a] K_{II}\left[\betab^b\right]}{K_I(\betab^a) K_I\left(\betab^b\right) + K_
{II}(\betab^a) K_{II}\left(\betab^b\right) }  \right).
\]
Therefore, it remains to show that $\hat{\Ic}[ \betab_S^a,\betab_T^b,\delta \gammab ]=\hat{\Ic}[ \betab_T^a,\betab_S^b,\delta \gammab ]=\hat{\Ic}[ \betab_S^a,\betab_S^b,\delta \gammab ]=0$. To this end, we first define
\[\hat{\Ic}_r\left[ \betab^a,\betab^b,\delta \gammab \right] := \int_{\partial B_r(\xt)} \delta \gammab \cdot \overline
\Sigmab\left(\betab^a,\betab^b\right) \nb\;dS.\]
Then we invoke the Cauchy-Schwarz inequality and the explicit expression of $\betab_T^b$ to obtain
\begin{equation}\label{eq:after_C_S}
\begin{aligned}
\left|\hat{\Ic}_r\left[ \betab_S^a,\betab_T^b,\delta \gammab \right]\right|&\le
C\|\betab_S^a\|_{L^2[\partial B_r(\xt)]} \left\|\betab_T^b\right\|_{L^2[\partial B_r(\xt)]} \le C\|\betab_S^a\|_{L^2[\partial B_r(\xt)]}, \\
\left|\hat{\Ic}_r\left[ \betab_S^a,\betab_S^b,\delta \gammab \right]\right|&\le
C\|\betab_S^a\|_{L^2[\partial B_r(\xt)]} \left\|\betab_S^b\right\|_{L^2[\partial B_r(\xt)]}, 
\end{aligned}	
\end{equation}
where $C>0$ is independent of $r$.

To continue, we need to invoke a trace inequality with a scaling \replaced[id=mc]{of}{ with} $r$ for any $f\in H^1[B_\rho(\xt)\setminus\Csc]$ and $r\in(0,\rho)$,
\begin{equation}\label{trace_inequality}
	 \|f\|_{L^2[\partial B_r(\xt)]}\le Cr^{1/2}\|f\|_{H^1[B_\rho(\xt)\setminus\Csc]},
\end{equation}
where $C>0$ is independent of $f$ and $r$\footnote{  To prove \eqref{trace_inequality}, we first write
 \begin{equation*}
 	\overline{f}:=\frac{1}{\pi\rho^2} \int_{B_\rho(\xt)} f\;d\Omega,\quad\hat{f}:=f-\overline{f}.
 \end{equation*}
 Then with a form of Poincar\'e's inequality and a scaling argument,
 \begin{equation*}
 	\left\|\hat{f}\right\|_{L^2[\partial B_r(\xt)]} \le Cr^{1/2} |f|_{H^1[B_r(\xt)\setminus\Csc]} \le Cr^{1/2} |f|_{H^1[B_\rho(\xt)\setminus\Csc]}.
 \end{equation*}
 On the other hand, since $\|\overline{f}\|_{L^2[\partial B_r(\xt)]}=(\overline{f}^2 2\pi r)^{1/2}$ and $\|\overline{f}\|_{L^2[B_\rho(\xt)]}=(\overline{f}^2 \pi \rho^2)^{1/2}$, we have
 \begin{equation*}
 	\left\|\overline{f}\right\|_{L^2[\partial B_r(\xt)]} = Cr^{1/2}\left\|\overline{f}\right\|_{L^2[B_\rho(\xt)]} \le 
 	Cr^{1/2}\left\|f\right\|_{L^2[B_\rho(\xt)]}.
 \end{equation*}
 Adding these two inequalities yields \eqref{trace_inequality}.}.

With \eqref{trace_inequality}, we then proceed to simplify \eqref{eq:after_C_S}:
\begin{equation*}
\begin{aligned}
\left|\hat{\Ic}_r\left[ \betab_S^a,\betab_T^b,\delta \gammab \right]\right|&\le
Cr^{1/2} \|\betab_S^a\|_{H^1[\partial B_\rho(\xt)\setminus\Csc]}, \\
\left|\hat{\Ic}_r\left[ \betab_S^a,\betab_S^b,\delta \gammab \right]\right|&\le
Cr\|\betab_S^a\|_{H^1[\partial B_\rho(\xt)\setminus\Csc]} \left\|\betab_S^b\right\|_{H^1[\partial B_\rho(\xt)\setminus\Csc]}.
\end{aligned}	
\end{equation*}

Thus, as $r\rightarrow0$, both $\hat{\Ic}_r\left[ \betab_S^a,\betab_T^b,\delta \gammab \right]$ and $\hat{\Ic}_r\left[ \betab_S^a,\betab_T^b,\delta \gammab \right]$ tend to zero. From symmetry in the first two slots of $\hat{\Ic}_r$, we also have $\lim_{r\rightarrow0}\hat{\Ic}[ \betab_T^a,\betab_S^b,\delta \gammab ]=0$.


\end{justification} 

\begin{rmk}[Relation to the energy release rate] 
The interaction integral functional
is directly related to the  energy release rate
 $\Gc:\Bsc^b \times \testset \to \Rbb$ \cite{griffith1920}, which can be
defined as 
\eq{ \Gc[ \betab,\delta \gammab] = \lim_{r\to 0 } \int_{\partial
B_r(\xt ) } \delta \gammab \cdot \Sigmab(\betab ) \nb\; dS,
\label{eq:energy_release_rate} 
} 
where $\Sigmab:\Rbb^{2 \times 2}\to\Rbb^{2 \times 2}$ is Eshelby's energy
momentum tensor \cite{eshelby1975} \added[id=mc]{and $\nb$ is used to denote the outward unit normal to $\partial B_{r}(\xt) $}. For linear elastic materials,
Eshelby's energy momentum tensor takes the form
\begin{equation*}
\Sigmab( \betab ) =
 \frac{1}{2} \cs(\betab):\betab \;\1 - \betab^\top \cs(\betab ) .
\end{equation*}
The above is related to \added[id=mc]{the} interaction energy momentum tensor by the following relation
\[
		  \overline\Sigmab\left(\betab^a,\betab^b\right) = \Sigmab\left(\betab^a+\betab^b\right)-\Sigmab(\betab^a)-\Sigmab\left(\betab^b\right) \quad \Rightarrow \quad \Sigmab\left(\betab^b\right) = \frac12 \overline \Sigmab\left(\betab^b,\betab^b\right).
\]
Comparing \eqref{eq:energy_release_rate} and \eqref{eq:int_integral_limit} and
exploiting the linearity of the constitutive relation, we have
the following relation between the interaction integral functional and the energy release
rate
\begin{equation} 
\label{eq:2}
\hat\Ic\left[\betab^a,\betab^b,\delta \gammab \right] =
\Gc\left[ \betab^a+\betab^b , \delta \gammab \right] - \Gc[ \betab^a, \delta
\gammab ] - \Gc\left[ \betab^b , \delta \gammab \right].  
\end{equation} 
After
 replacing with (\ref{eq:1}) and evaluating, the limit in
 \eqref{eq:energy_release_rate}  gives the widely known relation
\begin{equation*} 
\Gc[ \betab, \delta \gammab ] = \eta \left( \replaced[id=mc]{ K_I[\betab]^2 + K_{II}[\betab]^2}{ K_{I}(\betab)^2 +
K_{II}(\betab)^2 }\right).
\end{equation*} 
We can alternatively recover \eqref{eq:I_in_terms_of_Ks} by replacing
this relation into \eqref{eq:2}.  Thus, \eqref{eq:I_in_terms_of_Ks} can be
justified by the direct evaluation of the limit in
\eqref{eq:int_integral_limit} or by its relation to the energy
release rate.  It is worth noting that while $\Gc$ is a non-linear
functional in $\betab$, $\hat\Ic$ is linear in both $\betab^a$ and
$\betab^b$.
\end{rmk} 

\begin{rmk}[Constraints on $\Mc$]
In \eqref{eq:energy_release_rate}, $\delta
\gammab$ is understood as a
variation of material points that represents unit crack advancement. Thus, in
order to compute the energy release rate, we must enforce $\delta
\gammab=\gb_1(0)$ at $\xt$. This constraint is the justification  behind
the definition of $\testset$ in \eqref{eq:setv}. 
\end{rmk}

\subsection{Problem-dependent Interaction Integral Functional}
\label{subsxn:prob-dep-int}

We present here a functional $\Ic $ that takes the same value as $\hat\Ic$ when $\betab^a$ coincides with the solution of Problem \eqref{eq:elasticity}. Namely, if $\betab^a=\grad\ub$ where $\ub$ satisfies \eqref{eq:equilibrium} and \eqref{eq:traction},
then $\Ic[\betab^a,\betab^b,\delta\gammab] =
\hat\Ic[\betab^a,\betab^b,\delta\gammab]$ for any $\betab^b$. Therefore, in this case
$\Ic[\betab^a,\betab^b,\delta\gammab]$ is the interaction integral
between $\betab^a$ and $\betab^b$.

The motivation behind introducing $\Ic$ is to formulate a functional
defined over gradients of displacement fields that belong to classical
finite elements spaces, namely, a functional for which no second
derivatives of a numerical solution are needed.  This is possible
because, in contrast to $\hat \Ic$, $\Ic$ does not involve derivatives
of $\betab^a$.


Expanding the divergence in \eqref{eq:interinta}  and substituting with \eqref{eq:equilibrium} and \eqref{eq:traction} for $\betab^a=\grad\ub$  yields
\[
\begin{aligned}
\hat\Ic\left[\betab^a,\betab^b,\delta\gammab\right] = &\int_{\Csc^\pm_\rho } \delta \gammab \cdot \overline\taub \left(\betab^a , \betab^b \right)  \, dS \\
 - & \int_{ B_\rho(\xt) \setminus \Csc } \left[\overline\Sigmab \left(\betab^a , \betab^b \right) :\grad\delta\gammab  +  \delta \gammab \cdot \overline \lambdab\left(\betab^a , \betab^b \right) \right]\, dV,
\end{aligned}
\]
where 
\begin{equation}
\overline 
\taub\left(\betab^a , \betab^b \right) = \overline w\left(\betab^a , \betab^b \right)
\nb -
\betab^a\,^\top \cs\left( \betab^b \right)\nb  -  \betab^b\,^\top \bar
\tb,\label{eq:taub}
\end{equation}
and 
\begin{equation}
\overline \lambdab\left(\betab^a , \betab^b \right) = \betab^a: \grad\cs \left(\betab^b \right) -
 \cs\left(\betab^a\right):(\grad \betab^b\,)^\top - \betab^a\,^\top \dive
\cs\left( \betab^b \right) + \betab^b\,^\top \bb \,. \label{eq:lambda}
\end{equation}
{
\begin{rmk}[Indicial expression of relevant quantities]
For ease of implementation, we provide here the indicial representation of $\overline \Sigmab,\,\overline \taub, $ and $\overline \lambdab$ (making use of Einstein's repeated indeces convention), namely 
\[\begin{aligned}
\overline \Sigma_{ij}\left(\betab^a, \betab^b\right)  &= \sigma^a_{kl} \beta^b_{kl}\, \delta_{ij} - \beta^a_{ki} \sigma^b_{ kj} -\beta^b_{ki} \sigma^a_{ kj}, \\
\overline \tau_i\left( \betab^a, \betab^b\right) &=  \overline w n_i - \beta^a_{ji} \sigma^b_{jk} n_k - \beta^b_{ji} \overline{t}_j, \\
\overline \lambda_i\left( \betab^a, \betab^b\right)  &= \beta^a_{mn}\sigma^b_{mn,i} - \sigma^a_{kj} \beta^b_{ki,j}  - \beta^a_{ki} \sigma^b_{kj,j} + \beta^b_{ki}b_{k},
\end{aligned}
\]
where $\sigma^{a,b}_{ij}$ are understood as $\sigma(\betab^{a,b})_{ij}$.
\end{rmk}
}

Notice that for each pair $\left(\betab^a , \betab^b \right)$, $\overline \taub\left(\betab^a , \betab^b \right)$
and $\overline \lambdab\left(\betab^a , \betab^b \right)$ are functions over $\Csc^\pm_\rho$ and
$B_\rho(\xt) \setminus \Csc$, respectively. 

\todo{AL: The set $\Bsc^a$ does not include the exact solution. I
  think I need to add the span of the two singular behaviors. What do
  you guys think? I added the part that was missing. YS: I agree. }
To reflect the lower regularity needed for $\betab^a$,  we
first define a finite partition of $\Bc_\Csc$ as a  set
\replaced[id=mc]{$\{T_1,\ldots,T_N\}$}{$\{K_1,\ldots,K_N\}$} for some $N\in \mathbb N$ such that
\replaced[id=mc]{$T_i$}{$K_i$} is open for any $i$, \replaced[id=mc]{$T_i\cap T_j=\emptyset$}{$K_i\cap K_j=\emptyset$} for any \replaced[id=mc]{$i\neq j$}{$i,j$}, and
$\replaced[id=mc]{\bigcup_i \overline{T_i}}{\cup_i \overline{K_i}} = \overline {\Bc_\Csc}$. An example of such partition is a finite element mesh for $\Bc_\Csc$. Then, we can set 
\begin{equation*}
\begin{aligned}
\Bsc^a=& \left\{\betab \in  L^2\left(\Bc_\Csc, \Rbb^{2\times2}\right)
  \Big|\ \betab|_{ \replaced[id=mc]{T_i}{K_i}} \in H^1\left(\replaced[id=mc]{T_i}{K_i};\mathbb R^{2\times
  2}\right)  \text{ for any } \replaced[id=mc]{T_i}{K_i} \text{ in some finite partition of }
  \Bc_\Csc 
\right\} \\ &
\oplus \vectorspan\left\{\nabla
\ub^I,\nabla
\ub^{II}\right\} 
\end{aligned}
\end{equation*}
and define $\Ic : \Bsc^a\times\Bsc^b\times \testset\to \Rbb$ as 
\begin{equation}
\begin{aligned}
\Ic\left[\betab^a,\betab^b,\delta\gammab\right] = &\int_{\Csc^\pm_\rho } \overbrace{ \delta \gammab \cdot \overline\taub \left(\betab^a , \betab^b \right) }^{\Ic_1} \, dS \\
 - & \int_{ B_\rho(\xt) \setminus \Csc } \big[\, \underbrace{ \overline\Sigmab \left(\betab^a , \betab^b \right) :\grad\delta\gammab }_{\Ic_2} +  \underbrace{\delta \gammab \cdot \overline \lambdab\left(\betab^a , \betab^b \right)}_{\Ic_3}\, \big]\, dV.
\label{eq:3}
\end{aligned}
\end{equation}
In this way, $\betab^a$ can have discontinuities across a finite
number of interfaces, as in a finite element solution\added[id=mc]{s}, and still have
well-defined values at the crack faces (which a function in $L^2$
may not have). 
Note that $\Ic$ is linear in $\betab^b$ but affine in $\betab^a$, and
therefore, not symmetric with respect to them. The way $\Ic$ will be
used is by setting $\betab^a$ to be either the exact solution or a
numerical approximation to it, and $\betab^b$ to be an auxiliary field.

\subsection{The Fields} \label{subsxn:the_fields} We next proceed to construct
the \testfunctionname and auxiliary fields that will enable the extraction of the
stress intensity factors for curvilinear crack geometries.

\subsubsection{Material variation fields.} 
The
objective of this section is to construct vector fields $\delta \gammab$ that
belong to the set of  material variations $\testset$, see \eqref{eq:setv}. We  provide
two constructs,  but any $\delta \gammab \in \testset$ could be
used.  

We start from a general form $\delta \gammab(\xt + r{\boldsymbol e}_r) = q(r) \mathsf{t}(r)$ where $q:
\Rbb_0^+ \to \Rbb_0^+$ and $\mathsf{t}:\Rbb_0^+ \to \Rbb^2$. The function $q(r)$
represents the magnitude of the \testfunctionname field and is constructed to have support
within $B_{\rho}(\xt)$.  The function $\mathsf{t}(r)$ embodies the direction of
the \testfunctionname field and is taken to satisfy $|\mathsf{t}(r)|=1$, $\forall
r\in\Rbb_0^+$.

The magnitude and direction of the two \testfunctionname fields that we propose only
depend on $r$. We will thus abuse notation writing $\delta \gammab(r) $ in place
of $\delta \gammab(\xt + r{\boldsymbol e}_r)$.  

The scalar function $q(r) \in C^2(\Rbb_0^+)$ is defined  as
\begin{equation}
 q(r) = \begin{cases} 1, & \text{ if } r \le \rho_I, \\ 
f(r), & \text{ if }  \rho_I < r < \rho,\\ 0, & \text{ otherwise},
 \end{cases}
 \label{eq:cutoff}
\end{equation}
 with $f(r)$ being a fifth order polynomial and $\rho_I = \rho/4$. Note that
 to construct higher-order methods the regularity of $q(r)$, and  thus
 the polynomial order of $f(r)$, will have to be
 suitably adapted. 

In the sequel we list the two \testfunctionname fields:

\begin{enumerate}[(1)]
\item \emph{Unidirectional \testfunctionname
fields.} 
The first field is designed
to be constant within a distance $\rho_I$ from the crack tip. The field is then
constructed as 
\begin{equation} \delta
\gammab^\text{UNI}(r)= q(r) \gb_1(0). 
\label{eq:dguni} 
\end{equation}
This field satisfies 
\begin{equation}\label{eq:NoGradInside} 
\grad \delta\gammab^\text{UNI} = \0, \text{ for } r<\rho_I.  
\end{equation} 
Figure \ref{fig:deltaxconstant} shows its stream traces \added[id=mc]{alongside a circular arc crack}.

\item \emph{Tangential \testfunctionname fields.}
The second field is designed to be
tangential to the crack and is given by
\begin{equation}  \delta
\gammab^\text{TAN}(r) = q(r) \gb_1(r). 
\label{eq:dgtang} 
\end{equation}
This field satisfies 
\begin{equation}\label{eq:TanTestFunction}
\delta\gammab^\text{TAN} \cdot \nb = 0 \text{ on } \Csc^\pm_\rho.
\end{equation}
The stream traces of $\delta \gammab^\text{TAN}$\added[id=mc]{, for the particular case of a circular arc crack,} are shown in
Figure \ref{fig:deltaxtangent}.  
\end{enumerate}

\begin{figure}[htbp] 
\centering 
\subfloat[Constant within $|\x - \xt| \leq \rho_I$]{ 
\begin{tikzpicture}[x=2.5in,y=2.5in] \node at (0,0) {
\includegraphics[trim=1.5in 1.in 0.75in
0.5in,clip,height=0.3\textheight]{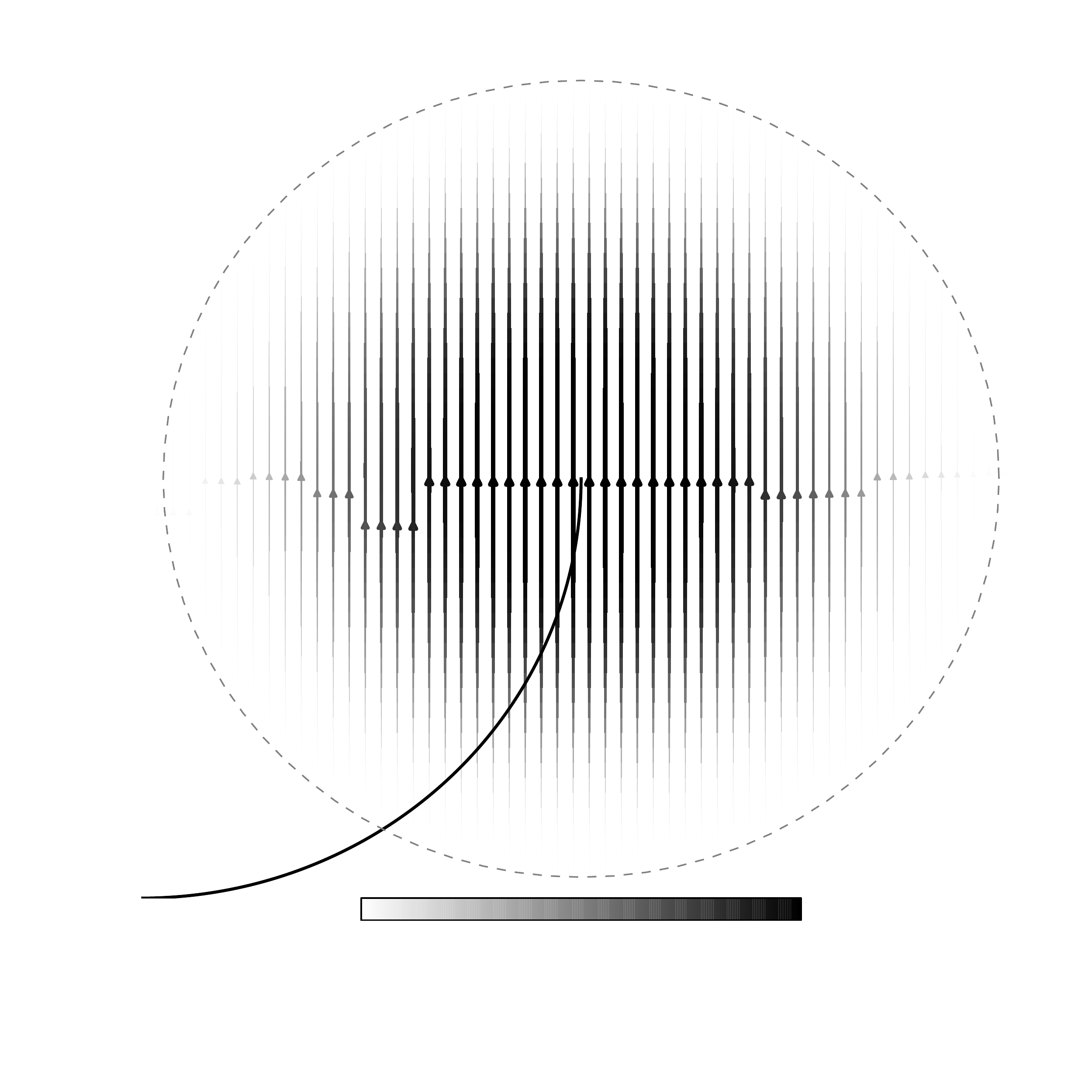} }; 
\node at (-0.29,-0.5) [right] { $0 \msp{0.49} | \delta\gammab|\msp{0.49} 1$}; 
\node at (-0.65,-0.425) [right] { Crack};
 \end{tikzpicture} \label{fig:deltaxconstant} 
}
\subfloat[Tangent to $\Csc$]{ 
\begin{tikzpicture}[x=2.5in,y=2.5in]
 \node at (0,0) { \includegraphics[trim=1.5in 1.in 0.75in
0.5in,clip,height=0.3\textheight]{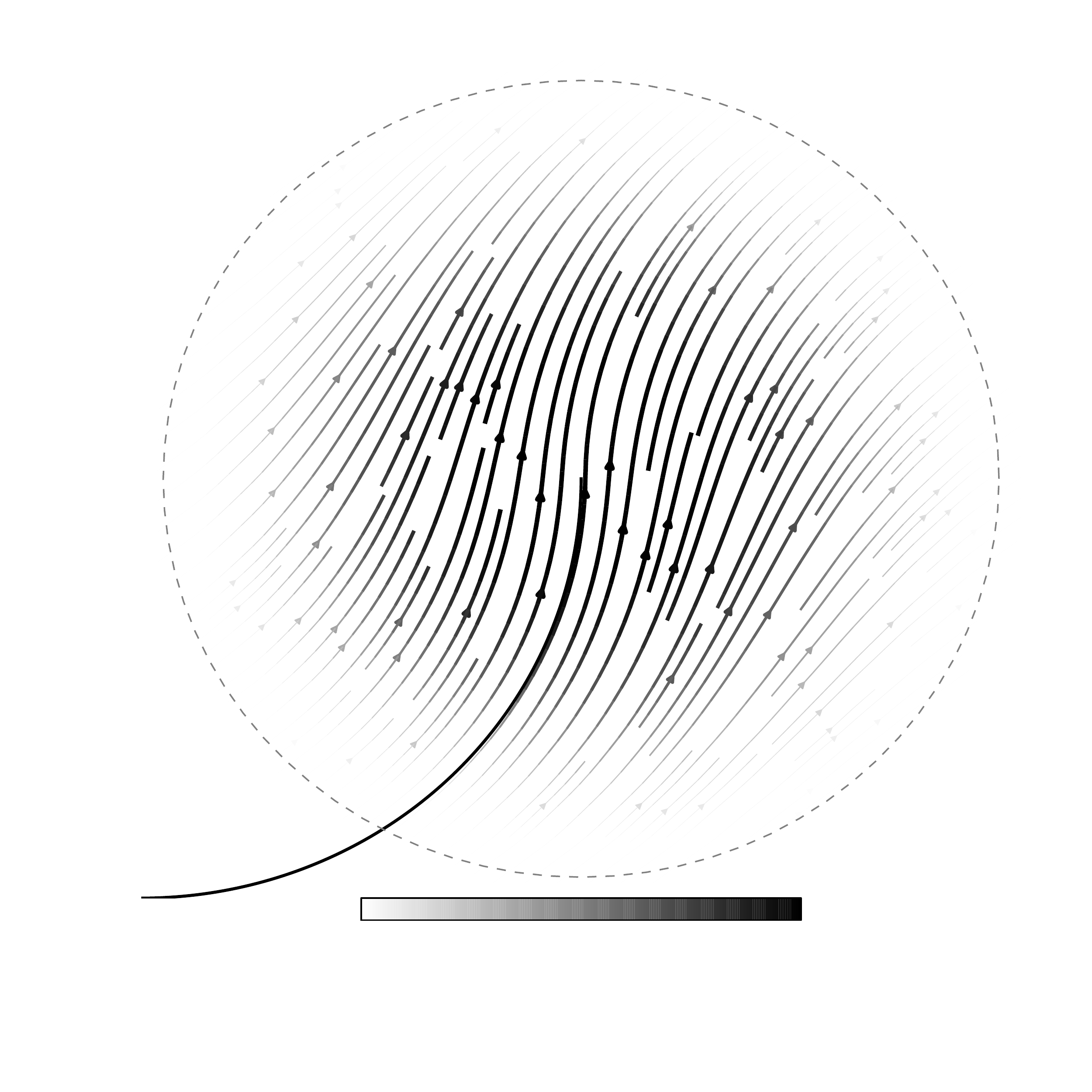} }; 
\node at (-0.29,-0.5) [right] { $0 \msp{0.49} | \delta\gammab|\msp{0.49} 1$};
 \node at (-0.65,-0.425) [right] { Crack}; 
\end{tikzpicture} 
\label{fig:deltaxtangent} 
}
\caption{Streamtraces for the \testfunctionname fields: (a)
$\delta \gammab^\text{UNI}$, (b) $\delta \gammab^\text{TAN}$.} 
\label{fig:deltax}
\end{figure} 

\begin{rmk}[Regularity of $\delta \gammab$] Note that since $\Gamma \in C^2(\Rbb_0^+)$ and $q \in C^2(\Rbb_0^+)$, both $\delta
\gammab^\text{UNI}$ and $\delta \gammab^\text{TAN}$ satisfy
the continuity requirement of $\testset$, namely, both are in $C^1(\Bc_{\Csc};\Rbb^2)$.
\end{rmk}

\subsubsection{Auxiliary fields.} 
\label{sec:auxiliary-fields}
As discussed
in \S \ref{subsxn:interaction_integral}, the objective is to construct tensor
fields
\begin{subequations}\label{eq:aux_field_requirements}{
\eq{
\betab\au_{I,II} \in \Bsc^b \label{eq:aux_field_requirement_1} 
}}
such that
 \begin{equation} \label{eq:sif_requirement} \left.  \begin{aligned} 
&K_{I}\replaced[id=mc]{[\betab\au_{I}]}{( \betab\au_{I} )} = 1, & K_{II}\replaced[id=mc]{[\betab\au_{I}]}{( \betab\au_{I} )} = 0, \\ 
&K_{I}\replaced[id=mc]{[\betab\au_{II}]}{( \betab\au_{II} )} = 0, & K_{II}\replaced[id=mc]{[\betab\au_{II}]}{( \betab\au_{II} )} = 1. \end{aligned} 
\right.  
\end{equation}

For a crack that is straight near the tip, namely,   $\Csc\cap B_\rho(\xt)$ is straight, a natural
choice is the strain fields of the solutions to pure modes $I$ and $II$ loading 
\citet{williams1952}. In fact, these solutions, appropriately scaled, satisfy
 \eqref{eq:sif_requirement} and the regularity requirement
\eqref{eq:aux_field_requirement_1}. Furthermore, the stress field
$\cs(\betab_{I,II}\au )$ is divergence free  and the fields $\betab_{I,II}\au $
are compatible, i.e., 
\begin{align} 
\dive \cs(\betab_{I,II}\au) &= \0
&&\text{in } B_\rho(\xt)\setminus \Csc, \label{eq:aux_field_requirement_2}\\
\exists \Phib \colon B_\rho(\xt)\setminus \Csc\to\Rbb^2 \text{ such that }
\betab_{I,II}\au &= \grad \Phib&&\text{in } B_\rho(\xt)\setminus \Csc,
\label{eq:aux_field_requirement_3} 
\end{align} 
as they are indeed derived from
gradients of vector fields.  Additionally the stress field \replaced[id=mc]{is}{  yields} traction-free \added[id=mc]{on the}
crack faces: \eq{ \cs(\betab_{I,II}\au) \nb =\0 \quad \text{on } \Csc_\rho^\pm.
\label{eq:aux_field_requirement_4} }	\end{subequations} These features allow
for significant simplifications of the interaction integral functional
in \eqref{eq:3}.

For curvilinear cracks, however, analytically obtaining auxiliary fields with the same
features  is not generally possible, since	a field that
satisfies all conditions (\ref{eq:aux_field_requirements})
is the solution of Problem \eqref{eq:elasticity} in the
neighborhood of $\xt$ for the
given curvilinear crack geometry $\Csc$. Instead, we will construct auxiliary fields that, although sufficiently regular and
satisfying \eqref{eq:sif_requirement}, may violate
\eqref{eq:aux_field_requirement_2}, \eqref{eq:aux_field_requirement_3}, or
\eqref{eq:aux_field_requirement_4}.  Needless to say that doing so  hinders
the simplification of the interaction integral functional, as  discussed in \S \ref
{subsxn:interaction_integrals}.

In the following we discuss two constructs of the auxiliary fields that satisfy
 \eqref{eq:aux_field_requirement_1} and
\eqref{eq:sif_requirement}: (1) \deleted[id=mc]{first} we present a compatible $\betab_{I,II}\au$
with divergence-free stress field $\cs(\betab_{I,II}\au)$, but for
which $\cs(\betab_{I,II}\au)$ \replaced[id=mc]{is not }{ does  not  yield }traction-free \added[id=mc]{on the} crack faces, and (2) then we introduce a 
variant of 
$\betab_{I,II}\au$ that is incompatible and whose stress field is
not divergence free, but its stress field \replaced[id=mc]{is }{ does
yield} traction-free  \added[id=mc]{on the} crack faces.

\begin{enumerate}[(1)]
 \item \emph{Divergence-free and compatible
(DFC) fields.} 
We first construct an
auxiliary field which satisfies conditions \eqref{eq:aux_field_requirement_1},
\eqref{eq:sif_requirement}, \eqref{eq:aux_field_requirement_2}, and
\eqref{eq:aux_field_requirement_3}.

To this extent consider the displacement fields obtained for a straight crack in
pure mode $\modes$ loading given by $\ub^\modes =\sum_{i,j} u^\modes_{ij} \gb_i(0)\otimes
\gb_j(0) $, where, for
completeness, the components $u^\modes_{ij}$ are recapitulated  in \S \ref{sxn:app1}. The auxiliary fields are then taken as
\begin{equation} \betab_{I,II}\au(r,\vartheta) :=
\betab_{I,II}^\text{DFC}(r,\vartheta) = \grad \ub^\modes(r,\vartheta),
\label{eq:betabdcf}
\end{equation} where for each $r$ the domain of definition of $\vartheta$ is $[-\pi
-\zeta(r) , \pi -\zeta(r)] $ as introduced in \S
\ref{subsxn:fracture_mechanics_problem}, rather than $[-\pi,\pi]$.



\item \emph{Traction-free (TF) fields.} 
We now construct auxiliary fields such that \eqref{eq:aux_field_requirement_1},
\eqref{eq:sif_requirement}, and \eqref{eq:aux_field_requirement_4} are
satisfied.

Consider the mapping  $\varphi\colon D_\rho \to [-\pi, \pi ] $ of the angular component of the polar
coordinate system introduced in \S \ref{subsxn:fracture_mechanics_problem}.
This mapping is designed to take a value of $\pm\pi$ on the crack faces and can
be constructed as 
\begin{equation*} \varphi( r, \vartheta ) = \vartheta +
\zeta( r  ).
\end{equation*} 
Values of $\varphi$ are plotted for a circular arc
crack geometry in Figure \ref{fig:gamma}.  

\begin{figure}[htbp] 
\centering 
\begin{tikzpicture}[x=2.5in,y=2.5in] 
\node at (0,0) { \includegraphics[trim=1.1in 1.in 0.75in
0.5in,clip,height=0.3\textheight]{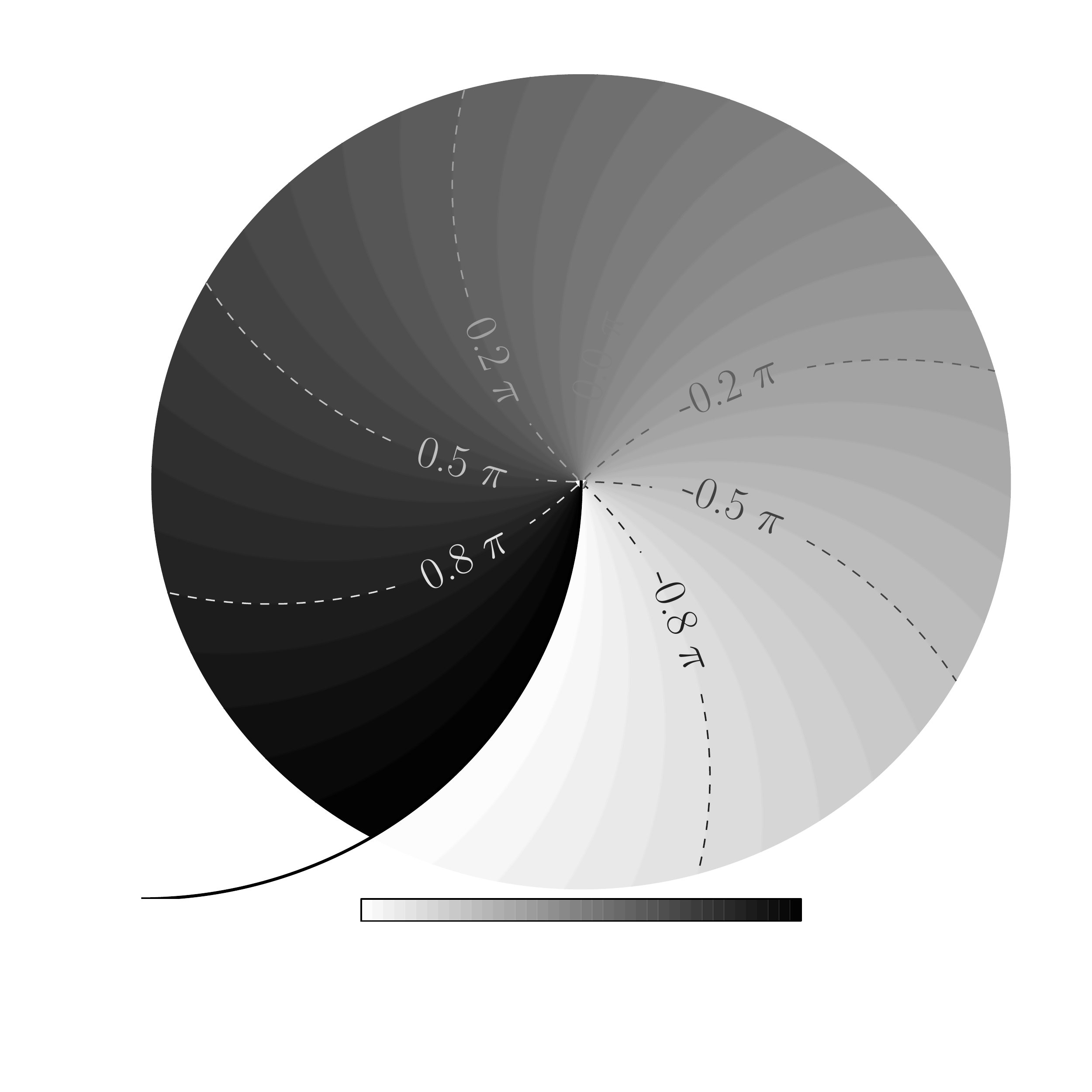} }; 
\node at (-0.315,-0.525)[right] { $-\pi\msp{0.375}\varphi
(r,\vartheta)\msp{0.4375} \pi$};
 \node at
(-0.65,-0.425) [right] { Crack};
 \end{tikzpicture} \caption{ The mapping
$\varphi(r,\vartheta) $ } \label{fig:gamma} 
\end{figure}

We then construct $\betab_\modes\au$ as 
\begin{equation}
\betab_\modes\au(r,\vartheta) :=  \betab_{I,II}^\text{TF}(r,\vartheta)
=\sum_{i,j} \bigg[ \grad \ub^\modes\bigg|_{( r, \varphi( r,\vartheta) )} : \gb_i(0)
\otimes \gb_j(0) \bigg] \gb_i(r) \otimes \gb_j(r).  
\label{eq:betabtf}
\end{equation} 
\added[id=al]{This auxiliary field is well defined for $r\in [0,\rho)$, where
$\zeta(r)$ is also well defined. Its values for $r>\rho$ do not
participate in the interaction integral, because of the support of
$\delta \gamma$, and hence are immaterial. Below we show that
$\betab_{I,II}^\text{TF}\in \vectorspan\left\{\nabla
\ub^I,\nabla
\ub^{II}\right\} \oplus  H^1(B_\rho(\x_t);\mathbb
R^{2\times 2})$, and hence that it can be extended to a function
$\betab_{I,II}^\text{TF}\in \vectorspan\left\{\nabla
\ub^I,\nabla
\ub^{II}\right\} \oplus H^1(\Bc_\Csc; \Rbb^{2\times2})= \Bsc^b$.}

The inspiration behind this construct is to transport $\grad
\ub^\modes$ from the straight crack faces, on which $\cs(\grad
\ub^\modes)$ is traction-free,  to the faces of the curvilinear crack, 
rotating $\grad
\ub^\modes$\added[id=mc]{,} and hence  $\cs(\grad
\ub^\modes)$\added[id=mc]{,} precisely by the angle between $\gb_1(r)$ and $\gb_1(0)$.  
This is generally  a incompatible field with
non-divergence-free stresses but traction-free crack faces.

\begin{justification}[{Traction-free} property]
We begin by computing the stresses from the constitutive relation
(\ref{eq:constitutiverelation}) on both sides of
(\ref{eq:betabtf}). Let then $\cs^\modes: \Rbb^+ \times (-\pi, \pi)\to
\Rbb^{2 \times 2}$ denote  $\cs(\grad\ub^\modes)$, which are precisely
the stress fields  of a straight crack ( see \S \ref{sxn:app1})
parallel to  the
local crack tip basis vector $\gb_1(0)$. These stress fields are
traction free along these straight faces, so $\cs^\modes(r,\pm
\pi)\gb_2(0)={\bf 0}$.  Then, on $\Csc_\rho^\pm$, we have
\begin{equation}\label{eq:TracFreeAuxField} 
\begin{aligned} 
\cs\left( \betab_{I,II}^\text{TF} \right)\nb&= \sum_{i,j} 
\bigg[ \cs^\modes ( r, \varphi( r,\vartheta) ) : \gb_i(0) \otimes \gb_j(0) \bigg] 
\gb_i(r) \otimes \gb_j(r)\nb
\\ &= \mp \sum_{i,j} \bigg[
\gb_i(0)\cdot \cs^\modes ( r, \pm\pi) 
 \gb_j(0)
\bigg] \gb_i(r) \delta_{j2} = {\bf 0},
\end{aligned} 
\end{equation} 
where we used that $\varphi = \pm \pi$ on $\Csc_\rho^\pm$.
\end{justification}

\todo{Guys, do we want the Justifications to be all in
  italics? I personally do not like it. YS: The reason I proposed using italics for them was to clearly separate them with the rest of the texts, for readers more interested in the method than the analysis. MC: I liked the italics as YS mentioned.}

\begin{justification}[Regularity of $\betab^\text{TF}$] It is not \emph{a
  priori} apparent that $\betab^\text{TF}\in \Bsc^b$, but it does. 
\deleted[id=al]{In fact, as we will see, $\betab^\text{TF}$ diverges at most as fast as $Cr^{-1/2}$ as $r\rightarrow0$, with $C$ independent of $r$.}
	To prove $\betab^\text{TF}\in \Bsc^b$
first note that $\betab_{I,II}^\text{DFC}\in\vectorspan\{\nabla
\ub^I,\nabla
\ub^{II}\}$. It is then enough to show that
$\betab_S:=\betab_{I,II}^\text{TF} - \betab_{I,II}^\text{DFC}\in
H^1(B_\rho(\x_t); \Rbb^{2\times2})$, and hence that it can be extended
to a function $\betab_S\in H^1(\Bc_\Csc; \Rbb^{2\times2})$. 
To this end, we write
\begin{equation}\label{betaS}
\begin{aligned}
\betab_S
&= \sum_{i,j} \left[ \grad \ub^\modes( r, \vartheta+\zeta(r) ) : \gb_i(0)
\otimes \gb_j(0)
- \grad \ub^\modes(r,\vartheta): \gb_i(r)
\otimes \gb_j(r)
 \right] \gb_i(r) \otimes \gb_j(r),
\end{aligned}
\end{equation}
where $\zeta(r)$ for $r>0$ satisfies that
{
\[
\cos\zeta(r) = s(r):=-\frac{\Gamma'(0)}{|\Gamma'(0)|} \cdot \frac{\Gamma(0)-\Gamma(r)}{|\Gamma(0)-\Gamma(r)|},\]
and $\zeta(0) = 0 = \lim_{r\to 0^+} \zeta(r)$. 
Hence,
\[
\begin{aligned}
\zeta'(r) &= \pm \frac{d}{dr} \arccos s(r) = 
\mp \frac{s'(r)}{\sqrt{1-s(r)^2}} \\
&=\mp \dfrac{1}{|\Gamma'(0)|}\left[
\dfrac{\Gamma'(0)\cdot\Gamma'(r)}{|\Gamma(r)-\Gamma(0)|} - 
\dfrac{\Gamma'(0)\cdot[\Gamma(r)-\Gamma(0)] \; \Gamma'(r)\cdot[\Gamma(r)-\Gamma(0)]}{|\Gamma(r)-\Gamma(0)|^3}\right]\\
&\quad\times\left[1-\left(-\frac{\Gamma'(0)}{|\Gamma'(0)|} \cdot \frac{\Gamma(0)-\Gamma(r)}{|\Gamma(0)-\Gamma(r)|}\right)^2\right]^{-1/2},
\end{aligned}
\]
which is well defined and continuous for any $r>0$. A tedious
calculation shows that $\zeta'(0)=\lim_{r\to 0^+}\zeta'(r)$  is also well-defined and given by
\[
|\zeta'(0)| = |\sqrt{-s''(0)}| = \frac{1}{2|\Gamma'(0)|^2}\left\{|\Gamma'(0)|^2|\Gamma''(0)|^2 - \left[\Gamma'(0)\cdot\Gamma''(0)\right]^2\right\}^{1/2}<\infty.
\]

Hence, there exists $C>0$ such that for all $r\in[0,\rho]$, 
\begin{subequations}
\begin{equation}\label{d_zeta_C}
	|\zeta'(r)|\le C,
\end{equation}
and thus
\begin{equation}\label{zeta_C}
	|\zeta(r)| \le Cr.
\end{equation}
Here and henceforth $C$ indicates a positive constant independent of
$r\in [0,\rho]$, whose value may change from line to line.
\end{subequations}

Next, as shown in \S\ref{sxn:app1},
$\grad\ub^\modes( r, \vartheta)=r^{-1/2}f^\modes(\theta)$, where
$f^\modes\in C^\infty( \mathbb R{;} \mathbb R^{2\times 2})$. From \eqref{betaS}, we can write
\begin{equation*}
\begin{aligned}
\betab_S
= r^{-1/2}\sum_{i,j} \left[ f(\vartheta+\zeta(r) ) : \gb_i(0)
\otimes \gb_j(0)
- f(\vartheta): \gb_i(r)
\otimes \gb_j(r)
 \right] \gb_i(r) \otimes \gb_j(r),
\end{aligned}
\end{equation*}
where we have omitted the superscript $\modes$ of $f$, as we shall do hereafter. It is then
straightforward to show that ${\betab_S}\in L^2(B_\rho(\x_t);
\Rbb^{2\times2})$. 

It remains to show that
 \[
 \frac{\partial\betab_S}{\partial r}, \frac{1}{r} \frac{\partial\betab_S}{\partial \vartheta}\in L^2\left(B_\rho(\x_t); \Rbb^{2\times2}\right).
 \] We first examine
\begin{equation}\label{beta_S_theta_der}
\begin{aligned}
\frac{1}{r} \frac{\partial\betab_S}{\partial \vartheta} &= r^{-3/2}\sum_{i,j} \left[ f'(\vartheta+\zeta(r) ) : \gb_i(0)
\otimes \gb_j(0)
- f'(\vartheta): \gb_i(r)
\otimes \gb_j(r)
 \right] \gb_i(r) \otimes \gb_j(r) \\
&= r^{-3/2}\sum_{i,j} \Bigl\{ [f'(\vartheta+\zeta(r) ) - f'(\vartheta)] : \gb_i(0)
\otimes \gb_j(0) \\
&\quad + f'(\vartheta) : [\gb_i(0)
\otimes \gb_j(0)
- \gb_i(r)
\otimes \gb_j(r)]
 \Bigr\} \gb_i(r) \otimes \gb_j(r).
\end{aligned}
\end{equation}
Since $f$ is $C^\infty$, we apply \eqref{zeta_C} and write
\begin{equation}
  \label{eq:5}
  \|f'(\vartheta+\zeta(r) ) - f'(\vartheta)\|_\infty \le \|f\|_{W^{2,\infty}} |\zeta(r)| \le Cr.
\end{equation}

On the other hand, we note that $\gb_1(r)=-\Gamma'(r)/|\Gamma'(r)|$ and $\gb_2(r)=\omega\cdot\gb_1(r)$ where $\omega:=-\ee_1\otimes\ee_2 + \ee_2\otimes\ee_1$, differentiating with respect to $r$ yields
\[
\gb_1'(r) = - \frac{\Gamma''(r)}{|\Gamma'(r)|} + \frac{\Gamma'(r)\cdot\Gamma''(r)}{|\Gamma'(r)|^3} \Gamma'(r), \quad
\gb_2'(r) = \omega \cdot \gb_1'(r),
\]
for $r\in [0,\rho]$.
Thus $\gb_1'(r)$ and $\gb_2'(r)$ are bounded, and hence
\begin{equation}
  \label{eq:4}
\begin{aligned}
\|\gb_i(0)\otimes \gb_j(0)- \gb_i(r)\otimes \gb_j(r)\|_\infty \le
Cr.
\end{aligned}  
\end{equation}
It follows from (\ref{beta_S_theta_der}), (\ref{eq:5}), and
 (\ref{eq:4} )that $\frac{1}{r}
\frac{\partial\betab_S}{\partial \vartheta}\in
L^2(B_\rho(\x_t);\Rbb^{2\times 2})$.

Now we compute
\[
\begin{aligned}
\frac{\partial\betab_S}{\partial r} &= -\frac{\betab_S}{2r} + r^{-1/2}\sum_{i,j} \Bigl\{ \zeta'(r) f'(\vartheta+\zeta(r) ) : \gb_i(0)
\otimes \gb_j(0) \\
&\quad- f(\vartheta): [\gb_i'(r)
\otimes \gb_j(r) + \gb_i(r)
\otimes \gb_j'(r)]
 \Bigr\} \gb_i(r) \otimes \gb_j(r) \\
&\quad + r^{-1/2}\sum_{i,j} \left[ f(\vartheta+\zeta(r) ) : \gb_i(0)
\otimes \gb_j(0)
- f(\vartheta): \gb_i(r)
\otimes \gb_j(r)
 \right] \\
&\quad \cdot [\gb_i'(r)
\otimes \gb_j(r) + \gb_i(r)
\otimes \gb_j'(r)]
\end{aligned}
\]
The analysis of the term $\betab_S/r$ is similar to the one performed
in \eqref{beta_S_theta_der}, and thus $\partial\betab_S/\partial r\in L^2(B_\rho(\x_t); \Rbb^{2\times2})$ follows from the boundedness of $\gb_1'$ and $\gb_2'$ and \eqref{d_zeta_C}.
%
%
%
}
\end{justification}

\comment{YS: Useful formulas:
\[
\begin{aligned}
s'(r) 
=
\begin{cases}
0, &r=0,\\
\dfrac{1}{|\Gamma'(0)|}\left[
\dfrac{\Gamma'(0)\cdot\Gamma'(r)}{|\Gamma(r)-\Gamma(0)|} - 
\dfrac{\Gamma'(0)\cdot[\Gamma(r)-\Gamma(0)] \; \Gamma'(r)\cdot[\Gamma(r)-\Gamma(0)]}{|\Gamma(r)-\Gamma(0)|^3}\right], &r>0.
\end{cases}
\end{aligned}
\]
and
\[
\begin{aligned}
s''(r) 
&= \frac{1}{|\Gamma'(0)|}  \biggl[\frac{\Gamma'(0)\cdot\Gamma''(r)}{|\Gamma(r)-\Gamma(0)|} - \frac{\Gamma'(0)\cdot[\Gamma(r)-\Gamma(0)]\;\Gamma''(r)\cdot[\Gamma(r)-\Gamma(0)]}{|\Gamma(r)-\Gamma(0)|^3}  \\
&-\frac{2\Gamma'(0)\cdot\Gamma'(r) \; \Gamma'(r)\cdot[\Gamma(r)-\Gamma(0)]
+ |\Gamma'(r)|^2 \Gamma'(0) \cdot[\Gamma(r)-\Gamma(0)]
}{|\Gamma(r)-\Gamma(0)|^3} \\
&+ \frac{3\Gamma'(0)\cdot[\Gamma(r)-\Gamma(0)] (\Gamma'(r)\cdot[\Gamma(r)-\Gamma(0)])^2}{|\Gamma(r)-\Gamma(0)|^5}
\biggr], \quad r>0,
\end{aligned}
\]
while
\[
s''(0) = - \frac{1}{4|\Gamma'(0)|^4} \left\{|\Gamma'(0)|^2|\Gamma''(0)|^2 - \left[\Gamma'(0)\cdot\Gamma''(0)\right]^2\right\}.
\]
}
%

\end{enumerate}

\subsection{Simplified Expressions for the Interaction Integral Functionals}
\label{subsxn:interaction_integrals} We describe three pairs of
\testfunctionname fields $\delta\gammab$ and auxiliary fields
$\betab^b=\betab\au$, and for each pair we provide the simplified
expressions of the interaction integral functional $\Ic[\betab,
\betab\au, \delta\gammab]$ that results from substituting the two
fields. In this section  we have removed subscripts $\modes$ from the
auxiliary fields, as the following results are independent of the
choice of the mode of interest, and doing so clarifies the
presentation.

We begin by stating two  results used in obtaining the simplified
expressions:  (1) for traction-free
auxiliary stress fields $\cs(\betab\au)$, such as
$\cs(\betab^\text{TF})$, and tangential \testfunctionname fields, such
as $\delta\gammab^\text{TAN}$, we have \eq{ \delta\gammab^\text{TAN}
  \cdot \overline \taub\left(\betab,\betab^\text{TF}\right) =
  -\delta\gammab^\text{TAN} \cdot \betab^\text{TF}\,^\top \bar
  \tb, \label{eq:esh_trac} } and (2) for compatible and
divergence-free auxiliary fields, such as $\betab^\text{DFC}$, we have
\eq{ \overline \lambdab\left(\betab, \betab^\text{DFC} \right) =
  \betab^\text{DFC}\,^\top \bb.
\label{eq:esh_dive}
}

\begin{justification}[Equations \eqref{eq:esh_trac} and \eqref{eq:esh_dive}]
We begin with \eqref{eq:esh_trac}. Recalling the expression \eqref{eq:taub} for $\overline\taub\left(\betab,\betab\au\right)$ we have, over $\Csc_\rho^\pm$,
\begin{equation}
\label{eq:TractionSubs}
\delta \gammab \cdot\overline\taub\left(\betab,\betab\au\right) = \overline
w\left(\betab,\betab\au\right) \delta\gammab \cdot \nb- \delta\gammab \cdot \betab
\,^\top\cs\left(\betab\au\right)\nb - \delta\gammab \cdot\betab\au\,^\top \overline
\tb.  
\end{equation}
Since we assumed that $\delta \gammab$ is a tangential \testfunctionname field ($\delta \gammab
\cdot \nb = 0 $) and because
$\cs(\betab\au)$ is traction free ($\cs(\betab\au)\nb =0 $ on $\Csc_\rho^\pm$) then \eqref{eq:esh_trac} holds.

Next, we look at \eqref{eq:esh_dive}. Recall that, from \eqref{eq:lambda},
\begin{equation*} 
\begin{aligned} 
\overline \lambdab(\betab,\betab\au) = \betab: \grad\cs (\betab\au ) -
 \cs\left(\betab\right) :(\grad \betab\au\,)^\top- \betab\,^\top \dive
\cs(\betab\au) + \betab\au\,^\top \bb.
  \end{aligned}
\end{equation*}
Since we assumed that $\cs(\betab\au)$ is divergence-free, the third term
above vanishes. Furthermore, since we assumed that $\betab\au$ is
compatible, there
exists $\Phib\colon  B_\rho(\xt)\setminus \Csc\to \Rbb^2$ such that $\betab\au = \grad \Phib$.
Exploiting the major and minor symmetries of the constitutive tensor $\Cbb$, we
have
\begin{equation*}
\begin{aligned} &\betab : \grad
\cs\left( \betab\au \right) -  \cs(\betab): \grad
\betab\au\,^\top = \betab: \Cbb : \grad \grad
\Phib - \betab : \Cbb: \grad \grad\Phib  = \0.  
\end{aligned}
\end{equation*}
\comment{
\[
	\dive \overline\Sigmab = \dive ( \betab: \Cbb : \betab\au - \betab^\top \cs(\betab\au) - \betab\au\,^\top \cs(\betab).
\]
Assume $\exists \Phib \in \Rbb^2$ with components $\Phi_i$ such that $\betab\au = \grad\Phib \Rightarrow \beta\au_{ij} = \Phi_{i,j}$. I am going to abuse notation and let $\cs \leftarrow \cs(\betab) = \Cbb:\betab, \cs\au \leftarrow \cs(\betab\au) = \Cbb:\betab\au $
\[
\begin{aligned}
	\Sigma_{ij,j} & = ( \betab_{mn}\Cbb_{mnpq}\beta\au_{pq} \delta_{ ij} -
	\beta_{ki} \sigma\au_{kj} - \beta\au_{ki}\sigma_{kj} )_{,j} \\
	& =   \betab_{mn,i}\Cbb_{mnpq}\beta\au_{pq}  +
	\betab_{mn}\Cbb_{mnpq}\beta\au_{pq,i} 
	- \beta_{ki} \cancel{\sigma\au_{kj,j}} - \beta\au_{ki}\cancel{\sigma_{kj,j} }
	- \beta_{ki,j} \sigma\au_{kj} - \beta\au_{ki,j}\sigma_{kj} = \\
	& =  \betab_{mn,i}\Cbb_{mnpq}\beta\au_{pq}  +
	\betab_{mn}\Cbb_{mnpq}\beta\au_{pq,i} 
	- \beta_{ki,j} \Cbb_{mnkj} \beta\au_{mn} - \beta\au_{ki,j} \Cbb_{mnkj} \beta_{mn} = \\
	& =  u_{m,ni}\Cbb_{mnpq}\Phi_{p,q}  +
	u_{m,n}\Cbb_{mnpq}\Phi_{p,qi} 
	- u_{k,ij} \Cbb_{mnkj} \Phi_{m,n} - \Phi_{k,ij} \Cbb_{mnkj} u_{m,n} = \\
		& =  u_{m,ni}\Cbb_{mnpq}\Phi_{p,q}  +
	u_{m,n}\Cbb_{mnpq}\Phi_{p,qi} 
	- u_{k,ji} \Cbb_{kjmn} \Phi_{m,n} - \Phi_{k,ij} \Cbb_{mnkj} u_{m,n} = 0
\end{aligned}
\]
In the third term I used $u_{k,ij} = u_{k,ji}$ and $\Cbb_{mnkj}= \Cbb_{kjmn}$

}
Thus, \eqref{eq:esh_dive} holds.
\end{justification}

We  now present the simplified expressions for the functional
obtained for each pair, as well as for the particular
case of rectilinear cracks, in order to re-connect these results
with what is commonly found in the literature.

\begin{enumerate}[(1)]
\item \emph{Unidirectional \testfunctionname
with divergence-free auxiliary fields.} 
We set $\delta\gammab=\delta\gammab^\text{UNI}$ and
$\betab\au=\betab^\text{DFC}$. Then  \eqref{eq:NoGradInside} implies that
the domain of integration of $\Ic_2$ reduces to $[B_\rho(\xt) \setminus
B_{\rho_I}(\xt)]\setminus \Csc$. Substituting  \eqref{eq:esh_dive}  in 
\eqref{eq:3} then simplifies to
\eq{ \begin{aligned}
\Ic\left[\betab,\betab^\text{DFC},\delta\gammab^\text{UNI}\right]  = & 
\int_{\Csc_\rho^\pm} 
\delta\gammab^\text{UNI}\cdot \overline\taub \left(\betab,\betab^\text{DFC}\right) \;dS
-\int_{ B_\rho(\xt)\setminus \Csc } \betab^\text{DFC}\,^\top \bb\cdot
\delta\gammab^\text{UNI}\;dV.  \\ 
& -\int_{ [B_\rho(\xt) \setminus B_{\rho_I}(\xt)]\setminus \Csc }
\overline\Sigmab\left(\betab,\betab^\text{DFC}\right):\grad\delta
\gammab^\text{UNI}\;dV.
\end{aligned} 
\label{eq:interint_uni_dive}
} 

\item \emph{Tangential \testfunctionname with divergence-free auxiliary fields.}
A slight variation of the previous pairing
is the combination $\delta\gammab=\delta\gammab^\text{TAN}$ and
$\betab\au=\betab^\text{DFC}$.  Applying \eqref{eq:TractionSubs},
\eqref{eq:TanTestFunction}, and \eqref{eq:esh_dive} yields
 \eq{
\begin{aligned}
\Ic\left[\betab,\betab^\text{DFC},\delta\gammab^\text{TAN}\right] = &
-\int_{\Csc_\rho^\pm}\left[\delta\gammab^\text{TAN}\cdot\betab
\,^\top\cs\left(\betab^\text{DFC}\right)\nb  +
\delta\gammab^\text{TAN}\cdot\betab^\text{DFC}\,^\top\overline \tb
\right]\;dS  \\ &-\int_{ B_\rho(\xt)\setminus \Csc  } \big[ 
\overline\Sigmab\left(\betab,\betab^\text{DFC}\right):\grad\delta\gammab^\text{TAN}+    \betab^\text{DFC}\,^\top \bb\cdot
\delta\gammab^\text{TAN} \big]\;dV.  \end{aligned} \label{eq:interint_tang_dive}}

\item \emph{Tangential \testfunctionname with traction-free auxiliary fields.}
Here we employ $\delta\gammab=\delta\gammab^\text{TAN}$ and
$\betab\au=\betab^\text{TF}$.  Applying \eqref{eq:TractionSubs},
\eqref{eq:TanTestFunction}, \eqref{eq:TracFreeAuxField} leads to
\eq{ \begin{aligned}
\Ic\left[\betab,\betab^\text{TF},\delta\gammab^\text{TAN}\right]  = & 
-\int_{\Csc_\rho^\pm}
\delta\gammab^\text{TAN}\cdot\betab^\text{TF}\,^\top\overline \tb \;dS \\
&-\int_{ B_\rho(\xt)\setminus \Csc } \big[
\overline\Sigmab\left(\betab,\betab^\text{TF}\right):\grad\delta
\gammab^\text{TAN} +
\overline \lambdab \left( \betab , \betab^\text{TF} \right)\cdot\delta \gammab^\text{TAN} \, \big] dV. 
 \label{eq:interint_tang_trac} \end{aligned} 
}

\item \emph{Locally rectilinear cracks.}
Finally it is worth noting that in the particular case of a locally linear crack
geometry, i.e. $\Gamma''(r) = 0, \forall r \in [0,\rho]$,
$\delta \gammab=\delta\gammab^\text{TAN}=\delta\gammab^\text{UNI}$ and
$\betab\au=\betab^\text{DFC}=\betab^\text{TF}$, the interaction integrals of
 \eqref{eq:interint_uni_dive}, \eqref{eq:interint_tang_dive} and
\eqref{eq:interint_tang_trac} all simplify to
\begin{equation*} \begin{aligned}
\Ic[ \betab,\betab\au, \delta \gammab] =& -\int_{\Csc_\rho^\pm}
\delta\gammab\cdot\betab\au\,^\top\overline \tb \;dS-\int_{[B_\rho(\xt)
\setminus B_{\rho_I}(\xt)]\setminus \Csc
}\overline\Sigmab(\betab,\betab\au):\grad \delta \gammab\; dV  \\ &-
\int_{B_\rho(\xt)\setminus\Csc } \betab\au\,^\top\bb \cdot
\delta\gammab\; dV \end{aligned} \end{equation*} which is the traditional
expression of the interaction integral for a straight crack first introduced in
\citet{yau1980} and commonly found in the literature. 
\end{enumerate}

\added[id=al]{The presence of singularities in some of the factors in each one of
the terms in \eqref{eq:3} raises the question of whether the integrals
therein are well-defined.  For the choices of material variation and
auxiliary fields  above, the three terms are in fact integrable. It
is straightforward\deleted[id=mc]{s} to see that  $\Ic_2$ is integrable, and the integrability of $\Ic_1$
and $\Ic_3$ is discussed below.}

\begin{justification}[Integrability of $\Ic_1$]
We show the integrability of $\Ic_1$ by considering each term of
$\overline\taub$ (\replaced[id=mc]{cf.}{c.f.} \ref{eq:taub}). For the first
term of $\overline\taub$, notice that $\overline w(\betab,\betab\au) {=O(r^{-1})}$ and
 $\delta \gammab^\text{UNI}\cdot \nb  {=O(r)} $  as $r\to 0$,  and hence
 the first term in $\delta\gammab\cdot\overline\taub$
 asymptotically behaves as a constant near the crack tip. For the second
 term  of $\overline\taub$, we only need to consider the
 case $\betab\au=\betab^\text{DFC}$. Notice that  on the crack faces,
 $\vartheta=\pm\pi - \zeta(r)$, and therefore  $\cos(\vartheta/2),
 \cos(3\vartheta/2)=O(r)$. Using the expressions for $\cs^\modes$
 in \S \ref{sxn:app1}, this implies that  
 \[\cs(\betab^\text{DFC}):\gb_1(0)\otimes\gb_2(0), \cs(\betab^\text{DFC})\colon\gb_2(0)\otimes\gb_2(0)=O(r)
 \] 
on the crack faces near the crack tip. Moreover, since $\nb=\gb_2(r)=\gb_2(0)+O(r)$ and $\gb_1(0)\cdot\gb_2(0)=0$, we have $\cs(\betab^\text{DFC})\nb 
 =O(r)$. Finally, since $\betab =O(r^{-1/2})$ close to the
 crack tip, we can conclude that $\betab^\top
 \cs(\betab^\text{DFC})\nb \sim r^{1/2}$ as $r\to 0$, and hence it is integrable. 
For the third term in $\overline\taub$,  if $\overline\tb \in L^\infty(\Csc_\pm)$ then $\delta \gammab \cdot 
 \betab\au\,^\top \bar\tb = O(r^{-1/2})$ as $r\to0$, which is
 integrable as well.
\end{justification}

\begin{justification}[Integrability of $\Ic_3$]
As discussed in \S \ref{subsxn:interaction_integrals}, we know $\overline
\lambdab(\betab,\betab^\text{DFC} ) = \betab^\text{DFC}\,^\top\bb$. If $\bb \in
L^\infty(\Bc_\Csc )$ then $\betab^\text{DFC}\,^\top \bb = O(r^{-1/2})$ as $r\to0$.
Therefore, for $\betab\au=\betab^\text{DFC},$
$\Ic_3$ is integrable.


For $\betab\au=\betab^\text{TF}$, we begin by taking advantage of \eqref{betaS} and the linearity of $\overline\lambdab$ in the second argument to write
\[
 \overline \lambdab\left( \betab, \betab^\text{TF}  \right) =
\overline \lambdab\left( \betab, \betab^\text{DFC} \right) + \overline\lambdab\left( \betab, \betab_S \right).
\]
But from earlier discussion about the regularity of
$\betab^\text{TF}$, $\grad\betab_S=O(r^{-1/2})$ as $r\rightarrow0$. Thus, it
is straightforward to show that $\overline\lambdab\left( \betab,
  \betab_S \right)=O(r^{-1})$, and hence $\Ic_3$ is integrable.
%
\end{justification}

\subsection{Choosing the Interaction Integral Functional to Use} \label{subsxn:concl_rem} Before introducing the
numerical approximation of the above integrals it is worth making some remarks
on which functional is best suited for a specific application.

When the crack faces are loaded, a boundary integral over the faces
has to be carried out irrespective of the auxiliary fields. For this
particular problem it may be appealing to choose a pairing with
$\betab^\text{DFC}$ such as \eqref{eq:interint_uni_dive} and
\eqref{eq:interint_tang_dive}.  Doing so reduces the numerical
complexity of the interaction integral as $\overline\lambdab$ greatly
simplifies (and vanishes identically in the absence of body forces).

If the crack faces are traction free, it can be appealing to compute
the value of the interaction integral merely as  a domain integral, as
in \eqref{eq:interint_tang_trac}. This eliminates the need to construct quadrature rules over the crack faces. 
\sorted{
\todo{AL: Maurizio, what's the
meaning of this? you are using piecewise linear functions}
\todo{ 
MC: I removed the comment and replaced it with a straight forward statement.
}
}
 Furthermore in the presence of
 body forces, the integrand $\overline\lambdab$ will be non-zero even
with $\betab^\text{DFC}$, thus requiring the computation of the domain
integral. For this particular case using $\betab^\text{TF}$ will
result in a computationally more efficient technique.

\begin{rmk}[Omission of unidirectional material variation with traction-free auxiliary fields] 
The pairing $\delta\gammab^\text{UNI}$ with $\betab^\text{TF}$ is
omitted because it provides no advantage over other pairings. In fact,
because of $\delta \gamma^{UNI}$,  we have to compute the boundary
integral $\Ic_1$ regardless of the loads on the body. Similarly,
because of $\betab^\text{TF}$, we have to perform the domain integral
associated with the divergence of the reciprocal energy momentum
tensor regardless of the loads on the body. It is thus apparent that
for this particular pairing  we do not eliminate neither $\Ic_1$ nor
$\Ic_3$, unlike for other pairings (when traction and body forces are
zero).  Therefore this pairing would result in computationally
inefficient formulation, with no apparent advantage over other pairings.
\end{rmk}
\sorted{
\todo{Maurizio, are there any considerations in terms of the observed
  performance in the numerical examples later on, that should be
  mentioned here?}
  
\todo{Computationally I have not noticed any advantage over using one method over the other. I will perform some comparisons albeit the computation of the SIFs is really not demanding (fractions of the solve time ) as the dimension of the problem grows. Considering also the length of the paper I do not think a discussion of the compute time would be worthwhile.}  
}


\section{Numerical Computation of the Interaction Integral} \label{sxn:computation} \newcommand{\FEM}{finite
element method }

In this section we are concerned with the computation of the
interaction integral between any of the auxiliary fields and the
solution $\ub$ of Problem \eqref{eq:elasticity}. The solution $\ub$
and its gradient $\betab=\nabla\ub$ are going to be approximated by a
convergent sequence of displacement fields
$\{\ub^h\}_h$ and strain fields $\{\betab^h=\nabla\ub^h\}_h$,
respectively, or discrete solutions.  We first give some general
considerations on the expected conditions for convergence of the
computed stress intensity factors that are independent of the method
adopted to compute $\betab^h$. Then we particularize some of
these results to a $\{\betab^h\}$ that stems from a sequence of finite element
approximations. Additionally, we discuss some minor changes needed when
the
exact domain needs to be approximated as well because of the presence
of curvilinear cracks. The result of this
section is an algorithm to compute $\Ic$, summarized at the end of
this section for readers interested in its implementation.

\subsection{Approximation of the Interaction Integral}
\label{sec:discr-inter-integr}

Given a sequence of discrete solutions $\betab^h\to \betab$ in a sense
to be specified later, it defines a sequence of values for the
interaction integral $\Ic[\betab^h,\betab\au,\delta \gammab]$, and
hence a sequence of approximate stress intensity factors 
\eq{
  K_\modes^h(\betab^h) = \frac{ \Ic\left[\betab^h,
      \betab\au_\modes,\delta\gammab\right]}{\eta}.
\label{eq:sif_discrete}
}
For the approximate stress intensity factors to converge to the exact
ones $K_{I,II}\replaced[id=mc]{[\betab]}{(\betab)}$ as $h\searrow 0$, it is enough for $\Ic$ to be continuous with
respect to its first argument in the topology in which $\betab^h$
converges to $\betab$. It is simple to see then that for the stress
intensity factors computed with 
 \eqref{eq:interint_tang_trac}, it is
enough to have $\betab^h\to \betab$ in $L^2(\Bc_{\Csc}\replaced[id=mc]{;}{,}\mathbb
R^{2\times 2})$, because these functionals do not
involve integration of $\betab$ over $\Csc_\rho^\pm$. In contrast, for
the stress intensity factors computed with
 \eqref{eq:interint_uni_dive} and \eqref{eq:interint_tang_dive}, we
 additionally need to request $\betab^h\to\betab$ in $L^2(\Csc_\pm\replaced[id=mc]{;}{,}\mathbb
R^{2\times 2})$.

\subsubsection{Finite-Element-Based Approximations }
For sequences $\{\ub^h\}_h$ constructed with some finite element
spaces there is an important advantage of having a functional $\Ic$
continuous in its first argument. That is, the order of convergence of
the stress intensity factors doubles the order of convergence of
$\betab^h$ to $\betab$ \cite[contain related results, and see \S
\ref{sxn:convergence_functional}]{GuMa1994,buscaglia2000sensitivity}, so the values of the stress intensity factors
are a lot more accurate than the discrete solution itself. It is not difficult
to check that $\Ic$ in \eqref{eq:interint_tang_trac} is continuous in
its first argument in $L^2(\Bc_{\Csc}\replaced[id=mc]{;}{,}\mathbb R^{2\times
  2})$. Therefore, we can conclude that if
$\|\betab^h\to\betab\|_{L^2(\Bc_{\Csc}\replaced[id=mc]{;}{,}\mathbb R^{2\times 2})} \le C
h^k$, then
$|K_\modes^h(\betab^h)-K_\modes(\betab)|\le C h^{2k}$, for some $C>0,
k\in \mathbb N$ independent of $h$.

The functional $\Ic$ given by \eqref{eq:interint_uni_dive} or
\eqref{eq:interint_tang_dive} is not continuous in its first argument
in $L^2(\Bc_{\Csc}\replaced[id=mc]{;}{,}\mathbb R^{2\times 2})$, because of the boundary
integrals. As described in \S \ref{sxn:convergence_functional}, the
result that states that the order of convergence of $K^h_\modes$
should double that of $\betab^h$ does not apply in this
case. Nevertheless, as shown later in the numerical examples, the
rates of convergence seem to double as well for these two functionals.

The values of $k$ of the numerical methods used for the numerical
examples in \S \ref{sxn:numerical_examples} are $0.5$ and $1$, and
thus these methods converge at the rates of $\Oc(h)$ and $\Oc(h^2)$,
respectively.  In order to achieve higher order of accuracy within the
context of finite element methods, it is necessary to make use of
alternative methods that can accurately resolve the stress
singularity, such as \cite{liu2004, shen2009, chiaramonte2014b}.
Furthermore, for curvilinear cracks, high-order approximations of
the crack faces are needed to attain a corresponding order of accuracy of the method.

\subsection{Discrete Interaction Integral Functional}
\label{sec:prec-defin-discr}

One of the delicate issues to be addressed in this section is the fact
that for curvilinear cracks each discrete solution is computed on an
approximation of the exact domain. 
The precise steps to handle the difference between exact and
approximate domains in finite element methods are fairly standard, and
hence are often skipped in the description of new methods. We decided
to discuss this part with some additional detail here because of the
presence of the boundary integrals. The uninterested reader could
simply skip to the next section. 


For each $h$,  the discrete
solution $\replaced[id=mc]{\ub^h}{\ub_h}$ is computed on a domain $\Bc^h_{\Csc}$ with crack
faces $\Csc^h_\pm$. We assume that as $h\searrow 0$ the approximate
domain and the approximate crack faces and their normal vectors
converge to the exact ones\footnote{A possible condition is that for each
  $h$ there exists a one-to-one map $\boldsymbol{\Psi}_h\in M_h =
  W^{1,\infty}(\Bc_\Csc; \Bc^h_{\Csc})$ that converges to the identity
  in $M_h$ at a suitable rate, with $\text{det }\nabla
  \boldsymbol{\Psi}_h>\epsilon$ for some $\epsilon>0$ uniformly in
  $h$, and such that $\Bc^h_{\Csc}=\boldsymbol{\Psi}_h(\Bc_\Csc)$ and
  $\Csc^h_\pm=\boldsymbol{\Psi}_h({\Csc_\pm})$.  This is a type of
  condition for finite element approximations, and it is simply a
  condition on the way the approximate domains are to be constructed;
  we will not need to explicitly construct $\boldsymbol{\Psi}_h$ to
  compute the interaction integrals. }. For example, a
standard isoparametric mapping will suffice. 

Because the discrete solutions are defined over different domains, the
interaction integral functional $\Ic$ needs to be approximated as
well (the integrals over $\Bc_\Csc$ and $\Csc_\pm$ could not be
computed for the discrete solutions otherwise). Thus, for each $h$ we construct a discrete interaction
integral functional $\Ic^h:\Bsc^a_h \times \Bsc^b \times \Mc \to
\Rbb$, where $\Bsc^a_h$ is defined analogously to
$\Bsc^a$, but considering $\Bc_\Csc^h$ as
the domain of the problem.  Then, 
given a sequence of solutions $\replaced[id=mc]{\ub^h}{\ub_h}$ converging to the
exact solution
$\ub$ in $H^1$\footnote{For example, $\ub-\ub_h\circ\boldsymbol
  \Psi_h\searrow 0$ in $H^1(\Bc_\Csc;\mathbb R^2)$.}, we expect  $
 \lim_{h\searrow
0}\Ic^h[\betab^h,\betab\au,\delta\gammab] =
\Ic[\betab^e,\betab\au,\delta\gammab]$, for any of the $\betab\au\in
\Bsc^b$ and any $\delta\gammab \in \Mc$, where  $\betab^h
=\nabla \ub^h$ and $\betab^e=\nabla \ub$.  Equivalently, letting  the approximate stress intensity factors
 $K_{\modes}^h:\Bsc^a_h\to \Rbb$ be
\eq{
K_\modes^h(\betab^h) = \frac{ \Ic^h\left[\betab^h, \betab\au_\modes,\delta\gammab\right]}{\eta},
\label{eq:sif_discrete-discrete}
}
we expect $\lim_{h\searrow 0 }\betab^h =\betab^e$ and $\lim_{h\searrow
  0 } K^h_\modes(\betab^h)=K_\modes(\betab^e)$. These ideas are
compactly shown in the following commutative diagram:
\begin{equation*}
\centering
\begin{tikzpicture}[descr/.style={fill=white,inner sep=2.5pt}]
  \matrix (m) [matrix of math nodes, row sep=3em,
  column sep=3em]
  { \betab^h & K^h_{\modes} \\
     \betab^e & K_{\modes} \\ };
  \path[->,font=\scriptsize]
  (m-1-1) edge node[auto] {$ \Ic^h $} (m-1-2) ;
  \path[->,font=\scriptsize]
	(m-1-1) edge node[left] {$ h \searrow 0 $} (m-2-1);
  \path[->,font=\scriptsize]
	(m-1-2) edge node[right] {$ h \searrow 0 $} (m-2-2);
  \path[->,font=\scriptsize]
  (m-2-1) edge node[auto] {$ \Ic $} (m-2-2) ;		
\end{tikzpicture}
\end{equation*}


\label{subsxn:discrete_interinta}
 \ic{The
approximate domain and crack faces are denoted by $\Bc^h$ and $\Csc^h_\pm$
respectively, while we define for convenience $\Bc^h_\Csc = \Bc^h \setminus
\Csc_\pm^h$.} 

The functional $\Ic^h$ is defined as
\begin{equation} 
\begin{aligned}
 \Ic^h\left[\betab^h, \betab\au,\delta \gammab\right] = &  \sum_{g \in
\Gc_\Csc }  \left.\delta \gammab \cdot
\overline\taub^h\left(\betab^h, \betab\au\right) \right|_{
\x_g }  w_g & \\ & - \sum_{g \in \Gc } \left.\left[
\overline\Sigmab\left(\betab^h,\betab\au\right):\grad \delta \gammab  +
\overline\lambdab\left(\betab^h, \betab\au\right)\cdot \delta \gammab
\right]\right|_{ \x_g} w_g, \label{eq:IhG}
\end{aligned} 
\end{equation} 
where $\Gc$ and $\Gc_\Csc$ denote the set of quadrature points over
$\Bc_\Csc^h$ and $\Csc_\pm^h\cap B_\rho(\xt)$, respectively.  Each integration
point $g$ in $\Gc$ or $\Gc_\Csc$ has position $\x_g$ and integration
weight $w_g$. We assumed that all quadrature points over $\Bc_\Csc^h$
belong to $\Bc_\Csc^h\cap\Bc_\Csc$ which is true for a small enough mesh size, to be able to evaluate
$\betab\au$, which is defined over $\Bc_\Csc$. Additionally, we defined
$\overline\taub^h$  as
an approximation to
$\overline\taub$ given by 
\[
\overline 
\taub^h\left(\betab^h, \betab\au\right) = \betab^h:\cs(\betab\au \circ  \mathfrak{p} )
\nb\circ \mathfrak{p} -
\betab^h\,^\top \cs( \betab\au \circ  \mathfrak{p} )\nb \circ \mathfrak{p} - \left( \betab\au\,^\top \bar
\tb \right)\circ \mathfrak{p},
\]
where $\mathfrak{p}:\Csc_\pm^h\mapsto \Csc_\pm $ is the constant-radius projection of a point onto the crack:
\begin{equation}\label{eq:pi}
 \mathfrak{p}(\x) := \Gamma(|\x-\x_t|).
\end{equation}
This projection is well defined when $\Csc_\pm^h$ and  $\Csc_\pm$
are close enough. Other projections are possible as well. This one is
convenient, since it is also involved
in the definition of $\betab^\text{TF}$. 


\begin{rmk}[{Appearance} of $\mathfrak{p}$ in the boundary integral]
  The functional (\ref{eq:IhG}) is an approximation to integrals over
  $\Bc^h_{\Csc}$ and $\Csc_\pm^h$, and the quadrature points of $\Gc$
  and $\Gc_\Csc$ belong to them.  The \testfunctionname and
  auxiliary fields are constructed over the exact domain $\Bc_{\Csc}$,
  and could be tangent or traction free  to 
  its boundary $\Csc_\pm$, respectively, but not necessarily to its approximation $\Csc_\pm^h$. Furthermore, the traction $\bar\tb$ is prescribed and only known
  over the exact crack $\Csc_\pm$.  To address this difficulty the auxiliary traction
  $\cs(\betab\au)\nb$, the auxiliary fields  $\betab\au$ themselves, and
  the applied traction $\overline\tb$ are evaluated at their constant
  radius projection onto the curved crack  $ \mathfrak{p}(\x)$
  for $\x\in\Csc_\pm^h \cap B_\rho(\xt)$.
 Note as well that $\delta \gammab \circ \mathfrak{p}=
  \delta \gammab$, since $\delta \gammab$ depends only on $r$.

  Figure \ref{fig:bdprojection} shows an example for piecewise linear
  interpolations of the exact geometry along with its discrete
  approximation and the mapping $ \mathfrak{p}$ in \eqref{eq:pi}. As
  the mesh is refined, the difference between $\mathfrak{p}(\x)$ and
  $\x$ should go to zero, and the Jacobian of the mapping $
  \mathfrak{p}:\Csc_\pm^h \to \Csc_\pm$ should be very close to unity, thus
  permitting the composition in the boundary integral without
  introducing significant errors.

\begin{figure}[htbp]
\centering 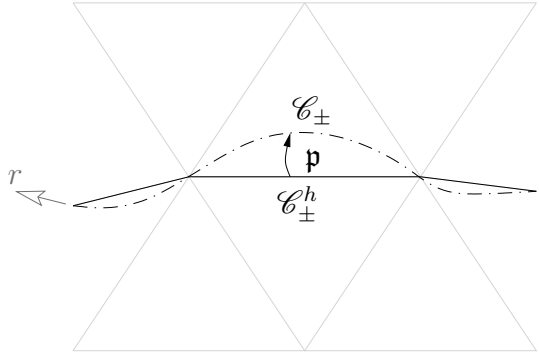 
\caption{Example of constant radius projection for piecewise linear
interpolations of $\Csc_\pm$.}
\label{fig:bdprojection} 
\end{figure}
\end{rmk}

\begin{rmk}[Convergence of the singular boundary integral] Recall that if the
applied crack-face traction is bounded at the origin, we expect the boundary
integral of \eqref{eq:interinta} to possess a singularity at the crack
tip as $ \betab\au \propto r^{-1/2} $ for $r\to0$. Integrating a singular function
using standard Gaussian quadrature over a successively refined discretization
was observed experimentally to lead to errors of the order $\Oc(h^{1/2} )$ (see \cite[Appendix B]{shen2009}).
Therefore, in the particular case in which $\overline \tb$ is bounded and non-zero
at $\xt$, it is necessary to address the numerical integration of the singular
function in order to preserve the expected convergence rate. 
Here we computed the singular integral\[
\int_{ \Csc^\pm_\rho} \delta \gammab \cdot \betab\au\,^\top \bar
\tb \;dS = \int_{0}^{\rho} \left(\delta \gammab \cdot \betab\au\,^\top \bar
\tb\right)_{\pm}(r) |\Gamma'(r) |\;dr
\] 
by  pulling back the integrand $\left(\delta \gammab \cdot \betab\au\,^\top \bar
\tb\right)_{\pm}(r) |\Gamma'(r) |$ from $[0,\rho]$ \added[id=mc]{to $\mathfrak{s}^{-1}([0, \rho])$ through the map $\mathfrak{s}( \tilde r ) =  ( \tilde r^2/\rho -\tilde r )\, q( \tilde r )  + \tilde r $ with $q$ of ~\eqref{eq:cutoff}. Then we simply use the quadrature rule $\Gc_\Csc$ over $\mathfrak{s}^{-1}([0, \rho])$.  Namely, if we let $r_g := | \x_g - \x_t |$ for all $g \in \Gc_\Csc$, we compute the above integral as} 
\[
\int_{ \Csc^\pm_\rho} \delta \gammab \cdot \betab\au\,^\top \bar
\tb \;dS  \approx \sum_{ g \in \Gc_\Csc }  \delta \gammab \cdot \betab\au\,^\top \bar
\tb \,\,  \bigg|_{ \mathfrak{s}( r_g) } \,  |\mathfrak{s}'(r_g) | \, w_g \,.
\]
The mapping $\mathfrak{s}$ effectively performs a local change of variable $r \mapsto r^2$ which removes the $\sqrt{r}$ singularity of the integrand thus allowing to recover optimal rates of convergence. The scaling of the $1/\rho$ in $\mathfrak{s}$ serves to ensure that the mapping $\mathfrak{s}$ is injective over $[0,\rho]$.
\deleted[id=mc]{
 from $[0,\rho]$ to $[0,\sqrt \rho]$ through the map  $r(\tilde r) =
\tilde r^2$, which removes the singularity and renders  a smooth integrand{Here the subscript
$\pm$ indicates the side of the crack on which the integrand is
evaluated.}. Thus, standard integration rules can be adopted on this new
domain.
}
\deleted[id=mc]{For the examples in this manuscript we adopted  a Gauss
quadrature rule with 23 sampling points, COMPLETE
finding a curve $\tilde{\Csc}^\pm_\rho$ such that with a singular map $s: r\mapsto r^2$, $s(\tilde{\Csc}^\pm_\rho)=\Csc^\pm_\rho$, then rewriting this integral over $\tilde{\Csc}^\pm_\rho$, which has a bounded smooth integrand, and finally employing
{mapping $\Csc^\pm_\rho$ in a singular way ($r\mapsto r^2$) to render a bounded smooth integrand, then employ} a {single} Gauss quadrature rule with 23 sampling points \added[id=ys]{for this smooth integrand}.This rule is an overshoot to render a small enough consistency error for the purpose of obtaining the convergence curves.
The rest of the terms were calculated with regular (in relation to the order of the finite element interpolant)  Gauss quadrature rules for each edge. We remark that the domain of the above integrand is $\Csc^\pm_{\rho}$ not $\Csc^h_\pm \cap B_{\rho}(\xt)$ as all of its terms are independent of the approximate solution.}

\todo{MC: Yongxing, could you verify that what I am stating above is reasonable and expand on it or modify if you feel necessary. YS: After our discussion, I still have a hard time understanding the need of using such a high order but $h$-independent way to compute this integral. We might appear overkilling such a mildly singular integral.
What I would do is one of the following: (1) Scale $\Csc^\pm_\rho$ so
that there is no singularity, and then use a piecewise regular Gauss
rule on the scaled edges; (2) Use a high order rule only for the edge
that contains the tip, and regular Gauss rule for all remaining
edges. (1) is more accurate but (2) is more practical. We can Skype
chat again.
AL: {\color{blue}I like the paragraph Yongxing built for this. It sounds very
reasonable. What is not clear is where is that the 23 Gauss points are
being used? Are they only used on the edge that contains the crack
tip? If so, it is not clear from the statement above. }
YS: No. Please see whether it is clear now without my explaining it here.
MC: {\color{orange} I changed the mapping to be scaled by the cutoff function in the above so that the domain of integration remains the same (i.e. $\mathfrak{s}(\rho) = \rho$). This allows to use the same gauss quadrature constructed for the other terms in the boundary integral thus providing a link to the discretization of domain. I'd be happy to Skype to talk more about it. Note that both here in the numerical calculation of the traction contribution and in \eqref{eq:IhG} we disregard the Jacobian of the constant radius projection in the integration assuming that is close to $1$. } }
\end{rmk}


\subsection{Summary of the method}
\label{sec:summary-method}

We provide here a very concise summary of the method for the reader seeking a guideline for a rapid implementation. 

The calculation of the stress intensity factors can be summarized in the following steps:
\begin{enumerate}[{\it Step 1:}] 
\item Compute the approximation $\betab^h$  to the gradient of the solution of \eqref{eq:elasticity}.
\item Construct $\delta \gammab $ to be either $\delta\gammab^\text{UNI}$ or $\delta \gammab^\text{TAN}$ from \eqref{eq:dguni} or \eqref{eq:dgtang} respectively.
\item Construct $\betab\au_\modes $ to be either $\betab^\text{TF}_\modes $ or $\betab^\text{DFC}_\modes$ from \eqref{eq:betabtf} or \eqref{eq:betabdcf} respectively.
\item With the pair $\delta \gammab ,\;  \betab\au_\modes $ as well as $\bb,\;\overline\tb$ and $\betab^h$ use \eqref{eq:IhG} to compute the value of $\Ic^h$. 
\item Compute the value of $K^h_\modes$ with the above $\Ic^h$
  following \eqref{eq:sif_discrete} (or
  \eqref{eq:sif_discrete-discrete} if appropriate).
\end{enumerate}

We recapitulate int Table \ref{table:simple} the simplifications of each integrand associated with each choice of pairing of $\betab\au$ and $\delta\gammab$.

\renewcommand{\arraystretch}{1.5}
\begin{table}[H]
\centering
\caption{Recapitulation of simplifications associated with the choice of fields. Omission of terms (--) stands for no simplification.}\label{table:simple}
\begin{tabular}{ l | c  c c }
\toprule
\hline
Fields & $\Ic_1 \, ( \Csc^\pm_\rho )$  & $\Ic_2\,(\Bc_\Csc)$  & $\Ic_3\,(\Bc_\Csc)$\\
\hline
$\delta\gammab^\text{UNI} , \betab^\text{DFC}$ & -- & $=0, \forall r < \rho_I$ & $\delta\gammab^\text{UNI} \cdot \betab^\text{DFC}\,^\top \bb$ \\
$\delta\gammab^\text{TAN} , \betab^\text{DFC}$ & $\delta\gammab\cdot[  \betab^\text{TF} \,^\top \overline \tb- \betab^h\,^\top \cs( \betab^\text{DFC})\nb]$ & -- &$ \delta\gammab^\text{TAN} \cdot \betab^\text{DFC}\,^\top \bb  $\\
$\delta\gammab^\text{TAN} , \betab^\text{TF}$ & $\delta\gammab\cdot \betab^\text{TF} \,^\top \overline \tb $ & -- &  -- \\
\hline
\bottomrule
\end{tabular}
\end{table}
\renewcommand{\arraystretch}{1}


\section{Numerical Examples} \label{sxn:numerical_examples}

\def\tables{1}

We next verify the proposed method through two examples. For each we provide
comparisons with analytical solutions. The first problem is concerned with a
circular arc crack in an infinite medium subjected to far-field stresses. The
second problem involves a power function crack subjected to crack face
tractions and body forces. 

For each example we compare the convergence of the stress intensity factors for
lower order methods, namely traditional continuous Galerkin finite element
methods for piecewise polynomial shape functions $P^k$, $k=1,2$, and
for the higher order discontinuous Galerkin extended finite element method (DG-XFEM)
\cite{shen2009}. Both methods are recapitulated in \S \ref{subsxn:numerical_elasticity_solution}.  

As discussed in \S  \ref{sxn:computation}, the interaction integral,
and hence the stress intensity factors, are expected to converge at twice the
rate of the derivatives of the solution. Thus we are expecting to observe
convergence of the order $\Oc(h)$ for lower order methods (whose derivatives
converge as $\Oc(h^{0.5})$) and $\Oc(h^2)$ for the higher-order DG-XFEM method
(whose derivatives converge as $\Oc(h^{1})$), where $h$ is the maximum
diameter of a triangle in each mesh in the family of meshes under consideration. 

In the following examples we will provide systematic convergence curves of the
error in the solution and in the computation of the stress intensity factors.
\ifnum\tables=1Tabulated errors and computed convergence rates will accompany
the above.\fi

We will present two error measures of the solution, one over the interior of
the domain and the other over the crack faces. The error in the solutions over
the interior of the domain will be measured as the $L^2$-norm of the error in
the gradient of displacements over $\Bc_{\Csc}$, and that over the crack faces will be measured as the
$L^2$-norm of the error in the gradient of displacement weighted by $r$ over
$\Csc^h_{\pm}$. Namely, with $\betab^e$ denoting  the
analytical solution of the gradient of displacement fields, we will consider as error measures
\eq{ 
\left\|\betab^h - \betab^e\right\|_{L^2(\Bc_\Csc)} = \bigg[\int_{\Bc_{\Csc} }
		  \textstyle\sum_{i,j} \left(\beta_{ij}^h -\beta_{ij}^e \right)^2 \;dV \bigg]^{1/2}
\label{eq:error_interior} } 
and 
\eq{ 
\left\|\betab^h - \betab^e\circ \mathfrak{p} \right\|_{L^2\left(\Csc^h_\pm, r\right)} = 
\left[\int_{\Csc_\pm^h } \textstyle\sum_{i,j} \left(\beta_{ij}^h -\beta_{ij}^e
\circ  \mathfrak{p} \right)^2\, r \; dS \right]^{1/2} = \left\|\left(\betab^h -
\betab^e\circ \mathfrak{p}\right)\sqrt{r} \right\|_{L^2\left(\Csc^h_\pm\right)}
\label{eq:error_boundary} . 
} In \eqref{eq:error_boundary} the
analytical gradient of displacements are evaluated at their constant radius projection onto the
exact geometry as discussed earlier in \S \ref{subsxn:discrete_interinta}. 

\sorted{
\todo{AL: Maurizio, please change all references to stresses above for strains
MC: done.}
}
\begin{rmk}[{Appearance} of $\sqrt{r}$ in the $L^2(\Csc^h_\pm)$ norm]
The appearance of the $\sqrt{r}$ factor is related to the scaling of
the factors that
multiply  $\betab^h$
in the integrand of $\Ic_1$. Namely $\betab^h$ appears as 
$	\betab^h\,^\top \cs(\betab\au)\nb $  and $	\betab^h: \cs( \betab\au)
\delta\gammab\cdot\nb $.  In the former we have 
$\cs(\betab\au)\nb\sim
r$ as  $r\to
0$, as previously discussed in section \S
\ref{subsxn:interaction_integrals}. In the latter we need to consider the scaling of
$\delta\gammab\cdot\nb$, which is either $\delta\gammab^\text{TAN}\cdot
{\bf n}=0$ or 
$\delta\gammab^\text{UNI}\cdot\nb \sim r$, as well as the scaling  $\cs(\betab\au)\sim 1/\sqrt{r}$, as
$r \to 0$. Hence $\betab^h$ in the latter case is multiplied by a
factor that scales as $\sqrt{r}$ as $r\to 0$. \added{Thus, only the rate of convergence of
  $\sqrt{r} \betab^h$ is needed to evaluate
  the rate of convergence of $\Ic_1$ .}
  
\end{rmk}
\sorted{
\todo{AL: The space on the boundary in which we are measuring the
  convergence changes from $h$ to $h$. Why did we compute the errors
  in that way? The ``more correct'' thing to do would have been to
  integrate over $\Csc_\pm$.
  
  MC: Following our Skype chat we decided to keep as is. 
  }
  }

The error in the stress intensity factors will be \deleted[id=mc]{simply} measured by the
normalized absolute value of the error in the computed stress intensity
factors. Namely, let $K_\modes^h := K^h_\modes\replaced[id=mc]{[\betab^h]}{( \betab^h)}$ be computed with
\eqref{eq:sif_discrete} (or  \eqref{eq:sif_discrete-discrete}) and $K^e_\modes$ be the exact (analytical) stress
intensity factors. We will be concerned with the behavior of 
{\begin{equation*}
\frac{\left|K_\modes^h - K_\modes^e\right|}{\left|K_\modes^e\right|}.
\end{equation*}}

We will also present for each example the value of the computed stress
intensity factors for various values of $\rho$, that is, for different
supports for $\delta \gammab$. As the interaction integral in
\eqref{eq:interinta} is independent of $\rho$, we would like to test
the independence of the computed stress intensity factors on the
support of $\delta \gammab$.

Lastly we remark that for each example we set the material constants
to $\lambda = 277.7\overline7, \mu=2500 $ ($E=1000,\nu = 0.2$) and we
assumed a plane strain state.

\subsection{Numerical Solution of the Elasticity Problem}
\label{subsxn:numerical_elasticity_solution}

We consider two types of finite element methods over a family of
meshes of triangles. In the following, the superscript $(\cdot)^h$
will denote quantities associated with the discrete approximation of
the problem. For each mesh in the family, the domain $\Bc$ is
approximated by $\Bc^h = \overline{\bigcup_{e} T^e}$, the collection
of open, straight triangles $T^e$.  \added[id=mc]{Let $\mathbb V$ denote the set of all vertices in the mesh.} Each mesh in the family conforms to the crack, namely, a
node sits at the crack tip, and there is no edge with its two vertices on different sides of the
crack. \replaced[id=mc]{To
handle the displacement discontinuity across the crack, vertices that lie on $\Csc^\pm$ are duplicated, and so are
edges whose two vertices lie on $\Csc^\pm$. }{ Vertices that lie on $\Csc^\pm$ are duplicated and 
grouped into the set $\mathbb V$, and so are
edges whose two vertices are in $\mathbb V$, to
handle the displacement discontinuity across the crack.}  The union of
these edges on either side of the crack forms the piecewise linear
approximation $\Csc_\pm^h$ to
 $\Csc_\pm$, and we set
$\Bc^h_{\Csc} = \Bc^h\backslash\Csc_\pm^h$.   For convenience, we define
\begin{equation*} {\mathbb V}^d = \{ a \in {\mathbb V} |\x_a \in \partial_d \Bc \}, 
\end{equation*} where $\x_a$ represents the position vector of vertex $a$. 
\sorted{
\todo{AL: It seems to me that $\partial {\cal B}^h=\partial {\cal B}$
  in all of our examples, true? In that case, I can remove the
  statement of boundaries being approximated.

If this is not the case, then I need to explain how we apply boundary
conditions on Dirichlet boundaries in the $k=2$ case

MC: I removed the ``The boundary $\partial \Bc$ is
approximated by $\partial \Bc^h$.'' comment.
}

\todo{AL: Do you think that there is a risk of confusing and element
  $K^e$ with a SIF? Also, why the notation $K^e$, and not simply $K$?
  
  MC: I replaced $K^e$ with $T^e$.}
}
In the following examples, we let $N_a \in H^1(\Bc^h_\Csc) $ be the
$P^k$ shape function associated with node
$a\in{\mathbb V}$, $k=1,2$, such that $N_a(\x_b)=\delta_{ab}$ for all
$a, b\in{\mathbb V}$, where $\delta_{ab}$ is the Kronecker delta. Of
course, for the piecewise quadratic case $k=2$, mid-edge nodes are
added .

The two methods adopted here are:


\begin{enumerate}[(1)]
\item \emph{Standard finite element method.} We seek an
approximate solution $\ub^h \in \Sc^h$, with \begin{equation*} \Sc^h = \left\{
\textstyle\sum_{a \in {\mathbb V} } N_a \ub_a = \ub^h  \in {H^1\left(\Bc^h_\Csc;\Rbb^2\right)
} \middle|  \ub^h( \x_a ) = \overline \ub( \x_a ), \forall a \in {\mathbb V}^d
\right\}.  \end{equation*} We further let \begin{equation*} \Vc^h = \left\{
\textstyle\sum_{a \in {\mathbb V} } N_a \delta\ub_a = \delta \ub^h  \in
{H^1\left(\Bc^h_\Csc;\Rbb^2\right)
} \middle|   \delta \ub^h( \x_a ) =0,
\forall a \in {\mathbb V}^d    \right\}.  \end{equation*}

The numerical approximation of $\ub$ is obtained by finding $\ub^h \in \Sc^h$
such that 
\begin{equation*}
\int_{\Bc^h_{\Csc}}   \cs\left( \betab^h \right):\grad \delta \ub^h \;dV= 
\int_{\Bc^h_{\Csc}}\bb \cdot \delta \ub^h\;dV +
\int_{\partial_\tau \Bc \cup \Csc^h_\pm } \bar \tb \cdot \delta \ub^h \; dS,
\quad \forall \delta \ub^h\in \Vc^h,
\end{equation*}
where 
\begin{equation*}
\betab^h = \grad\ub^h=\sum_{a \in {\mathbb V}}  \ub_a \otimes \grad N_a.
\end{equation*}

\item \emph{Discontinuous-Galerkin extended finite element method.}
Here we recapitulate the method proposed in \cite{shen2009} with slight improvements. Let $h<(1/2)[\dist(\xt, \partial\Bc)-\rho]$ and $r_c$ be such that 
\[\rho+h<r_c<\dist(\xt, \partial\Bc)-h,\] and 
\begin{equation*}
	\Bc^E_h = \overline{\bigcup\left\{T^e\middle| \area[T^e\cap B_{r_c}(\xt)]>0 \right\}},\quad
	\Bc^U_h = \overline{\Bc^h\setminus\Bc^E_h},
\end{equation*}
be the enriched and unenriched regions, respectively. Then we set
\begin{equation*}
	{\mathbb V}^E=\left\{a\in{\mathbb V}\middle|\x_a\in\Bc^E_h\right\},\quad
	{\mathbb V}^U=\left\{a\in{\mathbb V}\middle|\x_a\in\Bc^U_h\right\}.
\end{equation*}
Hence, there are nodes \replaced[id=mc]{that belong to}{ in} both ${\mathbb V}^E$ and ${\mathbb V}^U$. In fact, let $\Gamma^E_h=\partial\Bc^E_h$, then 
\begin{equation*}
	{\mathbb V}^E\cap{\mathbb V}^U = \left\{a\in{\mathbb V}\middle|\x_a\in\Gamma^E_h\right\}.
\end{equation*}

The discontinuous Galerkin extended finite element method (DG-XFEM) is built on the following set:
\begin{equation*}
\begin{aligned}
	\Sc^h = \left\{\ub^h\in {L^2\left(\Bc^h_\Csc;\Rbb^2\right)}\middle|
	\ub^h=k_I\ub^I + k_{II}\ub^{II} +\textstyle \sum_{a\in{\mathbb V}^E} N_a\ub_a^E \text{ in } \Bc^E_h,
	k_I, k_{II} \in \Rbb;\right. \\ \left.
	\ub^h=\textstyle\sum_{a\in{\mathbb V}^U} N_a\ub_a^U\text{ in } \Bc^U_h,
  \ub^h( \x_a ) = \overline \ub( \x_a ), \forall a \in {\mathbb V}^d
	\right\}.
	\end{aligned}
\end{equation*}
The corresponding test space is given by:
\begin{equation*}
\begin{aligned}
	\Vc^h = \left\{\delta\ub^h\in {L^2\left(\Bc^h_\Csc;\Rbb^2\right)}\middle|
	\delta\ub^h=\delta k_I\ub^I + \delta k_{II}\ub^{II} + \textstyle\sum_{a\in{\mathbb V}^E} N_a\delta\ub_a^E \text{ in } \Bc^E_h,
	\delta k_I, \delta k_{II} \in \Rbb; \right. \\\left.
	\delta\ub^h=\textstyle\sum_{a\in{\mathbb V}^U} N_a\delta\ub_a^U\text{ in } \Bc^U_h,
  \delta\ub^h( \x_a ) = \0, \forall a \in {\mathbb V}^d
	\right\}.
	\end{aligned}
\end{equation*}
Therefore, the kinematics of a typical function $\ub^h\in\Sc^h$ is independent in $\Bc^E_h$ and $\Bc^U_h$; 
a discontinuity across $\Gamma_h^E$ arises which is defined as
\begin{equation*}
	\left\llbracket\ub^h\right\rrbracket = \left[\left.\ub^h\right|_{\Bc^E_h}\right]_{\Gamma^E_h} 
	- \left[\left.\ub^h\right|_{\Bc^U_h}\right]_{\Gamma^E_h}
	=\left[k_I\ub^I + k_{II}\ub^{II} +\textstyle \sum_{a\in{\mathbb V}^E\cap{\mathbb V}^U} N_a\left(\ub_a^E-\ub_a^U\right)\right]_{\Gamma^E_h}.
\end{equation*}
\sorted{
\todo{AL: Yongxing, How is the discontinuity appearing? we are using the same
  notation for the degree of freedom of the same node on
  $\Gamma_h^E$. I think that we need to either duplicate the nodes, or
give alternative names to the degrees of freedom on each side of the
interface $\Gamma_h^E$. YS: Fixed.}
}
This discontinuous $\ub^h$ is handled through a DG-derivative $D_{DG}:\Sc^h\rightarrow\grad_h\Sc^h+\Wc^h$:
\begin{equation*}
	D_{DG}: \ub^h\mapsto \grad_h\ub^h + R(\llbracket\ub^h\rrbracket),
\end{equation*}
where $\grad_h\ub^h=\grad\ub^h$ in each $T^e$, $R(\llbracket\ub^h\rrbracket)$ is such that
\begin{equation*}
	\int_{\Bc_\Csc^h} R\left(\left\llbracket\ub^h\right\rrbracket\right): \wb^h \; dV
	= -\int_{\Gamma^E_h} \left\llbracket\ub^h\right\rrbracket\otimes\nb : \left\{\wb^h\right\} \; dS, \quad \forall \wb^h \in \Wc^h,
\end{equation*}
where
\begin{equation*}
	\Wc^h=\prod_e \left[P_1(T^e)\right]^{2\times2},
\end{equation*}
and on $\Gamma_h^E$
\begin{equation*}
	\left\{\wb^h\right\}=\frac12\left(\left[\left.\wb^h\right|_{\Bc^E_h}\right]_{\Gamma^E_h} 
	+ \left[\left.\wb^h\right|_{\Bc^U_h}\right]_{\Gamma^E_h}\right).
\end{equation*}

The solution to the problem stated in \S\ref{subsxn:elasticity_problem} is approximated by:
Find $\ub^h\in\Vc^h$ such that
\begin{equation*}
\begin{aligned}
\int_{\Bc^h_{\Csc}}   \cs\left( \betab^h \right):D_{DG} \delta \ub^h \;dV
+ 2\mu \alpha  \int_{\Bc^h_{\Csc}} R\left(\left\llbracket\ub^h\right\rrbracket\right):R\left(\left\llbracket\delta\ub^h\right\rrbracket\right) \;dV
\\= 
\int_{\Bc^h_{\Csc}}\bb \cdot \delta \ub^h\;dV +
\int_{\partial_\tau \Bc \cup \Csc^h_\pm } \bar \tb \cdot \delta \ub^h \; dS,
 \quad \forall \delta \ub^h\in \Vc^h,
\end{aligned}
\end{equation*}
where $\alpha$ can be any positive real number, and
\begin{equation*}
\betab^h = D_{DG} \ub^h.
\end{equation*}
\sorted{
\todo{AL: Yongxing, and the stabilization term? Also, what are the
  improvements over [2] that you mentioned? YS: Added the stabilization term. The improvement is the definition of $\Wc^h$, which contains only $P_1$ functions. That being said, I am not sure whether to indicate what the improvement is.}
  }
\end{enumerate}

We conclude the section by remarking that the approximate domain of integration
of the interaction integral for the particular choice of the method is given by 
the subset of elements with at least one vertex that lies within
$B_\rho(\xt)$ which we denote by $\mathbb K$. Refer to Fig. \ref{fig:submesh} for
an illustration of the above. 

\begin{figure}[htbp]
\centering 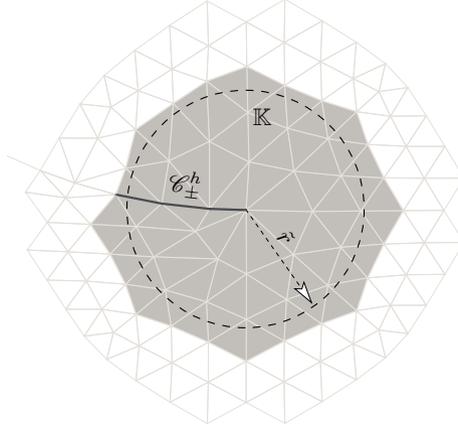 
\caption{ Example of a subset of
elements ${\mathbb K}$ in a finite element mesh. The elements in the shaded region are
the elements over which the interaction integral is computed. }
\label{fig:submesh} \end{figure}

Furthermore, we exploit the quadrature rule constructed over each element
$\Gc_{T^e}$ and its boundary $\Gc_{\partial T^e} $ to form $\Gc$ and $\Gc_\Csc$, respectively.  
Namely we
let the numerical interaction integral, in this specific setting of finite
element methods, become{
 \begin{equation*} 
\begin{aligned}
\Ic^h\left[\betab^h, \betab\au,\delta \gammab\right] \approx & \sum_{T^e \in {\mathbb K}, 
|\partial T^e\cap\Csc_\pm^h|>0} \sum_{\;g \in \Gc_{ \partial T^e} }  \delta \gammab \cdot
\overline\taub^h\left(\betab^h, \betab\au\right) \big|_{
\x_g }  w_g & \\ & -\sum_{T^e \in {\mathbb K}} \sum_{g \in \Gc_{T^e} } \left[
\overline\Sigmab\left(\betab^h,\betab\au\right):\grad \delta \gammab  +
\overline\lambdab\left(\betab^h, \betab\au\right)\cdot \delta \gammab
\right]\big|_{ \x_g} w_g. & 
\end{aligned} \end{equation*}}

\subsection{Circular Arc Crack}
\newcommand{\example}{CircularArcCrack}
\newcommand{\directory}{Figures/CircularArcCrack} 


We consider an infinite plate with a circular arc shaped crack subjected  to uniform tension from infinity. The analytical solution was derived in
\cite{muskhelishvili1977}, and a recapitulation of the  solution can be found
in \cite{rangarajan2014}. The resulting stress intensity factors
for uniform far field tension loading are (see, e.g., \cite{cotterell1989})
\begin{equation*} \begin{aligned} K^e_I = \dfrac{ \sigma \cos \left(\dfrac{\alpha
}{2}\right) \sqrt{\pi  R \sin (\alpha ) }}{\sin ^2\left(\dfrac{\alpha
}{2}\right)+1}, \quad
K^e_{II} = \dfrac{\sigma \sin \left(\dfrac{\alpha }{2}\right) \sqrt{\pi  R \sin
(\alpha )}}{\sin ^2\left(\dfrac{\alpha }{2}\right)+1},
\end{aligned}
\end{equation*}
where $R$ is the radius of the circular arc crack, $\alpha$ is half the angle
subdued by the crack, and $\sigma$ is the far field tension as shown in Fig.
\ref{fig:circarce}. 
\sorted{
\todo{AL: What is the miniarrow in the middle of the circle in Fig. \ref{fig:circarce}? Is it
  trying to mark the center? We should put a little circle there then
  
  MC: I replaced the arrow with a circle}
  }

Only a finite subdomain was considered and exact tractions were specified on
the boundaries. Given the symmetry of the problem, only half of the subdomain
was modeled and appropriate symmetry boundary conditions on the axis
of symmetry were specified.  Figure \ref{fig:circarcbc} shows  a
representation of the modeled subdomain and boundary conditions. For the
simulations we took $\alpha = \pi/2, R = 1$, the modeled domain was
given by $\Bc = [0,2]\times [-1.5 ,0.5]$, and  the crack centered at the
origin.

To establish the accuracy of the methods, the solution was computed for
different levels of refinement of the discretized domain. The meshes were
generated by conforming recursive subdivisions of the coarsest mesh to the
exact geometry. 
\sorted{
\todo{AL: Maybe we should remove the divergence of the displacement
  plot, it does not contribute anything to the understanding, only to
  show that we may have computed the solution correctly, but we are
  showing the errors, so I would remove it. If we do, we need to
  remove the last sentence of last paragraph.
  
  MC: I agree, done.
  }
}

The error measures  \eqref{eq:error_interior} and
\eqref{eq:error_boundary} were observed to decrease as $\Oc(h^{1/2})$  for the
lower-order method, and as $\Oc( h )$ for the second-order method. Figure
\ref{fig:convergececac} shows the convergence plot of the solution, and Table
\ref{table:convergence_circular_arc_crack} summarizes the error as well as the
computed rates of convergence.
 
As expected, the error in the stress intensity factors are observed to converge
with order $\Oc(h^{1})$ and $\Oc(h^2)$ for the lower- and higher-order methods,
respectively. Figure \ref{fig:convergence_sif_circular_arc_crack} provides the convergence curves
for the stress intensity factors using the three pairings of \testfunctionname and
auxiliary fields. Errors and computed rates are reported in Table
\ref{table:convergence_sif_circular_arc_crack}.
\sorted{
\todo{AL: Whenever Figure is not the first word of a sentence, we
  should use Fig. Also, we only say Equation (xx) when it is the
  beginning of a sentence, otherwise we just say (xx)
  
  MC: I believe I took care of this. }
}

Lastly we show that the evaluation of the interaction integral is independent
of the support of $\delta \gammab$. To this end, Fig.
\ref{fig:radialindependence_circular_arc_crack} shows the error in computed
stress intensity factors of the most refined mesh for five values of $\rho/\rho_\text{max}$,
ranging from $\sim 0.7 $ to $1$ with $\rho_\text{max} = 0.5 $. The independence of the interaction integral on the choice of the
support of the \testfunctionname field is apparent from these results.


\begin{figure}[htbp]
\begin{minipage}{0.5\textwidth}
\centering
\begin{tikzpicture}[x=2.5cm,y=2.5cm]
\def\originx{0}
\def\originy{0}
\def\sizeb{1}
\def\sizep{1.35}
\def\alph{0}
\def\radius{0.75}
\draw [color=black!0!white,fill=none](-\sizep,-\sizep)-- (0,-\sizep) -- (\sizep,-\sizep)--(\sizep,\sizep)--(-\sizep,\sizep)--(-\sizep,0)  --(-\sizep,-\sizep);
\draw [color=black!10!white,thick,fill=black!10!white](-\sizeb,-\sizeb)-- (0,-\sizeb) -- (\sizeb,-\sizeb)--(\sizeb,\sizeb)--(-\sizeb,\sizeb)--(-\sizeb,0)  --(-\sizeb,-\sizeb);
\draw [color=black!0!white,dashed](-\sizep,-\sizep)-- (0,-\sizep) -- (\sizep,-\sizep)--(\sizep,\sizep)--(-\sizep,\sizep)--(-\sizep,0)  --(-\sizep,-\sizep);
\draw [black,-] ({- cos( \alph ) *\radius}  ,{-sin( \alph ) *\radius + \radius/2} ) 
    arc ({180+\alph}:{360-\alph}:\radius);
\draw [arrows={latex-latex},black] ({- cos( \alph ) *\radius/2}  ,{-sin( \alph ) *\radius/2 + \radius/2} ) 
    arc ({180+\alph}:{360-\alph}:\radius/2);
\node at (0,0) [below] { $2\alpha$};
\draw [color=black!40!white,thick,-latex]({- cos( \alph ) *\radius},{-sin( \alph ) *\radius + \radius/2 })-- (0, + \radius/2 )node[below]{} -- ({cos( \alph ) *\radius/2},{-sin( \alph ) *\radius/2+ \radius/2})node[above]{$R$}-- ({cos( \alph ) *\radius},{-sin( \alph ) *\radius+ \radius/2})node[below]{};
\draw[mark=*,color=black!40!white,mark size=2pt] plot coordinates {(0, + \radius/2)};
\node at (0,-\radius+\radius/2) [below] { $\Csc^\pm$};
\foreach \t in {-\sizeb,-0.90001,...,\sizeb}
            \draw [black!50!white, opacity=1.0, -latex, thick]
                (\t,-\sizeb)-- (\t,-\sizeb-\sizeb/7.5) ;
\foreach \t in {-\sizeb,-0.90001,...,\sizeb}
            \draw [black!50!white, opacity=1.0, -latex, thick]
                (\t,\sizeb)-- (\t,\sizeb+\sizeb/7.5) ;
\foreach \t in {-\sizeb,-0.90001,...,\sizeb}
            \draw [black!50!white, opacity=1.0, -latex, thick]
                (\sizeb,\t)-- (\sizeb+\sizeb/7.5,\t) ;
\foreach \t in {-\sizeb,-0.90001,...,\sizeb}
            \draw [black!50!white, opacity=1.0, -latex, thick]
                (-\sizeb,\t)-- (-\sizeb-\sizeb/7.5,\t) ;      
\node at ( 0 , -\sizeb-\sizeb/7.5 ) [below] {$ \cs |_{x,y\to \infty} = \sigma\1  $};
\end{tikzpicture}
\captionof{figure}{The circular arc crack problem.}
\label{fig:circarce}
\end{minipage}
\begin{minipage}{0.5\textwidth} 
\centering
\begin{tikzpicture}[x=2.5cm,y=2.5cm]
\def\originx{0}
\def\originy{0}
\def\sizeb{1}
\def\sizep{1.35}
\draw [color=black!0!white](-\sizep,-\sizep)-- (0,-\sizep) -- (\sizep,-\sizep)--(\sizep,\sizep)--(-\sizep,\sizep)--(-\sizep,0)  --(-\sizep,-\sizep);
\draw [color=black!10!white,thick,fill=black!10!white](-\sizeb,-\sizeb)-- (0,-\sizeb) -- (\sizeb,-\sizeb)--(\sizeb,\sizeb)--(-\sizeb,\sizeb)--(-\sizeb,0)  --(-\sizeb,-\sizeb);
\draw [color=black,thick,dashed] (-\sizeb,\sizeb )-- (0,\sizeb )node[above] {$ \cs\ee_y  = \overline\tb^e$ } -- (\sizeb,\sizeb ) -- (\sizeb,0) node[rotate=270,above] {$\cs\ee_x  = \overline\tb^e$ } -- (\sizeb,-\sizeb) -- (0,-\sizeb) node[below]{$-\cs\ee_y  = \overline\tb^e$}--(-\sizeb,-\sizeb);
\draw [color=black,thick,fill=black!10!white,dashed](-\sizeb,\sizeb)--(-\sizeb,0.1*\sizeb) node[above,rotate=90]{$\ub \cdot \ee_x = 0, \, \ee_y \cdot \cs\ee_x = 0 $ } --(-\sizeb,-\sizeb);
\draw [arrows={latex-latex}] (-\sizeb,-\sizeb*0.8)node[left] {$\ee_y$} -- (-\sizeb,-\sizeb)node[left]{$\ub = \0 $} -- (-\sizeb*0.8,-\sizeb)node[below] {$\ee_x$};
\draw [black,-] (-\sizeb,-0.5\sizeb) 
    arc (270:360:1*\sizeb);
\node at (0,0) { $\Csc^\pm$};
\node at (0.9*\sizeb,0.9*\sizeb) { $\Bc$};
\end{tikzpicture}
\captionof{figure}{Modeled subdomain. Here $\overline{\bf t}^e$ is the
exact traction on a face.}
\label{fig:circarcbc}
\end{minipage}
\end{figure}
\ifnum\tables=1
\begin{table}[htb]
\centering
\caption{ Convergence rates of the derivatives of the solution for the circular
arc crack problem }
\label{table:convergence_circular_arc_crack}
\subfloat[Domain convergence]{
\begin{tabular}{ l | c c | c c | c c }
\cmidrule{2-7}
\multicolumn{1}{ c  }{ } & \multicolumn{6}{c}{ $\| \betab^h - \betab^e\|_{L^2(\Bc) } $ }  \\
\cmidrule{2-7}
\multicolumn{1}{ c  }{ } & \multicolumn{2}{ c | }{ $P^1$}  & \multicolumn{2}{  c|  }{ $P^2$ } & \multicolumn{2}{  c  }{ $DG-XFEM$  }
 \\
 \cmidrule{1-7}
$h_\text{max}/h $& Err. & $\Oc$ &  Err. & $\Oc$ &  Err. & $\Oc$   \\
\hline\hline
1 & 0.00055 & -- &0.00025 & -- &0.00028 &  -- \\
2 &0.00037 & 0.57 & 0.00019 & 0.43 &0.00015 & 0.92\\
4 & 0.00026 & 0.52 & 0.00013 & 0.51 &0.00007 & 0.97\\
8 & 0.00018 & 0.51 & 0.00009 & 0.51 &0.00004 & 0.96\\
16 &0.00013 & 0.51 & 0.00006 & 0.51 &0.00002 & 0.98 \\
\hline
\end{tabular}
}

\subfloat[Trace convergence]{
\begin{tabular}{ l |  c c | c c | c c  }
\cmidrule{2-7}
\multicolumn{1}{ c  }{ } & \multicolumn{6}{c}{  $\| \betab^h - \betab^e \circ  \mathfrak{p} \|_{L^2(\Csc^h_\pm,r) }  $ } \\
\cmidrule{2-7}
\multicolumn{1}{ c  }{ } & \multicolumn{2}{ c | }{ $P^1$}  & \multicolumn{2}{  c|  }{ $P^2$ } & \multicolumn{2}{  c }{ $DG-XFEM$ }\\
 \cmidrule{1-7}
$h_\text{max}/h$& Err. & $\Oc$ &  Err. & $\Oc$ &  Err. & $\Oc$   \\
\hline\hline
1 & 0.00055 & --&0.00025 & -- &0.00028 &  -- \\
2 & 0.00037 & 0.57 &0.00019 & 0.43 &0.00015 & 0.92\\
4 & 0.00026 & 0.52 &0.00013 & 0.51 &0.00007 & 0.97\\
8 & 0.00018 & 0.51 &0.00009 & 0.51 &0.00004 & 0.96\\
16 & 0.00013 & 0.51 &0.00006 & 0.51 &0.00002 & 0.98 \\
\hline
\end{tabular}
}
\end{table}
\fi

\begin{figure}[htbp]
\centering
\subfloat[Convergece in the  $L^2( \Bc_{\Csc} ) $ norm] { 
\includegraphics[trim=0.65in 0.5in 0.5in 0.5in,clip ,
 width=0.5\textwidth]{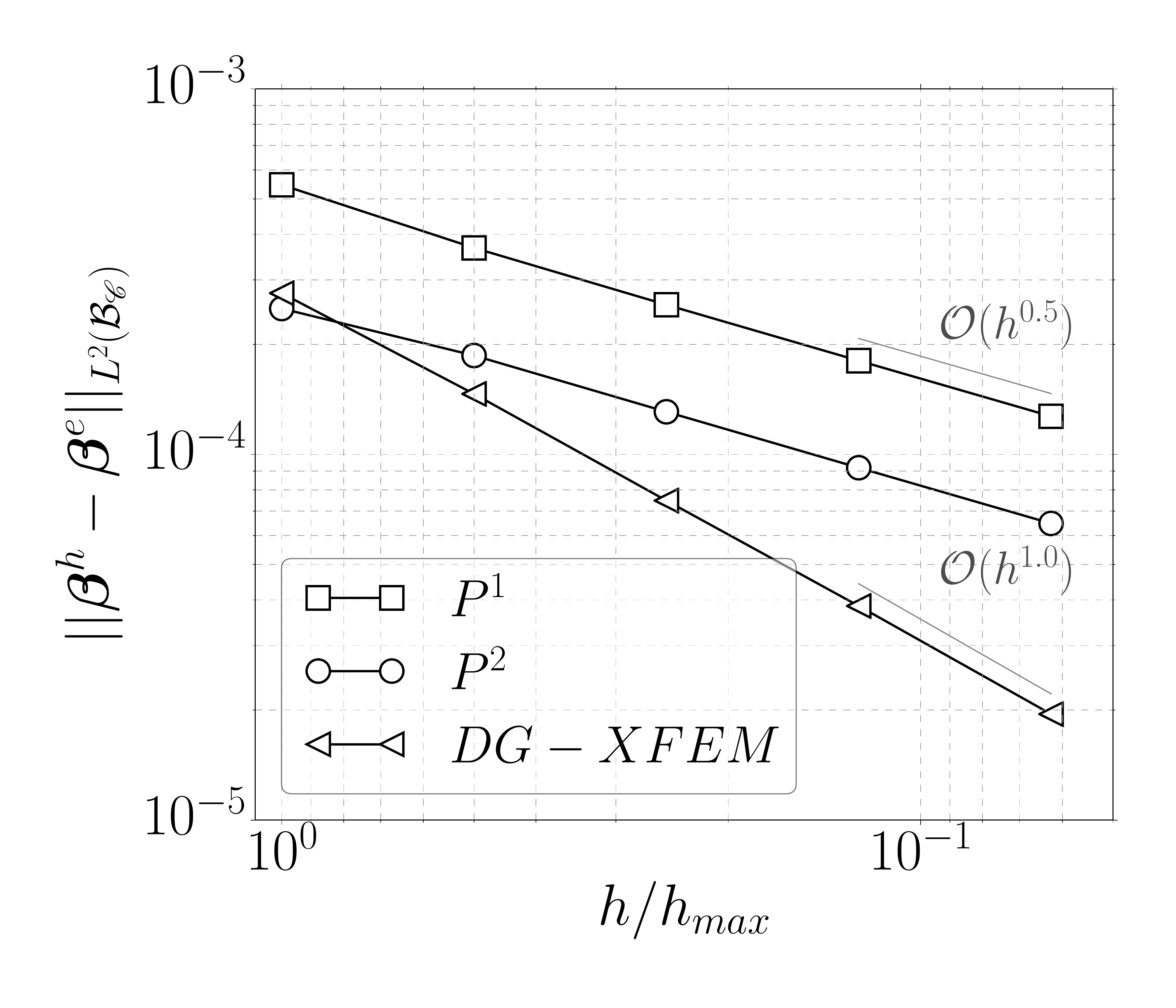}
}
\subfloat[Convergece in the  $L^2( \Csc_\pm^h,r ) $ norm] { 
\includegraphics[trim=0.65in 0.5in 0.5in 0.5in,clip ,width=0.5\textwidth]{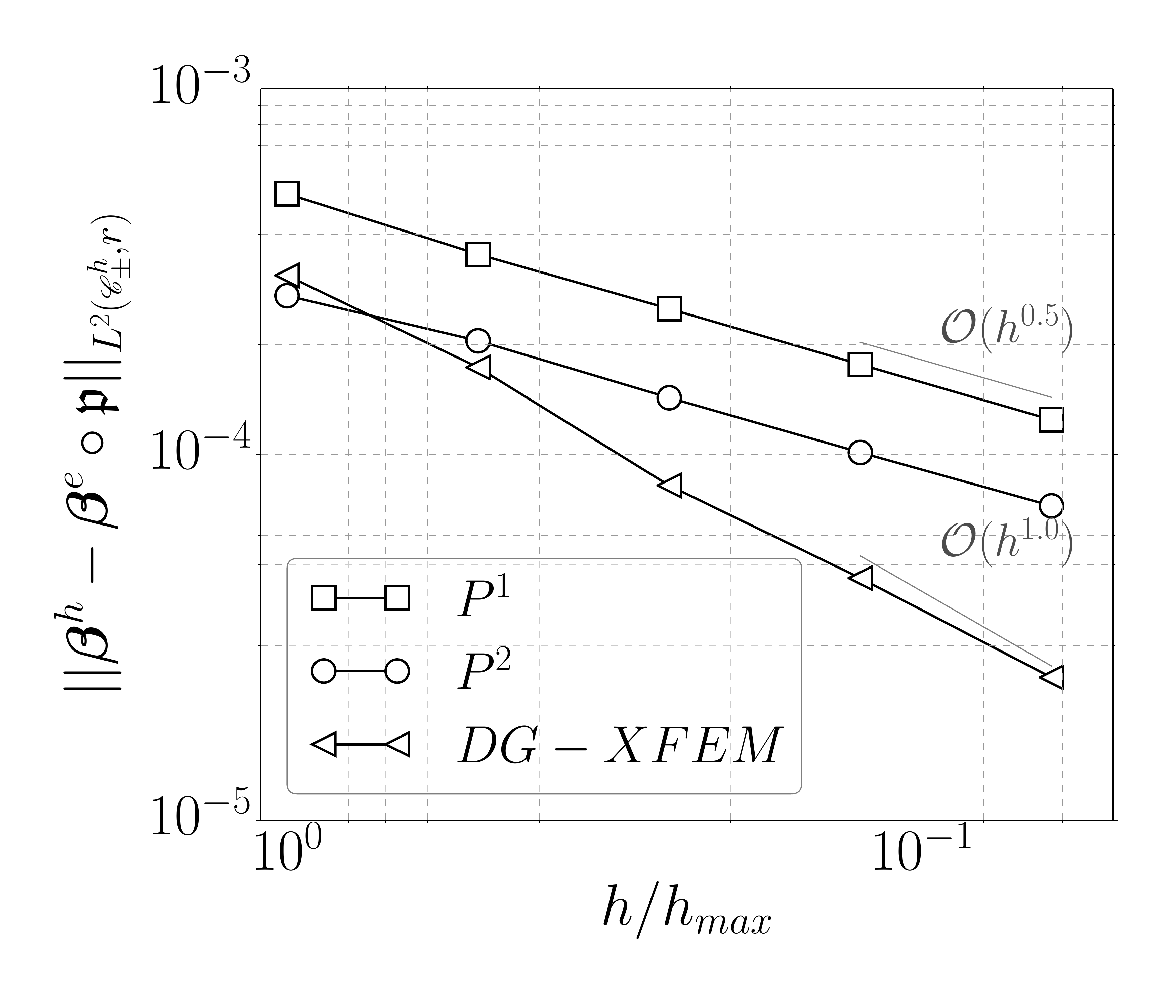} 
}
\caption{Convergence of the solution. }
\label{fig:convergececac}
\end{figure}


\ifnum\tables=1
\ifnum1=1
\begin{table}[htbp]
\centering
\caption{ Convergence rates for stress intensity factors  }
\label{table:convergence_sif_circular_arc_crack}
\subfloat[ {\color{black} Traction free auxiliary fields ($\betab\au = \betab^\text{TF}$) and tangential \testfunctionname ($\delta \gammab = \delta\gammab^\text{TAN} $)}.]{
\begin{tabular}{ l | c c | c c | c c | c c | c c | c c}
\cmidrule{2-13}
\multicolumn{1}{ c  }{ } & \multicolumn{4}{ c | }{ $P^1$}  & \multicolumn{4}{  c|  }{ $P^2$ } & \multicolumn{4}{  c }{ $DG-XFEM$ } \\
 \cmidrule{2-13}
\multicolumn{1}{ c  }{ }& \multicolumn{2}{ c | }{  $K_I$  } &\multicolumn{2}{c |}{ $K_{II } $} &  \multicolumn{2}{c|}{ $K_I$ } & \multicolumn{2}{ c |}{ $K_{II}$ } & \multicolumn{2}{ c| }{$K_I$ } & \multicolumn{2}{ c }{ $K_{II}$ }  \\
 \cmidrule{1-13}
$h_\text{max}/h $& Err. & $\Oc$ &  Err. & $\Oc$ &  Err. & $\Oc$ &  Err. & $\Oc$ &  Err. & $\Oc$ &  Err. & $\Oc$   \\
\hline\hline
1/1 &1e-01 & -- &5e-02 & -- &3e-02 & -- &7e-03 & -- &1e-02 & -- &8e-03 & -- \\
1/2 &6e-02 & 0.72 &3e-02 & 0.81 &2e-02 & 0.69 &7e-03 & 0.03 &3e-03 & 1.87 &2e-03 & 2.02 \\
1/4 &3e-02 & 0.99 &1e-02 & 1.01 &1e-02 & 1.00 &4e-03 & 0.93 &8e-04 & 1.98 &5e-04 & 1.98 \\
1/8 &2e-02 & 0.94 &7e-03 & 0.92 &5e-03 & 0.95 &2e-03 & 0.79 &2e-04 & 1.94 &1e-04 & 2.00 \\
1/16 &9e-03 & 0.98 &3e-03 & 1.00 &2e-03 & 1.04 &1e-03 & 1.04 &6e-05 & 1.67 &3e-05 & 2.25 \\

\hline
\end{tabular}
}\\
\subfloat[ {\color{black} Divergence-free ($\betab\au = \betab^\text{DFC}$) and unidirectional \testfunctionname  ($\delta \gammab = \delta\gammab^\text{UNI} $ ). }]{
\begin{tabular}{ l | c c | c c | c c | c c | c c | c c}
\cmidrule{2-13}
\multicolumn{1}{ c  }{ } & \multicolumn{4}{ c | }{ $P^1$}  & \multicolumn{4}{  c|  }{ $P^2$ } & \multicolumn{4}{  c }{ $DG-XFEM$ } \\
 \cmidrule{2-13}
\multicolumn{1}{ c  }{ }& \multicolumn{2}{ c | }{  $K_I$  } &\multicolumn{2}{c |}{ $K_{II } $} &  \multicolumn{2}{c|}{ $K_I$ } & \multicolumn{2}{ c |}{ $K_{II}$ } & \multicolumn{2}{ c| }{$K_I$ } & \multicolumn{2}{ c }{ $K_{II}$ }  \\
 \cmidrule{1-13}
$h_\text{max} /h$& Err. & $\Oc$ &  Err. & $\Oc$ &  Err. & $\Oc$ &  Err. & $\Oc$ &  Err. & $\Oc$ &  Err. & $\Oc$   \\
\hline\hline
1/1 &1e-01 & -- &4e-02 & -- &3e-02 & -- &4e-03 & -- &1e-02 & -- &7e-03 & -- \\
1/2 &7e-02 & 0.72 &2e-02 & 0.81 &2e-02 & 0.70 &5e-03 & 0.37 &3e-03 & 1.86 &1e-03 & 2.36 \\
1/4 &3e-02 & 0.99 &1e-02 & 1.02 &1e-02 & 0.99 &3e-03 & 0.87 &7e-04 & 2.00 &3e-04 & 2.35 \\
1/8 &2e-02 & 0.94 &6e-03 & 0.91 &5e-03 & 0.96 &2e-03 & 0.72 &2e-04 & 1.97 &4e-05 & 2.68 \\
1/16 &9e-03 & 0.98 &3e-03 & 0.99 &3e-03 & 1.04 &9e-04 & 1.05 &6e-05 & 1.70 &7e-06 & 2.58 \\

\hline
\end{tabular}
}\\
\subfloat[ {\color{black} Divergence-free ($\betab\au = \betab^\text{DFC}$) and tangential \testfunctionname  ($\delta \gammab = \delta\gammab^\text{TAN} $ ).}]{
\begin{tabular}{ l | c c | c c | c c | c c | c c | c c}
\cmidrule{2-13}
\multicolumn{1}{ c  }{ } & \multicolumn{4}{ c | }{ $P^1$}  & \multicolumn{4}{  c|  }{ $P^2$ } & \multicolumn{4}{  c }{ $DG-XFEM$ } \\
 \cmidrule{2-13}
\multicolumn{1}{ c  }{ }& \multicolumn{2}{ c | }{  $K_I$  } &\multicolumn{2}{c |}{ $K_{II } $} &  \multicolumn{2}{c|}{ $K_I$ } & \multicolumn{2}{ c |}{ $K_{II}$ } & \multicolumn{2}{ c| }{$K_I$ } & \multicolumn{2}{ c }{ $K_{II}$ }  \\
 \cmidrule{1-13}
$h_\text{max}/h $& Err. & $\Oc$ &  Err. & $\Oc$ &  Err. & $\Oc$ &  Err. & $\Oc$ &  Err. & $\Oc$ &  Err. & $\Oc$   \\
\hline\hline
1/1 &1e-01 & -- &4e-02 & -- &3e-02 & -- &4e-03 & -- &1e-02 & -- &7e-03 & -- \\
1/2 &6e-02 & 0.72 &2e-02 & 0.83 &2e-02 & 0.70 &5e-03 & 0.47 &3e-03 & 1.88 &2e-03 & 2.28 \\
1/4 &3e-02 & 0.99 &1e-02 & 1.02 &1e-02 & 1.00 &3e-03 & 0.90 &8e-04 & 1.98 &3e-04 & 2.32 \\
1/8 &2e-02 & 0.94 &6e-03 & 0.91 &5e-03 & 0.95 &2e-03 & 0.69 &2e-04 & 1.94 &5e-05 & 2.71 \\
1/16 &9e-03 & 0.98 &3e-03 & 0.99 &3e-03 & 1.04 &8e-04 & 1.05 &6e-05 & 1.68 &1e-05 & 2.05 \\

\hline
\end{tabular}
}
\end{table} 
\else
\begin{table}[ht] 
\centering 
\caption{Convergence of the stress intensity factors for the circular arc crack problem}
\label{table:convergence_sif_circular_arc_crack}
\foreach \method/\cap in {1/{$\betab = \betab^{DFC},\delta\gammab = \delta\gammab^{UNI}$},2/{$\betab = \betab^{DFC},\delta\gammab = \delta\gammab^{TAN}$},3/{$\betab = \betab^{TF},\delta\gammab = \delta\gammab^{TAN}$}}{
\foreach \mode in {1,2}{\subfloat[\cap -- Mode \ifnum\mode=1 I \else II \fi]{
\begin{tikzpicture}
\ifnum\mode=1%
	\begin{loglogaxis}[
				height=0.28\textheight,
				grid=minor,			
				xlabel={$h/h_{0} $},
				ylabel={$|K_I - K_I^e|/|K_I^e|$},					
				x dir=reverse,
				legend entries={$P^1$,$P^2$,$P^3$,$P^4$},
				legend columns=-1,
				legend style={at={(0.5,1.15)},anchor=north},									
				ymin=1.e-4,ymax=1.e-1
				]%
\else%
	\begin{loglogaxis}[
				height=0.28\textheight,
				grid=minor,			
				xlabel={$h/h_{0} $},		
				ylabel={$|K_{II} - K_{II}^e|/|K^e_{II}|$},		
				x dir=reverse,
				legend entries={$P^1$,$P^2$,$P^3$,$P^4$},
				legend columns=-1,
				legend style={at={(0.5,1.15)},anchor=north},									
				ymin=1.e-4,ymax=1.e-1,
				ylabel near ticks,yticklabel pos=right				
				]
\fi%
	\foreach \i [evaluate=\i as \ival using 1*\i]  in {1,2}{
		\ifnum\mode=1 
			\addplot table[x index=0,y index = 3] {Data/\example/solution_0_\i_\method.dat};						\else
			\addplot table[x index=0,y index = 4] {Data/\example/solution_0_\i_\method.dat}; 
		\fi	
	};
	\end{loglogaxis}
\end{tikzpicture}}}
} 
\label{table:convergence_sc}
\end{table}
\fi
\fi

\ifnum0=1
\begin{figure}[htbp] 
\centering
\foreach \method/\cap in {1/{$\betab = \betab^{DFC},\delta\gammab = \delta\gammab^{UNI}$},2/{$\betab = \betab^{DFC},\delta\gammab = \delta\gammab^{TAN}$},3/{$\betab = \betab^{TF},\delta\gammab = \delta\gammab^{TAN}$}}{
\foreach \mode in {1,2}{\subfloat[\cap -- Mode \ifnum\mode=1 I \else II \fi]{
\begin{tikzpicture}
\ifnum\mode=1%
	\begin{loglogaxis}[
				height=0.28\textheight,
				grid=minor,			
				xlabel={$h/h_{0} $},
				ylabel={$|K_I - K_I^e|/|K_I^e|$},					
				x dir=reverse,
				legend entries={$P^1$,$P^2$,$P^3$,$P^4$},
				legend columns=-1,
				legend style={at={(0.5,1.15)},anchor=north},									
				ymin=1.e-4,ymax=1.e-1
				]%
\else%
	\begin{loglogaxis}[
				height=0.28\textheight,
				grid=minor,			
				xlabel={$h/h_{0} $},		
				ylabel={$|K_{II} - K_{II}^e|/|K^e_{II}|$},		
				x dir=reverse,
				legend entries={$P^1$,$P^2$,$P^3$,$P^4$},
				legend columns=-1,
				legend style={at={(0.5,1.15)},anchor=north},									
				ymin=1.e-4,ymax=1.e-1,
				ylabel near ticks,yticklabel pos=right				
				]
\fi%
	\foreach \i [evaluate=\i as \ival using 1*\i]  in {1,2}{
		\ifnum\mode=1 
			\addplot table[x index=0,y index = 3] {Data/\example/solution_0_\i_\method.dat};						\else
			\addplot table[x index=0,y index = 4] {Data/\example/solution_0_\i_\method.dat}; 
		\fi	
	};
	\end{loglogaxis}
\end{tikzpicture}}}
	
}
\caption{Convergence of the stress intensity factors for the circular arc crack  }
\label{fig:convergence_sif_circular_arc_crack}
\end{figure}
\else
\begin{figure}[htbp]
\centering
\subfloat[{ \color{black} $(\betab\au , \delta\gammab)= ( \betab^\text{TF}_\text{I}, \delta\gammab^\text{TAN})$ . } ] { 
\includegraphics[trim=0.5in 0.25in 0.5in 0.55in,clip,width=0.5\textwidth]{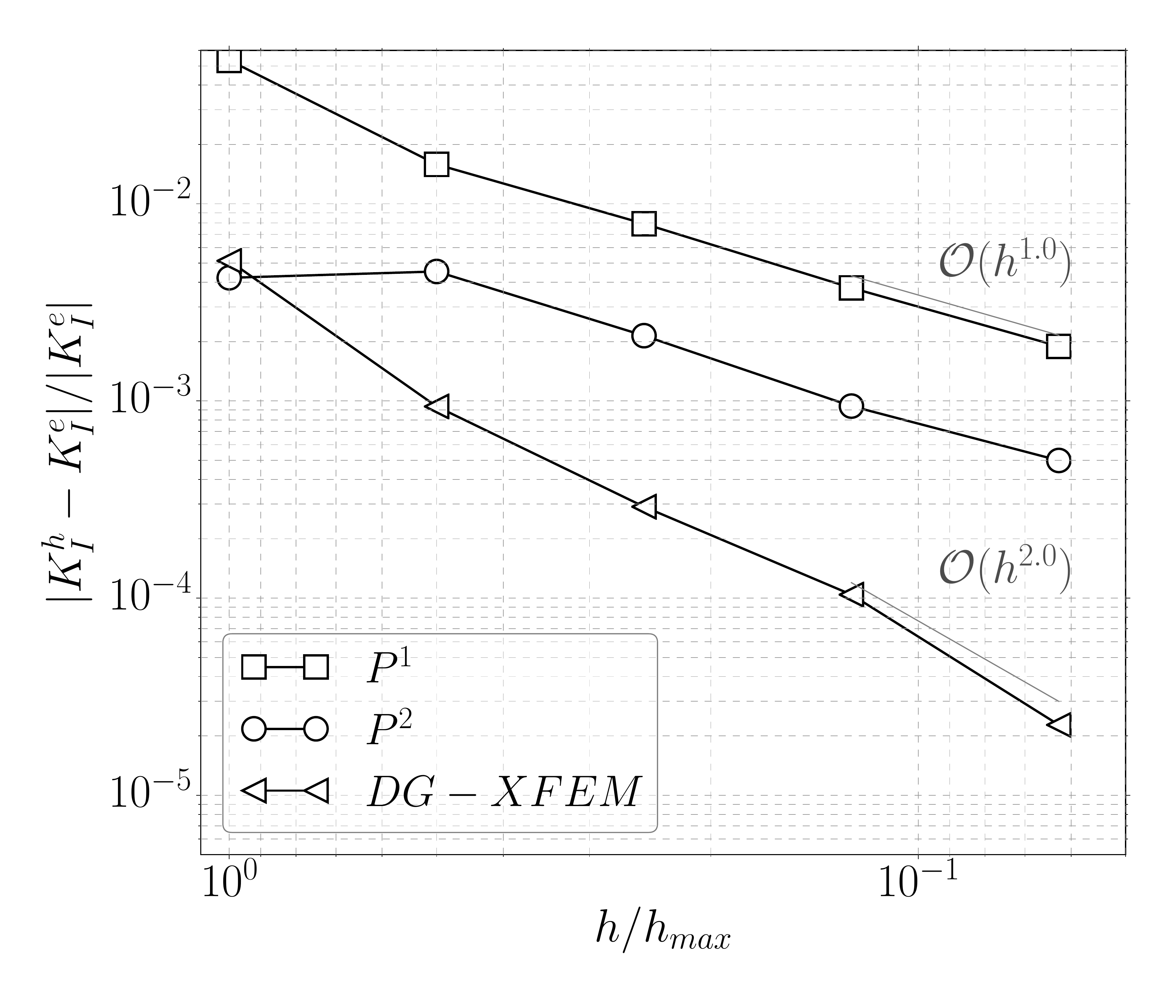} }
\subfloat[{ \color{black} $(\betab\au , \delta\gammab)= ( \betab^\text{TF}_\text{II}, \delta\gammab^\text{TAN} )$. }  ] { 
\includegraphics[trim=0.5in 0.25in 0.5in 0.55in,clip,width=0.5\textwidth]{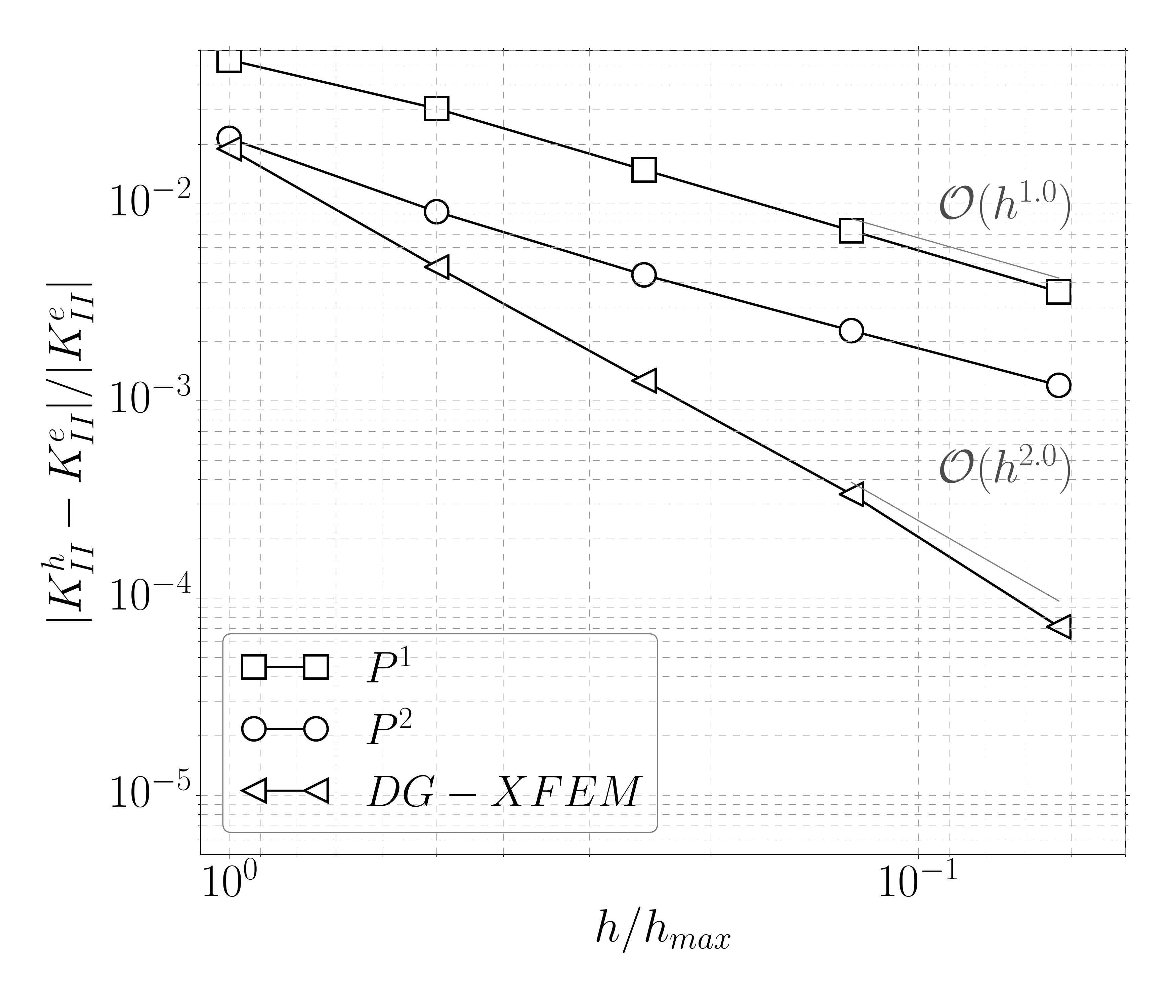} } \\
\subfloat[{ \color{black} $(\betab\au , \delta\gammab)= ( \betab^\text{DFC}_\text{I}, \delta\gammab^\text{UNI} )$. } ] { 
\includegraphics[trim=0.5in 0.25in 0.5in 0.55in,clip,width=0.5\textwidth]{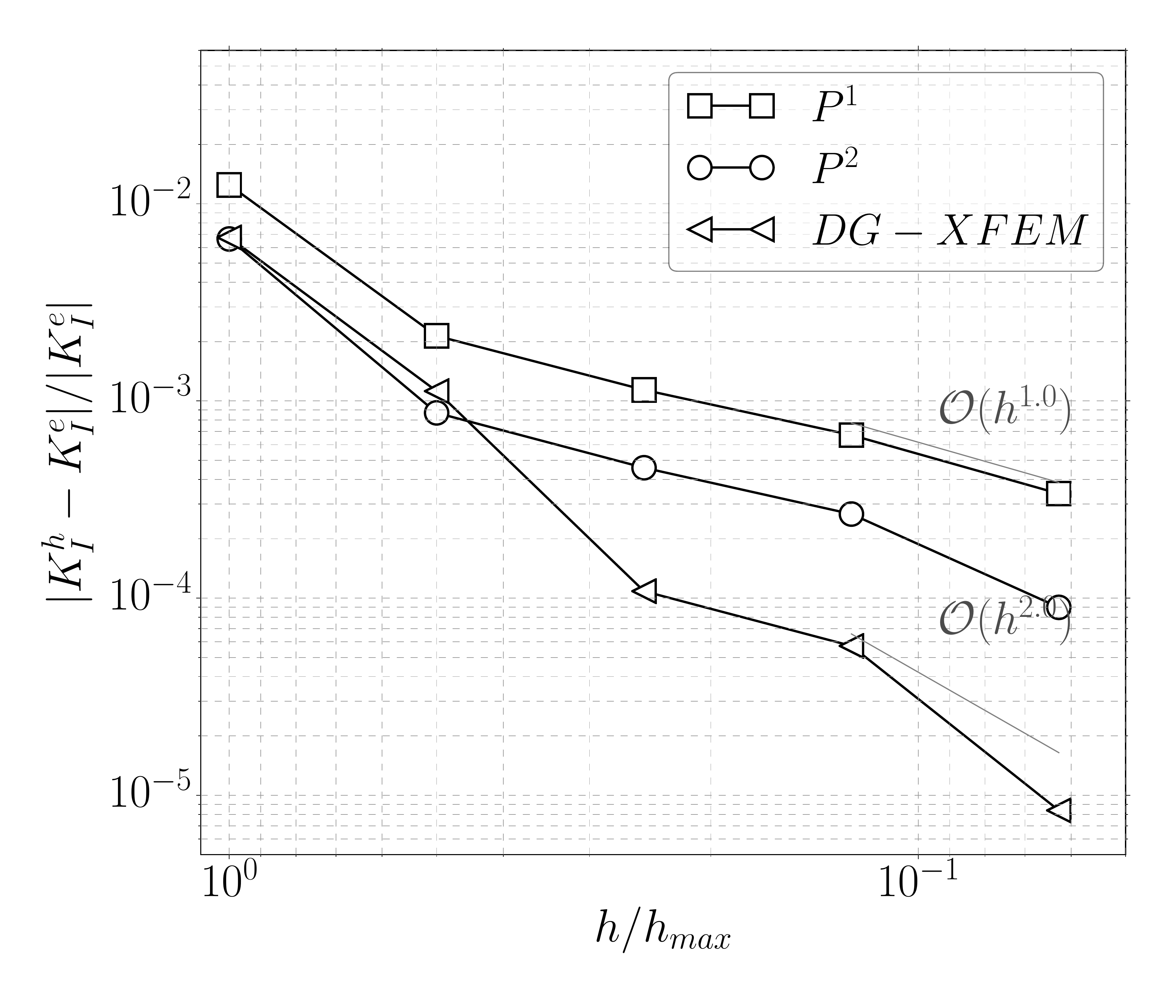} }
\subfloat[{ \color{black} $(\betab\au , \delta\gammab)= ( \betab^\text{DFC}_\text{II}, \delta\gammab^\text{UNI} )$. } ] { 
\includegraphics[trim=0.5in 0.25in 0.5in 0.55in,clip,width=0.5\textwidth]{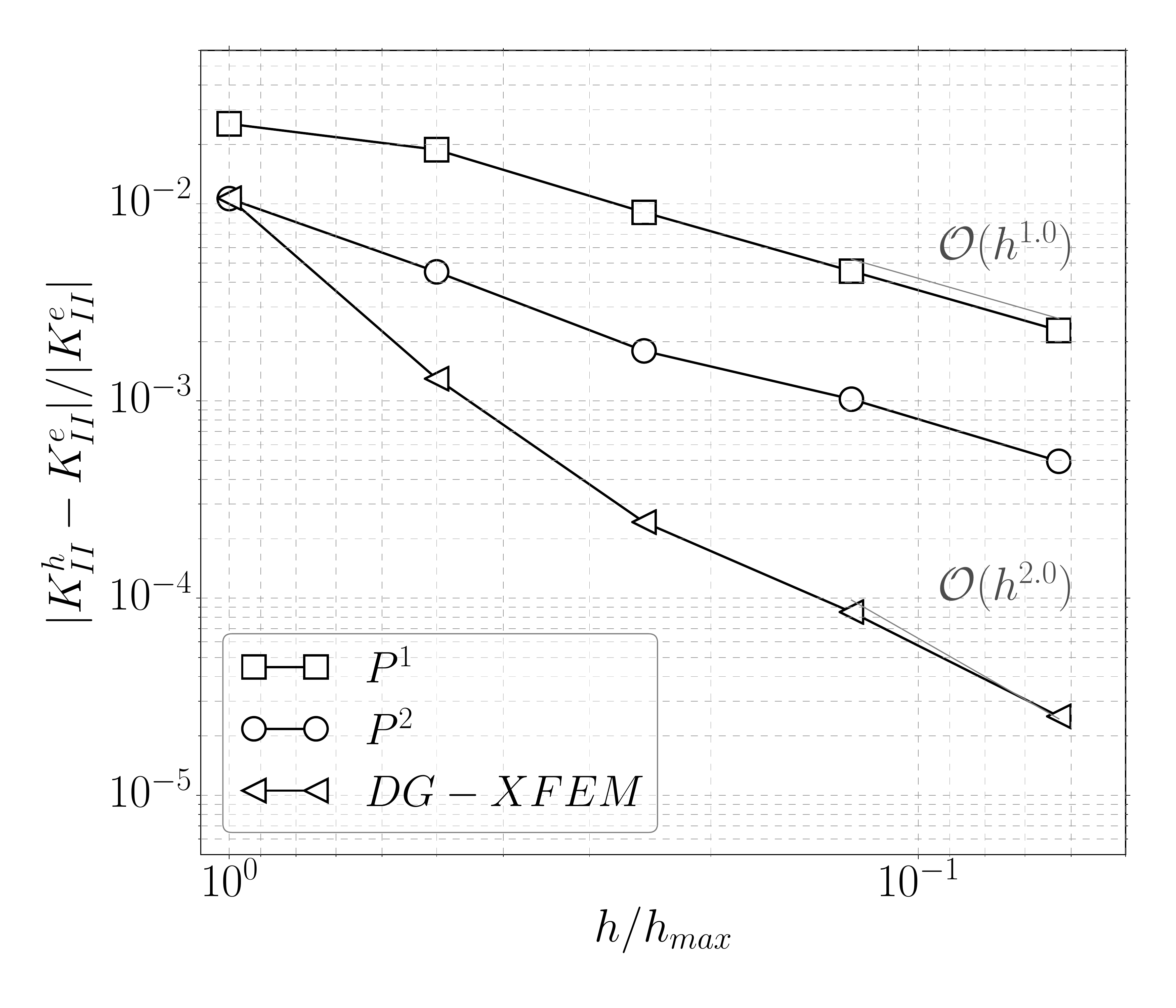} } \\
\subfloat[{ \color{black} $(\betab\au , \delta\gammab)= ( \betab^\text{DFC}_\text{I}, \delta\gammab^\text{TAN} )$. } ] { 
\includegraphics[trim=0.5in 0.25in 0.5in 0.55in,clip,width=0.5\textwidth]{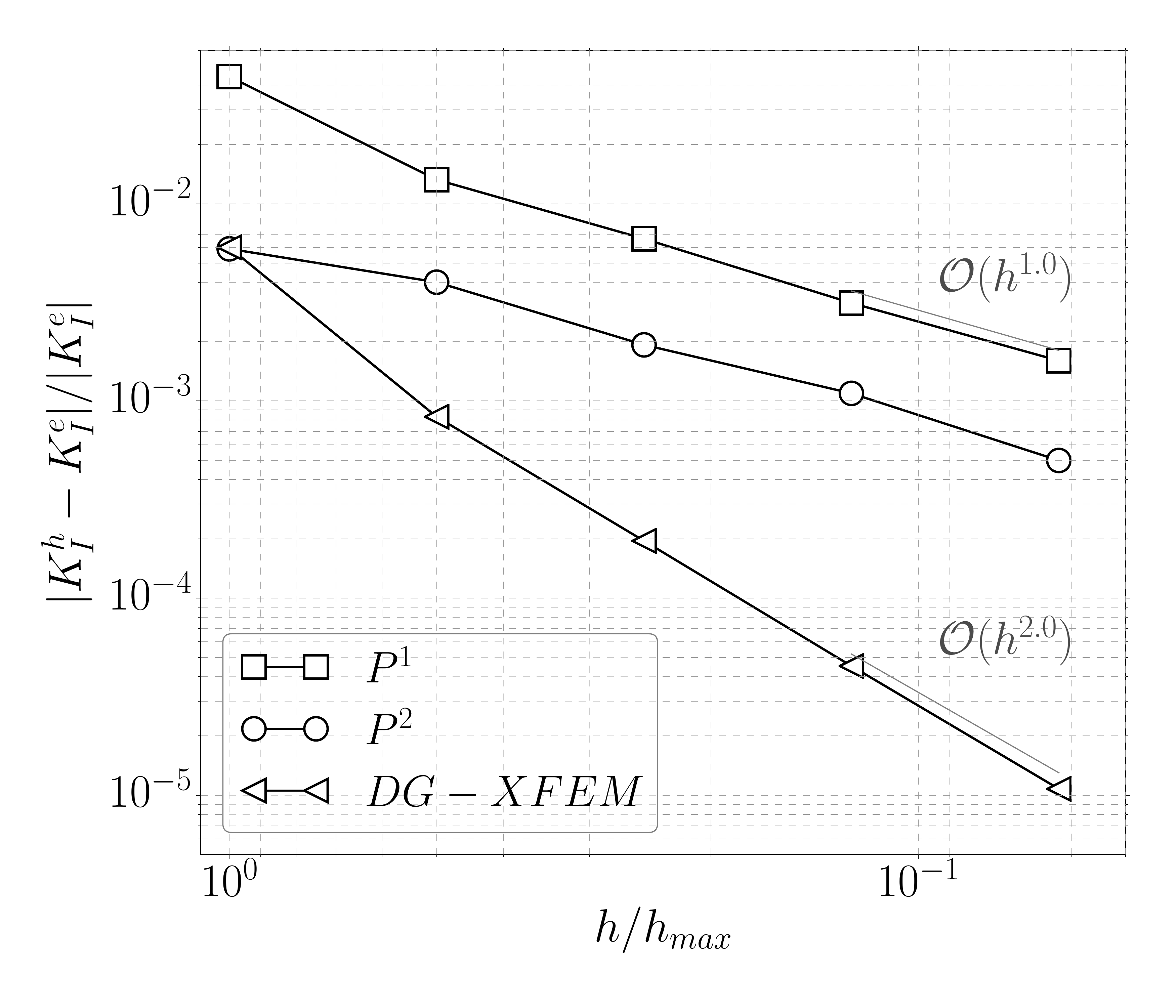} }
\subfloat[{ \color{black} $(\betab\au , \delta\gammab)= ( \betab^\text{DFC}_\text{II}, \delta\gammab^\text{TAN} )$. }] { 
\includegraphics[trim=0.5in 0.25in 0.5in 0.55in,clip,width=0.5\textwidth]{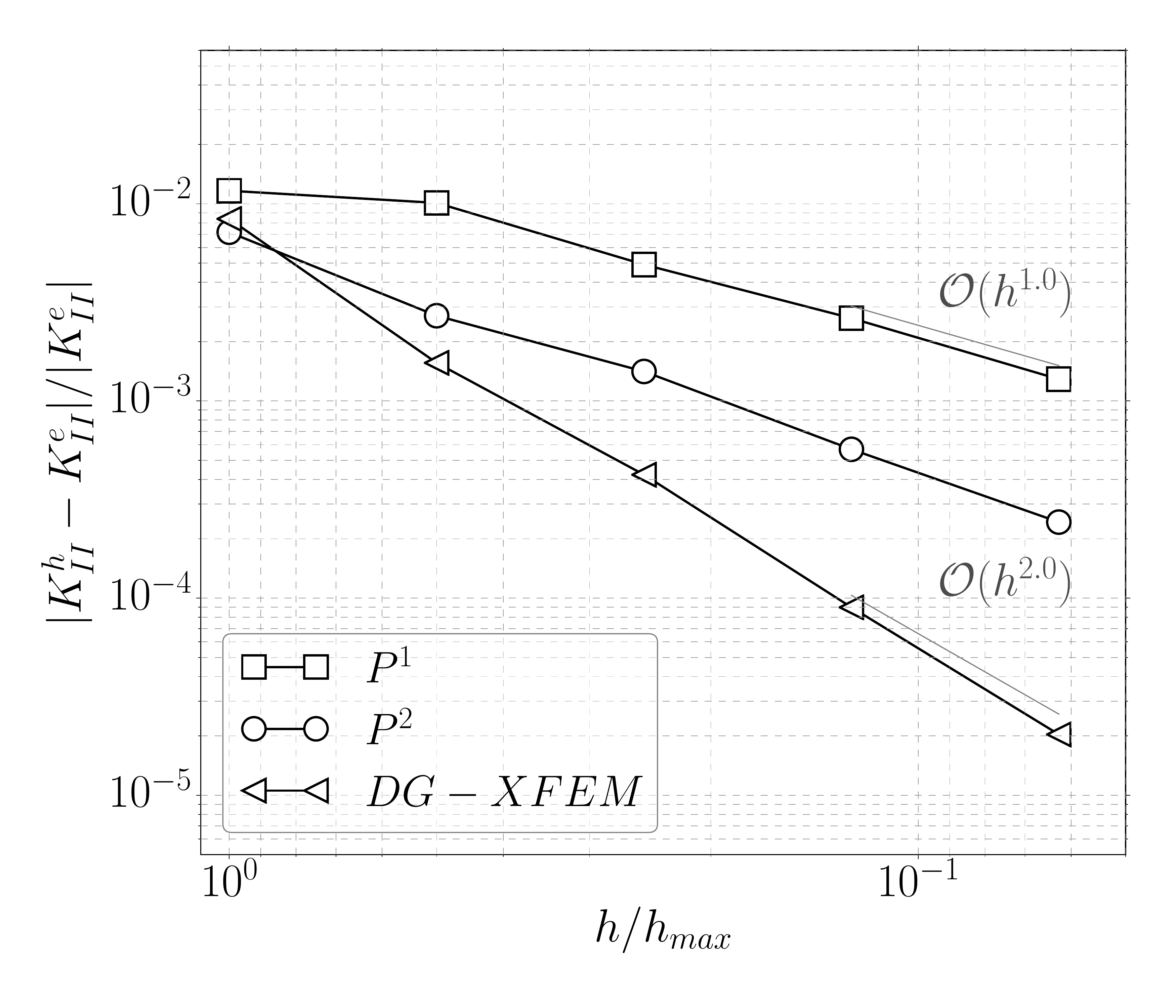} }
\caption{Convergence of the stress intensity factors for the circular arc crack  }
\label{fig:convergence_sif_circular_arc_crack}
\end{figure} 
\fi
\begin{figure}[htbp]
\centering
\includegraphics[trim=0.5in 0.65in 0.725in 0.55in,clip,width=0.5\textwidth]{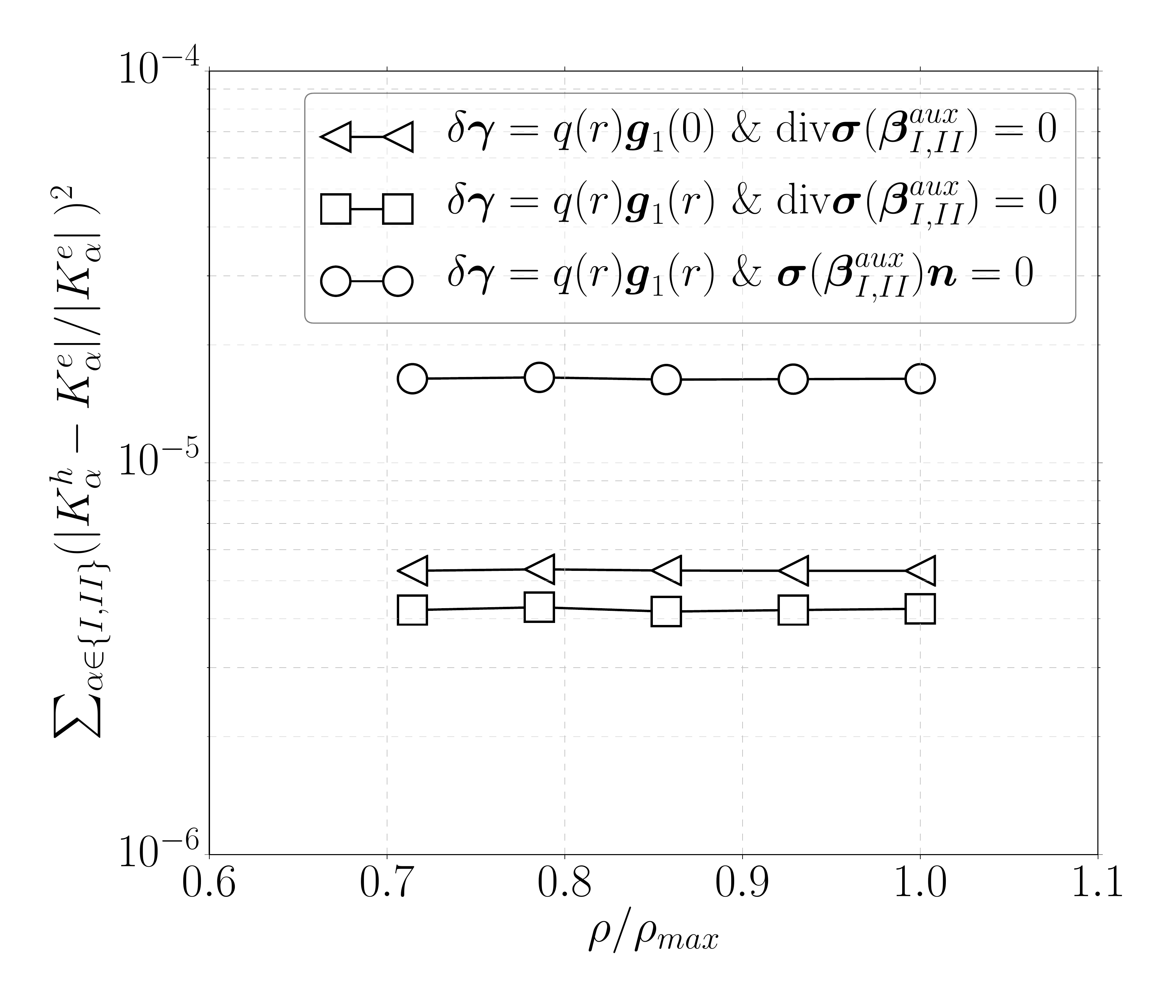}
\captionof{figure}{Contour independence of the interaction integral}
\label{fig:radialindependence_circular_arc_crack}
\end{figure}

\subsection{Power Function Crack}  
\renewcommand{\directory}{Figures/PowerFunctionCrack/}
\renewcommand{\example}{PowerFunctionCrack}
The second example we consider is the one of the power function crack $\Csc=\{(x,  x^3)|\, x \in [0,1]\}$ loaded by a force field $\bb$ and crack face traction $\overline\tb$, see Fig. \ref{fig:powcrack}.

The exact stress field is constructed by a superposition of a singular stress field $\hat \cs^e$ with a bounded field $\cs_{\bb}$ as 
\begin{equation*}
\cs^e = \hat \cs^e +  \cs_{\bb}.
\end{equation*}
The field $\hat\cs^e$ is constructed as 
\begin{equation*}
\hat\cs^e = {\left( \sigma^I_{ij} + \sigma^{II}_{ij} \right)} \gb_i(0) \otimes \gb_{j}(0),
\end{equation*}
where $\sigma_{ij}^\modes$ are given in \S\ref{sxn:app1} and are evaluated for values of $\vartheta \in [ -\pi -\zeta(r) , \pi - \zeta(r) ]$, as discussed in \S \ref{subsxn:fracture_mechanics_problem}.  The bounded stress field $\cs_{\bb} $ is constructed as
\begin{equation*}
\cs_\bb =  x \ee_x\otimes\ee_x + y \ee_y \otimes \ee_y.
\end{equation*}
Note that 
\begin{equation*}
\dive \cs^e =  \dive\hat \cs^e + \dive \cs_{\bb} =  \dive \cs_{\bb} = -\bb,
\end{equation*}
where
\begin{equation*}
\bb = -(\ee_x + \ee_y).
\end{equation*}

It is worth remarking that the stress intensity factors of $\cs^e$ will
correspond to those of the singular field $\hat \cs^e$ without any
perturbation from the bounded field. In fact, given that both $\lim_{r\to0}
\sqrt{r} \hat\cs^e ,\lim_{r\to0} \sqrt{r} \cs_{\bb} =\0 $ exist, we have 
\begin{equation*}
\lim_{r\to0} \sqrt{r} \cs^e = \lim_{r\to 0} \sqrt{r} (\hat\cs^e + \cs_{\bb}) = \lim_{r\to 0} \sqrt{r} \cs_\modes  + \lim_{r\to 0} \sqrt{r}  \cs_{\bb} = \lim_{r\to 0} \sqrt{r} \cs_\modes. 
\end{equation*}
For our example we take the stress intensity factors of  $\hat\cs^e$ to be
$K^e_I = K^e_{II} = 1$.

Figures \ref{fig:powcrack} and \ref{fig:powcrackbc} show the schematic of the problem and the modeled
domain with the applied boundary conditions, respectively. For the simulations the modeled domain was given by
$\Bc = [0,2]\times [-0.25,1.75]$.

Like for the previous example, we computed the solution for several
levels of refinement and investigated the convergence of the computed stress
intensity factors. 
\sorted{
\todo{AL: we may want to remove the plot of the div of dipl figure,
  and the reference to it in the last sentence
  
  MC: done.}
  }

The error measures \eqref{eq:error_interior} and
\eqref{eq:error_boundary} were observed to decrease as $\Oc(h^{1/2})$ and
$\Oc(h)$ for the first- and second-order methods, respectively. The values are
plotted in  Fig. \ref{fig:convergece_power_function_crack} and the errors and
the computed rates of convergence are tabulated in Table
\ref{table:convergence_power_function_crack}.

The stress intensity factors were observed to converge to the analytical value
as $\Oc( h)$ and $\Oc( h^2)$ when using the solution of the first- and second-order
methods, respectively. The values of the error in the stress intensity factors
are plotted in Fig. \ref{fig:convergece_sif_power_function_crack}, and the
errors, as well as the computed convergence rates, are provided in Table
\ref{table:convergence_sif_power_function_crack}.

Lastly,  in Fig.
\ref{fig:radial_independence_power_function_crack} we illustrate the independence of the
interaction integral on the size of the support of $\delta \gammab$ by plotting
the stress intensity factors for five values of $\rho/\rho_\text{max}$ ranging from $\sim 0.7
 $ to $1$, for $\rho_\text{max} = 0.5 $. 

\begin{figure}[htbp]
\begin{minipage}{0.5\textwidth}
\centering
\begin{tikzpicture}[x=2.5cm,y=2.5cm]
\def\originx{0}
\def\originy{0}
\def\sizeb{1}
\def\sizep{1.35}
\def\alph{60}
\def\radius{1}
\draw [color=black!0!white,fill=none](-\sizep,-\sizep)-- (0,-\sizep) -- (\sizep,-\sizep)--(\sizep,\sizep)--(-\sizep,\sizep)--(-\sizep,0)  --(-\sizep,-\sizep);
\draw [color=black!10!white,thick,fill=black!10!white](-\sizeb,-\sizeb)-- (0,-\sizeb) -- (\sizeb,-\sizeb)--(\sizeb,\sizeb)--(-\sizeb,\sizeb)--(-\sizeb,0)  --(-\sizeb,-\sizeb);
\draw [color=black!0!white,dashed](-\sizep,-\sizep)-- (0,-\sizep) -- (\sizep,-\sizep)--(\sizep,\sizep)--(-\sizep,\sizep)--(-\sizep,0)  --(-\sizep,-\sizep);
\draw [black,domain=-1:0] plot (\x, { (\x+\sizeb)^3-\sizeb/2});
\draw [arrows={latex-latex}] ( -3*0.1 ,1*0.1 + \sizeb/2 )node[left] {$\gb_2(0)$} -- (0,\sizeb/2) --  (1*0.1,\sizeb/2 +3 *0.1)node[right] {$\gb_1(0) $};
\foreach \t in {-\sizeb,-0.90001,...,-0.001}
            \draw [black!50!white, opacity=1.0, -latex, thick]
                (\t, { pow(\t+\sizeb,3)-\sizeb/2} )-- ({\t - pow(\t+\sizeb,2)*0.1*3}, { pow(\t+\sizeb,3)-\sizeb/2  + 0.1}) ;
\foreach \t in {-\sizeb,-0.90001,...,-0.001}
            \draw [black!50!white, opacity=1.0, -latex, thick]
                (\t, { pow(\t+\sizeb,3)-\sizeb/2} )-- ({\t + pow(\t+\sizeb,2)*0.1*3}, { pow(\t+\sizeb,3)-\sizeb/2  - 0.1 }) ;
\foreach \t in {-\sizeb,-0.90001,...,\sizeb}
            \draw [black!50!white, opacity=1.0, -latex, thick]
                (\t,\sizeb)-- (\t,\sizeb+\sizeb/7.5) ;
\foreach \t in {-\sizeb,-0.90001,...,\sizeb}
            \draw [black!50!white, opacity=1.0, -latex, thick]
                (\t,-\sizeb)-- (\t,{-\sizeb-\sizeb/7.5 - abs(\t)/7 }) ;                
\foreach \t in {-\sizeb,-0.90001,...,\sizeb}
            \draw [black!50!white, opacity=1.0, -latex, thick]
                (\sizeb,\t)-- (\sizeb+\sizeb/7.5+ \t*\t/7,\t) ;
\foreach \t in {-\sizeb,-0.90001,...,\sizeb}
            \draw [black!50!white, opacity=1.0, -latex, thick]
                (-\sizeb,\t)-- ( {-\sizeb-\sizeb/5 + \t*\t/7},\t) ;      
\node at ( 0 , -\sizeb-\sizeb/5 ) [below] {$ \tb = \cs^e \nb$};
\node at ( -\sizeb*0.02,-0.02*\sizeb ) [right] {$ \tb^\pm = \cs^{e} \nb^\pm$};
\node at (0.9*\sizeb,0.9*\sizeb) { $\Bc$};
\node at (-0.05\sizeb,\sizeb*0.25 ) [opacity=1.0,above right] {$f(x) = x^3 $};
\node at (\sizeb/2,-\sizeb/2 ) [opacity=1.0,below] {$\bb =\ee_x + \ee_y$};
\draw [arrows={latex-latex}] (-\sizeb+0.9\sizeb,-\sizeb*0.8+0.15\sizeb)node[left] {$\ee_y$} -- (-\sizeb+0.9\sizeb,-\sizeb+0.15\sizeb) -- (-\sizeb*0.8+0.9\sizeb,-\sizeb+0.15\sizeb)node[below] {$\ee_x$};
\end{tikzpicture}
\captionof{figure}{The power function crack problem}
\label{fig:powcrack}
\end{minipage}
\begin{minipage}{0.5\textwidth} 
\centering
\begin{tikzpicture}[x=2.5cm,y=2.5cm]
\def\originx{0}
\def\originy{0}
\def\sizeb{1}
\def\sizep{1.35}
\draw [color=black!0!white,fill=none](-\sizep,-\sizep)-- (0,-\sizep) -- (\sizep,-\sizep)--(\sizep,\sizep)--(-\sizep,\sizep)--(-\sizep,0)  --(-\sizep,-\sizep);
\draw [color=black!10!white,thick,fill=black!10!white](-\sizeb,-\sizeb)-- (0,-\sizeb) -- (\sizeb,-\sizeb)--(\sizeb,\sizeb)--(-\sizeb,\sizeb)--(-\sizeb,0)  --(-\sizeb,-\sizeb);
\draw [color=black!0!white,dashed](-\sizep,-\sizep)-- (0,-\sizep) -- (\sizep,-\sizep)--(\sizep,\sizep)--(-\sizep,\sizep)--(-\sizep,0)  --(-\sizep,-\sizep);
\draw [black,domain=-1:0] plot (\x, { (\x +\sizeb)^3 - \sizeb/2});
\node at (-0.25, { (-0.25 + \sizeb)^3 - \sizeb/2}) [right] {$\Csc^\pm$};

\draw [color=black!40!white,thick,latex-latex](-\sizeb,\sizeb/3*1.1- \sizeb/2) node[right] {$y$}-- (-\sizeb,- \sizeb/2)node[below]{} -- (-\sizeb + \sizeb/3*1.1,- \sizeb/2) node[right] {$x$};

\draw [color=black,thick,dashed] (-\sizeb,\sizeb )-- (0,\sizeb )node[above] {$ \cs\ee_y  = \overline\tb^e$ } -- (\sizeb,\sizeb ) -- (\sizeb,0) node[rotate=270,above] {$\cs\ee_x  = \overline\tb^e$ } -- (\sizeb,-\sizeb) -- (0,-\sizeb) node[below]{$-\cs\ee_y  = \overline\tb^e$}--(-\sizeb,-\sizeb);
\draw [color=black,thick,fill=black!10!white,dashed](-\sizeb,\sizeb)--(-\sizeb,0) node[above,rotate=90]{$-\cs\ee_x  = \overline\tb^e$ } --(-\sizeb,-\sizeb);
\node at ( {-0.5 +0.1 },{pow( cos(deg(-0.5+\sizeb)/2*pi ),2)*0.75 } ) [right] {$ \cs\nb^\pm = \overline\tb^{e\pm}$};
\node at (-\sizeb,-\sizeb ) [left] {$\ub = \0$};
\node at (\sizeb,-\sizeb ) [right] {$\ub\cdot\ee_y= \0$};
\draw [arrows={latex-latex}] (-\sizeb,-\sizeb*0.8)node[left] {$\ee_y$} -- (-\sizeb,-\sizeb) -- (-\sizeb*0.8,-\sizeb)node[below] {$\ee_x$};

\node at (0.9*\sizeb,0.9*\sizeb) { $\Bc$};
\end{tikzpicture}
\captionof{figure}{Modeled subdomain}
\label{fig:powcrackbc}
\end{minipage}
\end{figure}
\ifnum\tables=1
\begin{table}[htb]
\centering
\caption{ Convergence rates of the derivatives of the solution for the circular arc crack problem }
\label{table:convergence_power_function_crack}
\subfloat[Domain convergence]{
\begin{tabular}{ l | c c | c c | c c }
\cmidrule{2-7}
\multicolumn{1}{ c  }{ } & \multicolumn{6}{c}{ $\| \betab^h - \betab^e\|_{L^2(\Bc) } $ }  \\
\cmidrule{2-7}
\multicolumn{1}{ c  }{ } & \multicolumn{2}{ c | }{ $P^1$}  & \multicolumn{2}{  c|  }{ $P^2$ } & \multicolumn{2}{  c  }{ $DG-XFEM$  }
 \\
 \cmidrule{1-7}
$h_\text{max}/h $& Err. & $\Oc$ &  Err. & $\Oc$ &  Err. & $\Oc$   \\
\hline\hline
1&0.00080 & -- &0.00040 & -- &0.00026 & --\\
2&0.00061 & 0.41 &0.00032 & 0.36 &0.00012 & 1.09\\
4&0.00043 & 0.50 &0.00022 & 0.50 &0.00006 & 1.03\\
8&0.00031 & 0.47 &0.00016 & 0.49 &0.00003 & 1.00\\
16&0.00022 & 0.50 &0.00011 & 0.50 &0.00002 & 1.00\\

\hline
\end{tabular}
}

\subfloat[Trace convergence]{
\begin{tabular}{ l |  c c | c c | c c  }
\cmidrule{2-7}
\multicolumn{1}{ c  }{ } & \multicolumn{6}{c}{  $\| \betab^h - \betab^e \circ  \mathfrak{p}  \|_{L^2(\Csc^h_\pm,r) }  $ } \\
\cmidrule{2-7}
\multicolumn{1}{ c  }{ } & \multicolumn{2}{ c | }{ $P^1$}  & \multicolumn{2}{  c|  }{ $P^2$ } & \multicolumn{2}{  c }{ $DG-XFEM$ }\\
 \cmidrule{1-7}
$h_\text{max}/h$& Err. & $\Oc$ &  Err. & $\Oc$ &  Err. & $\Oc$   \\
\hline\hline
1&0.00066 & -- &0.00037 & -- &0.00035 & --\\
2&0.00050 & 0.40 &0.00029 & 0.34 &0.00013 & 1.39\\
4&0.00036 & 0.48 &0.00021 & 0.48 &0.00006 & 1.06\\
8&0.00026 & 0.43 &0.00015 & 0.48 &0.00003 & 0.97\\
16&0.00019 & 0.49 &0.00011 & 0.50 &0.00002 & 0.99\\

\hline
\end{tabular}
}
\end{table}
\fi
\begin{figure}[htbp]
\centering
\subfloat[Convergence in the $L^2(\Bc_{\Csc} )$ norm] { 
\includegraphics[trim=0.65in 0.5in 0.5in 0.5in,clip , width=0.5\textwidth]{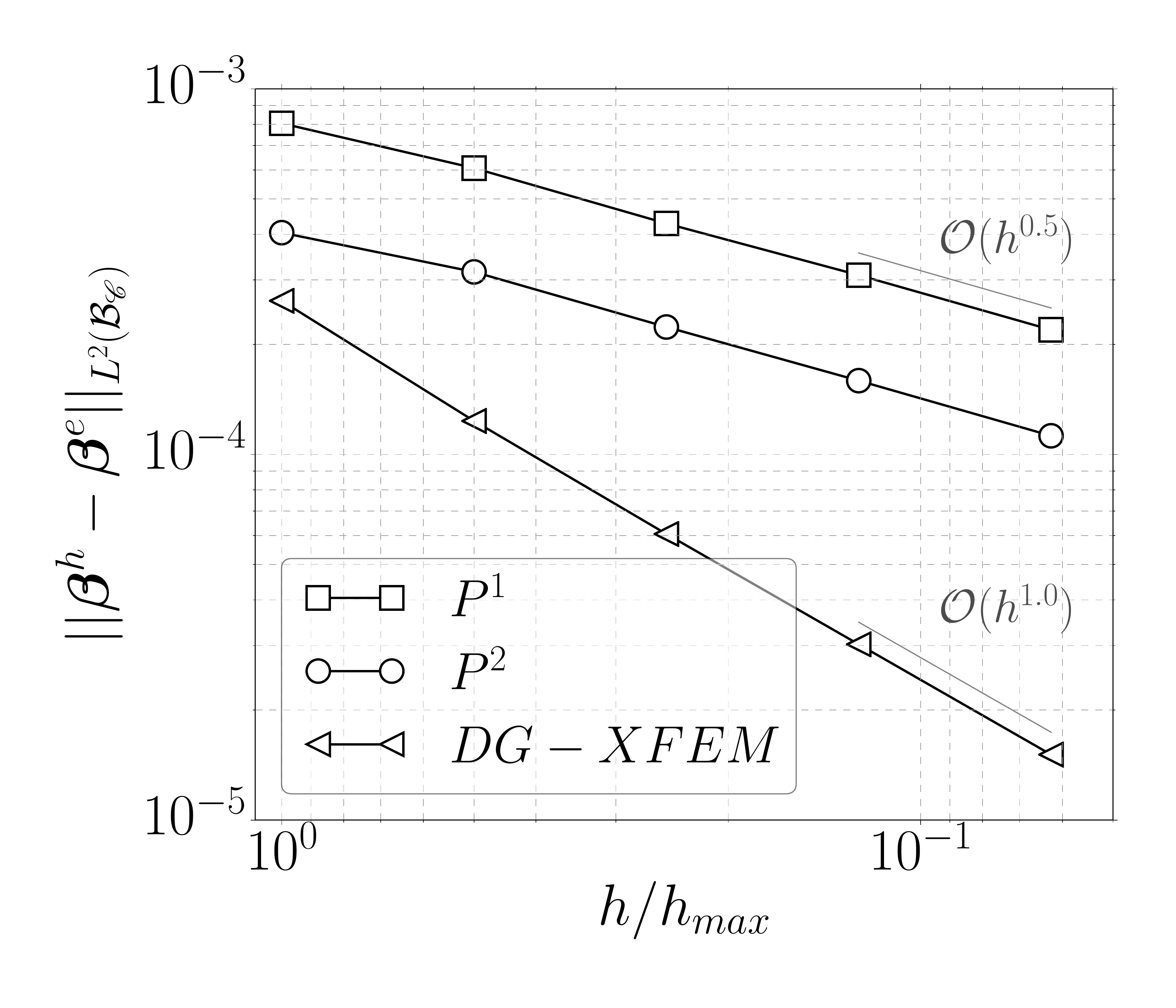} 
}
\subfloat[Convergence in the  $L^2( \Csc_\pm^h,r ) $ norm] { 
\includegraphics[trim=0.65in 0.5in 0.5in 0.5in,clip ,width=0.5\textwidth]{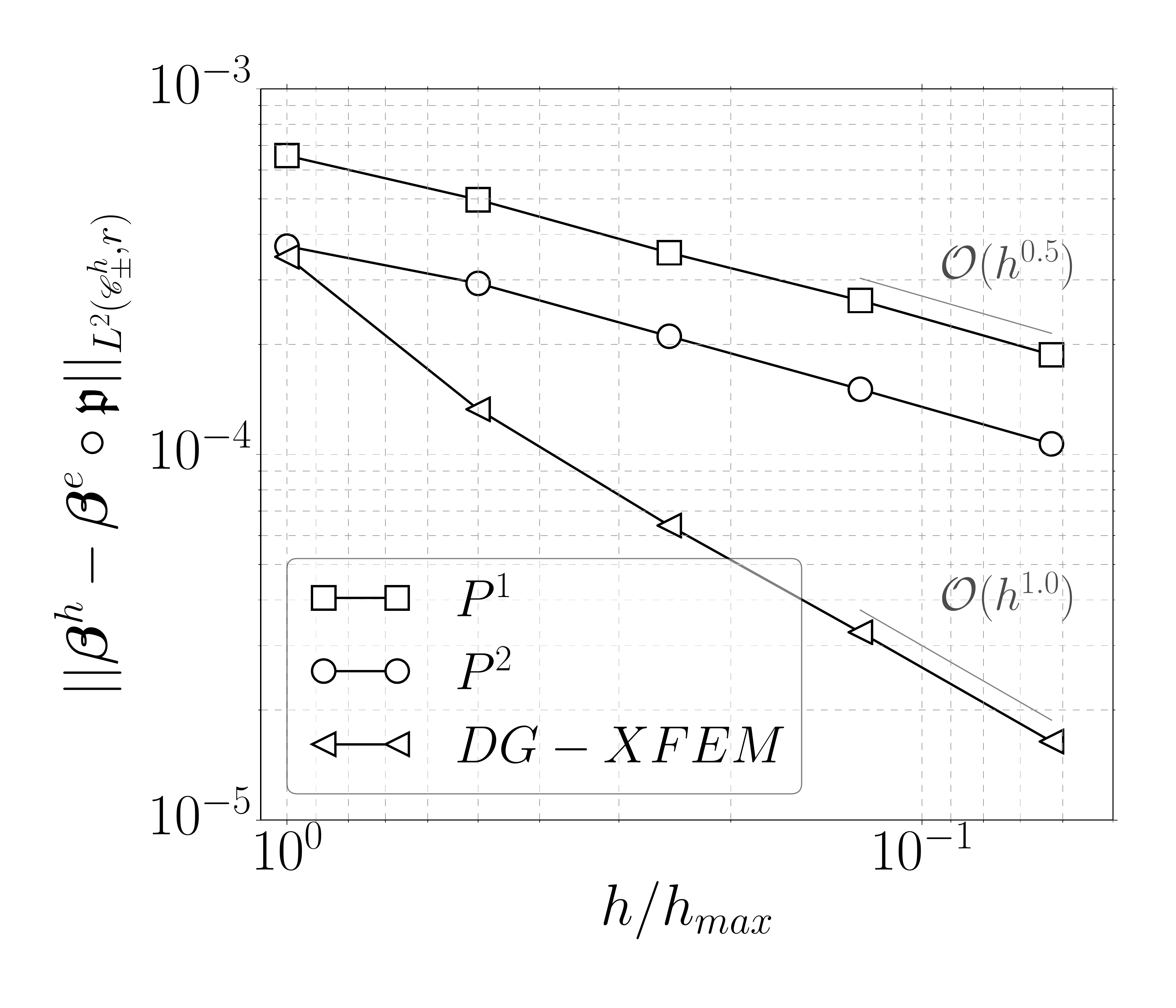} 
}
\caption{Convergence of the solution.}
\label{fig:convergece_power_function_crack}
\end{figure}


\ifnum\tables=1
\begin{table}[htbp]
\centering
\caption{ Convergence rates for stress intensity factors  }
\label{table:convergence_sif_power_function_crack}
\subfloat[{\color{black} Traction free auxiliary fields ($\betab\au = \betab^\text{TF}$) and tangential \testfunctionname ($\delta \gammab = \delta\gammab^\text{TAN} $)}.]{
\begin{tabular}{ l | c c | c c | c c | c c | c c | c c}
\cmidrule{2-13}
\multicolumn{1}{ c  }{ } & \multicolumn{4}{ c | }{ $P^1$}  & \multicolumn{4}{  c|  }{ $P^2$ } & \multicolumn{4}{  c }{ $DG-XFEM$ } \\
 \cmidrule{2-13}
\multicolumn{1}{ c  }{ }& \multicolumn{2}{ c | }{  $K_I$  } &\multicolumn{2}{c |}{ $K_{II } $} &  \multicolumn{2}{c|}{ $K_I$ } & \multicolumn{2}{ c |}{ $K_{II}$ } & \multicolumn{2}{ c| }{$K_I$ } & \multicolumn{2}{ c }{ $K_{II}$ }  \\
 \cmidrule{1-13}
$h_\text{max}/h $& Err. & $\Oc$ &  Err. & $\Oc$ &  Err. & $\Oc$ &  Err. & $\Oc$ &  Err. & $\Oc$ &  Err. & $\Oc$   \\
\hline\hline

\hline
\end{tabular}
}\\
\subfloat[{\color{black} Divergence-free ($\betab\au = \betab^\text{DFC}$) and unidirectional \testfunctionname  ($\delta \gammab = \delta\gammab^\text{UNI} $ ). }]{
\begin{tabular}{ l | c c | c c | c c | c c | c c | c c}
\cmidrule{2-13}
\multicolumn{1}{ c  }{ } & \multicolumn{4}{ c | }{ $P^1$}  & \multicolumn{4}{  c|  }{ $P^2$ } & \multicolumn{4}{  c }{ $DG-XFEM$ } \\
 \cmidrule{2-13}
\multicolumn{1}{ c  }{ }& \multicolumn{2}{ c | }{  $K_I$  } &\multicolumn{2}{c |}{ $K_{II } $} &  \multicolumn{2}{c|}{ $K_I$ } & \multicolumn{2}{ c |}{ $K_{II}$ } & \multicolumn{2}{ c| }{$K_I$ } & \multicolumn{2}{ c }{ $K_{II}$ }  \\
 \cmidrule{1-13}
$h_\text{max}/h $& Err. & $\Oc$ &  Err. & $\Oc$ &  Err. & $\Oc$ &  Err. & $\Oc$ &  Err. & $\Oc$ &  Err. & $\Oc$   \\
\hline\hline

\hline
\end{tabular}
}\\
\subfloat[{\color{black} Divergence-free ($\betab\au = \betab^\text{DFC}$) and tangential \testfunctionname  ($\delta \gammab = \delta\gammab^\text{TAN} $ ).}]{
\begin{tabular}{ l | c c | c c | c c | c c | c c | c c}
\cmidrule{2-13}
\multicolumn{1}{ c  }{ } & \multicolumn{4}{ c | }{ $P^1$}  & \multicolumn{4}{  c|  }{ $P^2$ } & \multicolumn{4}{  c }{ $DG-XFEM$ } \\
 \cmidrule{2-13}
\multicolumn{1}{ c  }{ }& \multicolumn{2}{ c | }{  $K_I$  } &\multicolumn{2}{c |}{ $K_{II } $} &  \multicolumn{2}{c|}{ $K_I$ } & \multicolumn{2}{ c |}{ $K_{II}$ } & \multicolumn{2}{ c| }{$K_I$ } & \multicolumn{2}{ c }{ $K_{II}$ }  \\
 \cmidrule{1-13}
$h_\text{max}/h$& Err. & $\Oc$ &  Err. & $\Oc$ &  Err. & $\Oc$ &  Err. & $\Oc$ &  Err. & $\Oc$ &  Err. & $\Oc$   \\
\hline\hline

\hline
\end{tabular}
}
\end{table} 
\fi

\ifnum1=0
\begin{figure}[htbp] 
\centering
\foreach \method/\cap in {1/{$\betab = \betab^{DFC},\delta\gammab = \delta\gammab^{UNI}$},2/{$\betab = \betab^{DFC},\delta\gammab = \delta\gammab^{TAN}$},3/{$\betab = \betab^{TF},\delta\gammab = \delta\gammab^{TAN}$}}{
\foreach \mode in {1,2}{\subfloat[\cap -- Mode \ifnum\mode=1 I \else II \fi]{
\begin{tikzpicture}
\ifnum\mode=1%
	\begin{loglogaxis}[
				height=0.28\textheight,
				grid=minor,			
				xlabel={$h/h_{0} $},
				ylabel={$|K_I - K_I^e|/|K_I^e|$},					
				x dir=reverse,
				legend entries={$P^1$,$P^2$,$P^3$,$P^4$},
				legend columns=-1,
				legend style={at={(0.5,1.15)},anchor=north},									
				ymin=1.e-4,ymax=1.e-1
				]%
\else%
	\begin{loglogaxis}[
				height=0.28\textheight,
				grid=minor,			
				xlabel={$h/h_{0} $},		
				ylabel={$|K_{II} - K_{II}^e|/|K^e_{II}|$},		
				x dir=reverse,
				legend entries={$P^1$,$P^2$,$P^3$,$P^4$},
				legend columns=-1,
				legend style={at={(0.5,1.15)},anchor=north},									
				ymin=1.e-4,ymax=1.e-1,
				ylabel near ticks,yticklabel pos=right				
				]
\fi%
	\foreach \i [evaluate=\i as \ival using 1*\i]  in {1,2}{
		\ifnum\mode=1 
			\addplot table[x index=0,y index = 3] {Data/\example/solution_0_\i_\method.dat};						\else
			\addplot table[x index=0,y index = 4] {Data/\example/solution_0_\i_\method.dat}; 
		\fi	
	};
	\end{loglogaxis}
\end{tikzpicture}}}
	
}
\caption{Convergence of the stress intensity factors for the power function crack problem }
\label{fig:convergece_sif_power_function_crack}
\end{figure}
\else
\begin{figure}[htbp]
\centering
\subfloat[{ \color{black} $(\betab\au , \delta\gammab)= ( \betab^\text{TF}_\text{I}, \delta\gammab^\text{TAN})$ . } ] { 
\includegraphics[trim=0.5in 0.25in 0.5in 0.55in,clip,width=0.5\textwidth]{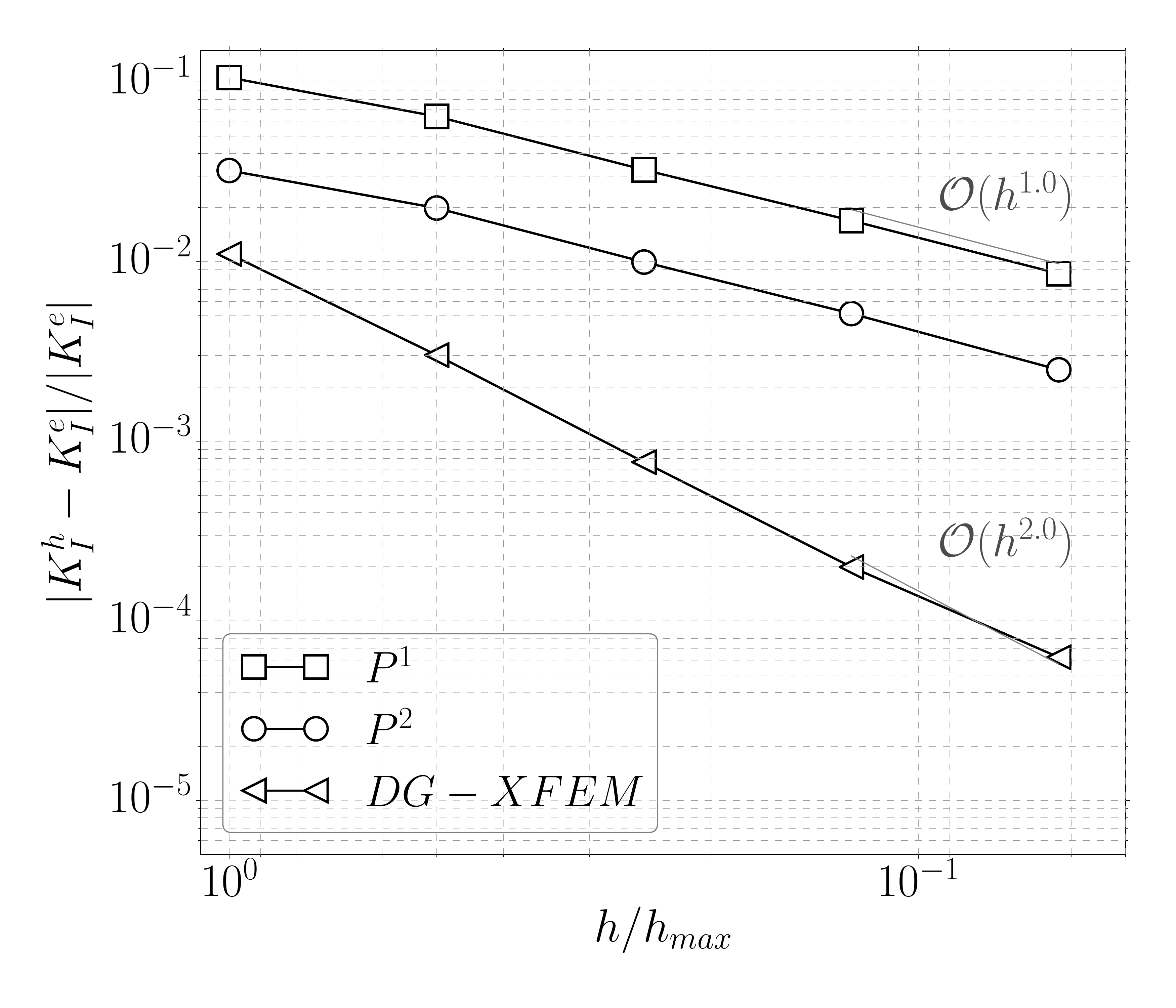} }
\subfloat[{ \color{black} $(\betab\au , \delta\gammab)= ( \betab^\text{TF}_\text{II}, \delta\gammab^\text{TAN} )$. } ] { 
\includegraphics[trim=0.5in 0.25in 0.5in 0.55in,clip,width=0.5\textwidth]{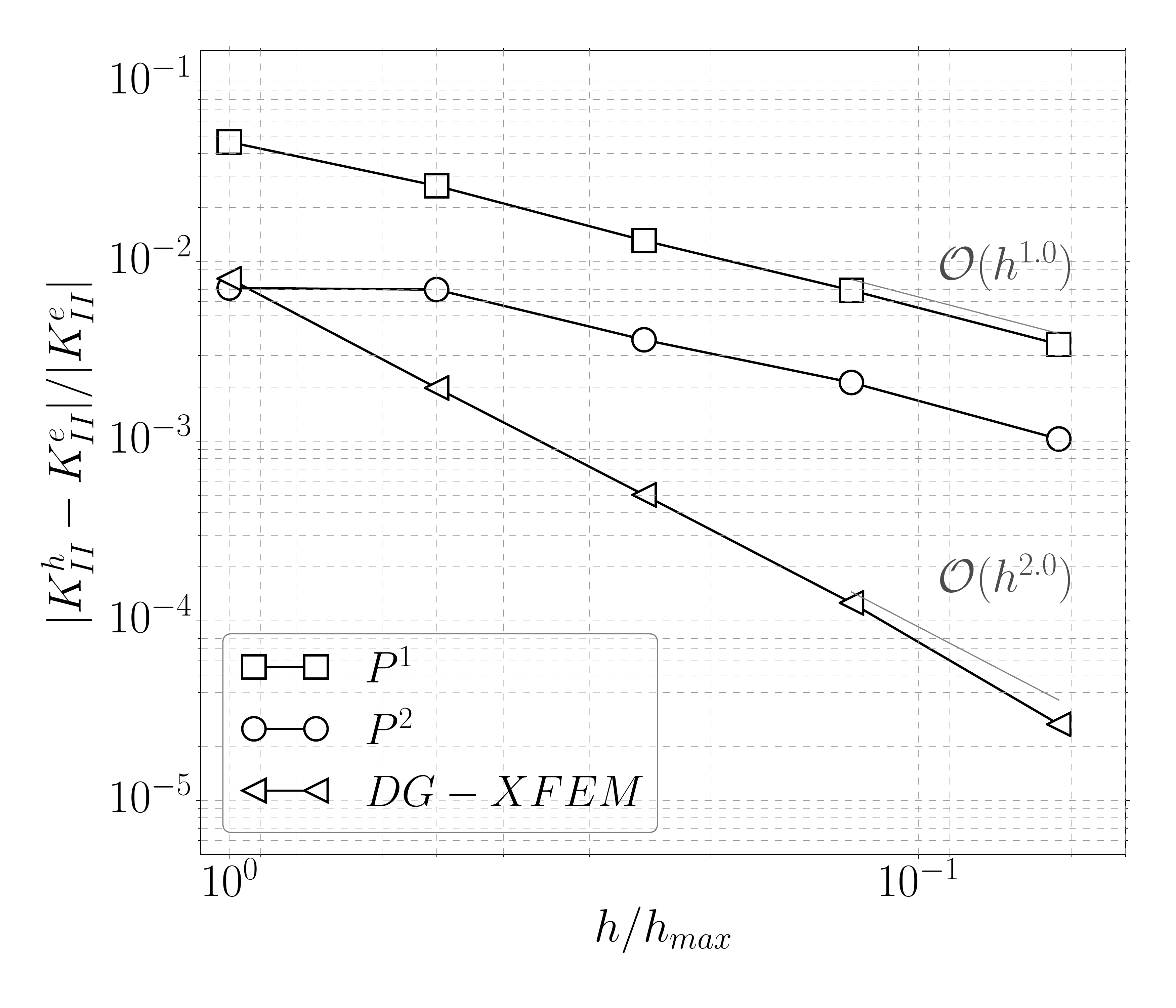} } \\
\subfloat[{ \color{black} $(\betab\au , \delta\gammab)= ( \betab^\text{DFC}_\text{I}, \delta\gammab^\text{UNI} )$. } ] { 
\includegraphics[trim=0.5in 0.25in 0.5in 0.55in,clip,width=0.5\textwidth]{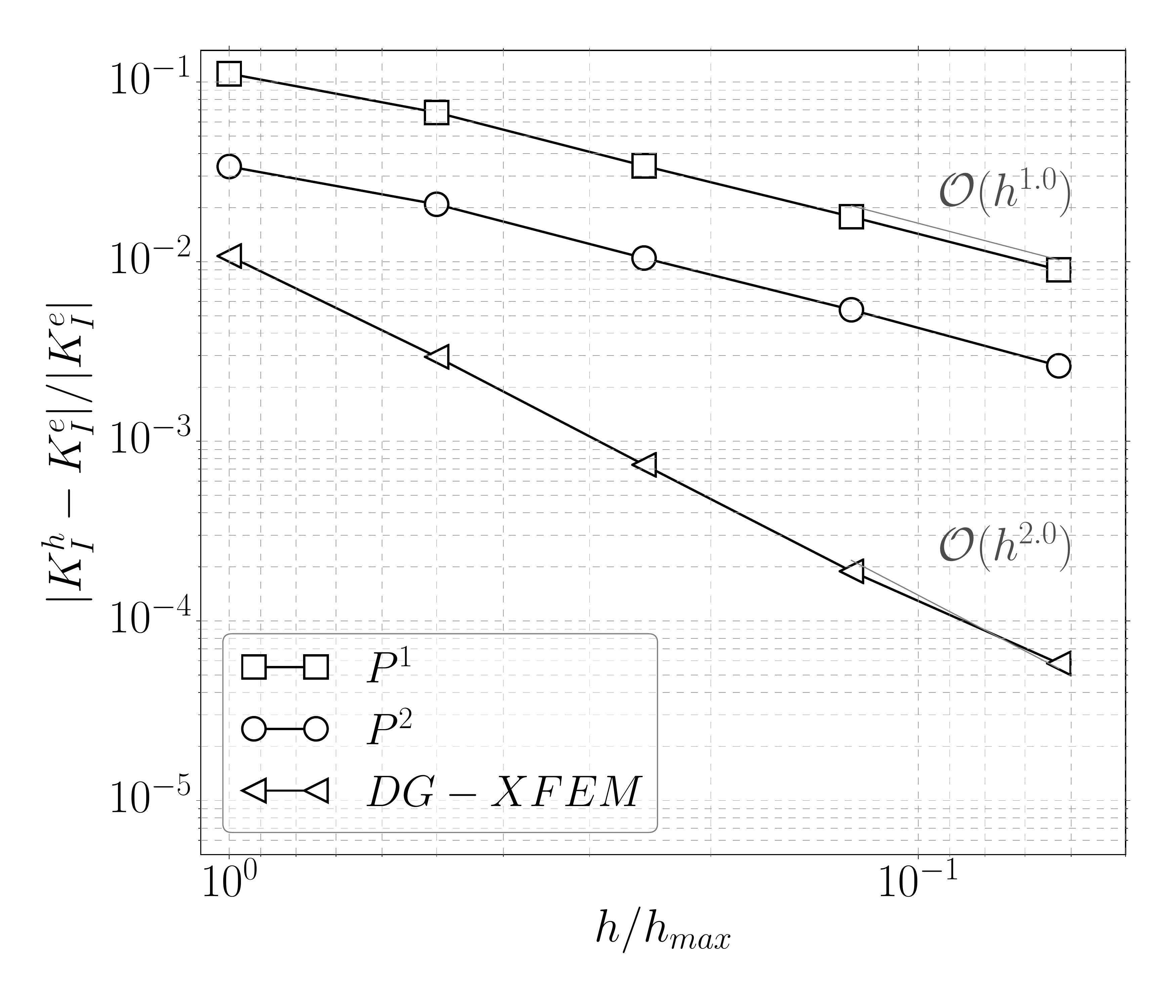} }
\subfloat[{ \color{black} $(\betab\au , \delta\gammab)= ( \betab^\text{DFC}_\text{II}, \delta\gammab^\text{UNI} )$. }] { 
\includegraphics[trim=0.5in 0.25in 0.5in 0.55in,clip,width=0.5\textwidth]{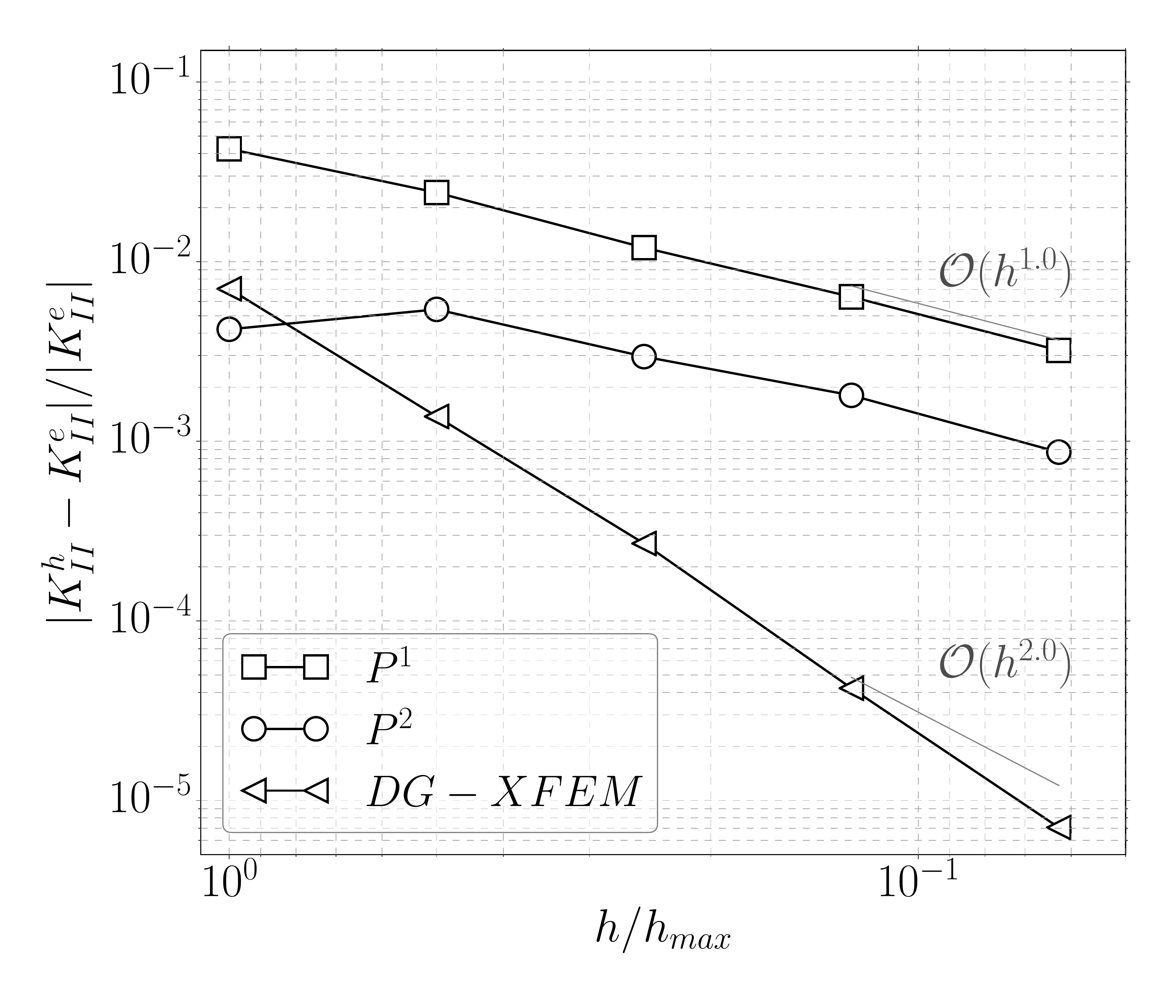} } \\
\subfloat[{ \color{black} $(\betab\au , \delta\gammab)= ( \betab^\text{DFC}_\text{I}, \delta\gammab^\text{TAN} )$. } ] { 
\includegraphics[trim=0.5in 0.25in 0.5in 0.55in,clip,width=0.5\textwidth]{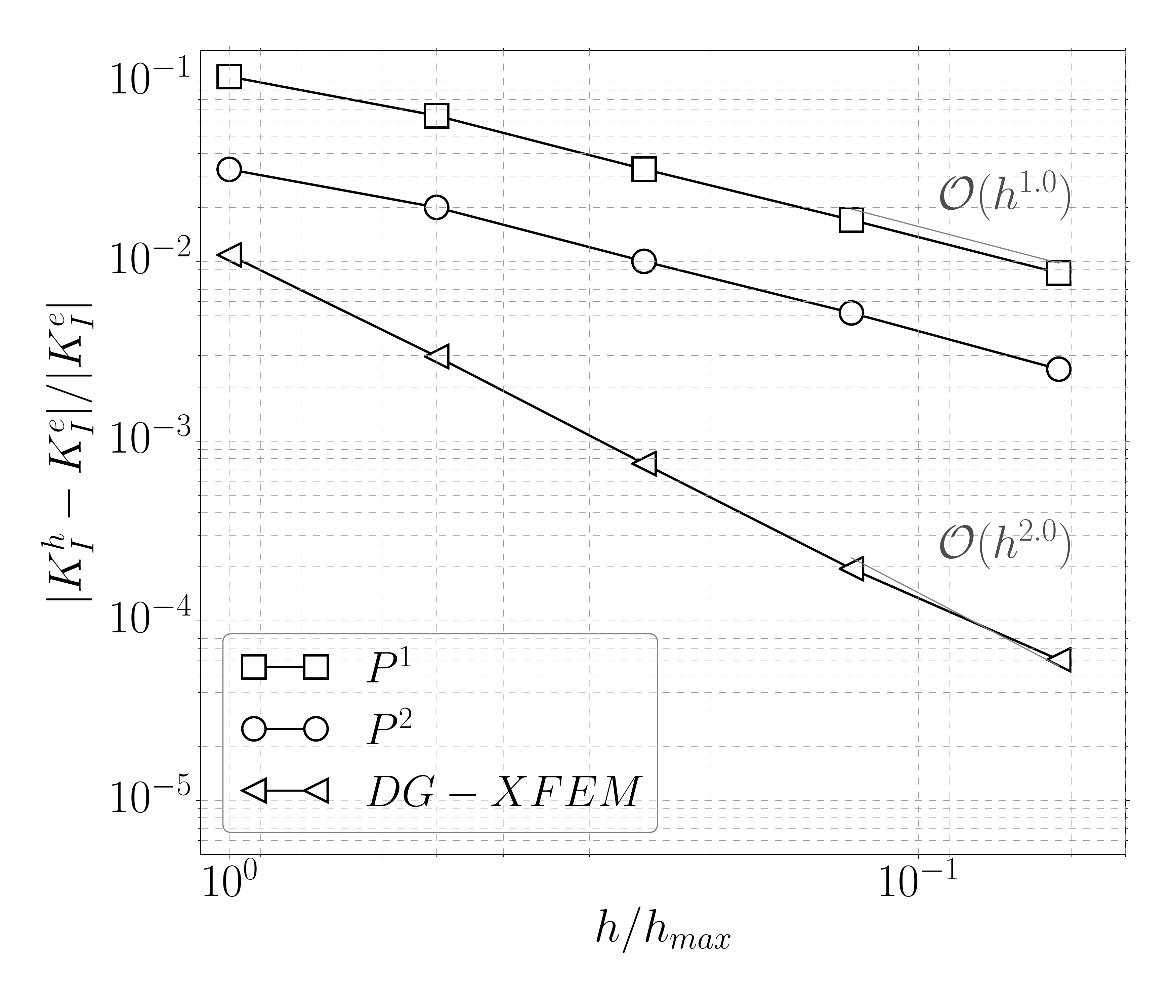} }
\subfloat[{ \color{black} $(\betab\au , \delta\gammab)= ( \betab^\text{DFC}_\text{II}, \delta\gammab^\text{TAN} )$. }] { 
\includegraphics[trim=0.5in 0.25in 0.5in 0.55in,clip,width=0.5\textwidth]{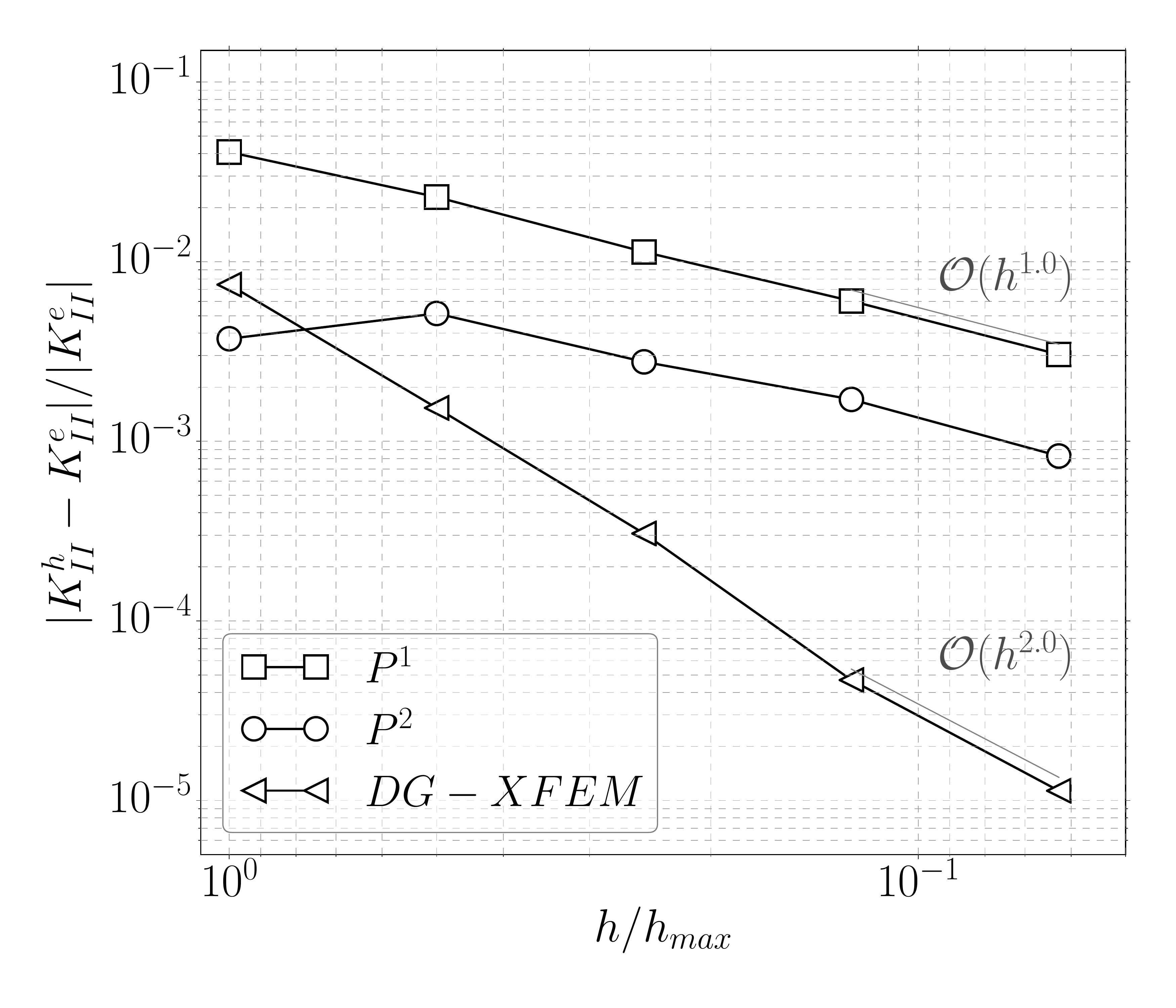} }
\caption{Convergence of the stress intensity factors for the power function crack problem }
\label{fig:convergece_sif_power_function_crack}
\end{figure} 
\fi
\begin{figure}[htbp]
\centering
\includegraphics[trim=0.5in 0.65in 0.725in 0.55in,clip,height=0.25\textheight]{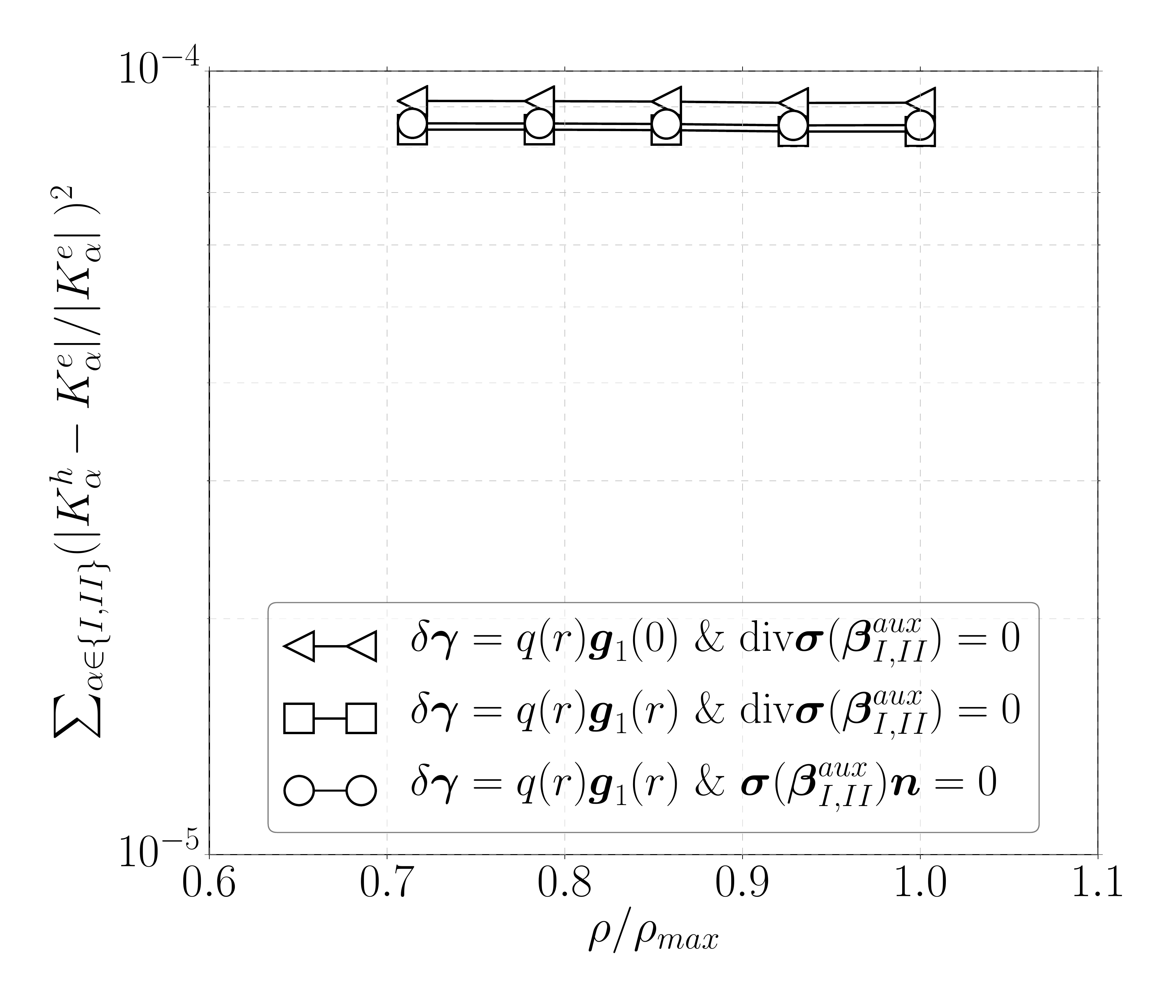}
\captionof{figure}{Contour independence of the interaction integral}
\label{fig:radial_independence_power_function_crack}
\end{figure}


\clearpage
\appendix
\section{Modes I and II Asymptotic Solutions}
\label{sxn:app1}
We recall below the components of displacement, gradient of displacements   and the stress fields for a straight crack lying on the axis $ \vartheta=\pm\pi, r \in [0,\infty)$ as derived in \cite{williams1957}. The components are given for a set of right-handed orthonormal basis with the $1$ axis aligned with the crack.

For $\kappa = 3 - 4\nu $ for plane strain and $\kappa = (3-4\nu)/(1+\nu)$ for plane stress, the displacements are given by
\eqnonumber{
\begin{aligned}
u^I_1 &=\sqrt{\frac{r}{2 \pi }} \frac{1}{2\mu} \cos \left(\frac{\vartheta }{2}\right) (\kappa-\cos (\vartheta )), \\
u^I_2 &= \sqrt{\frac{r}{2 \pi }} \frac{1}{2\mu}  \sin \left(\frac{\vartheta }{2}\right) (\kappa -\cos (\vartheta ));
\end{aligned}
}
\eqnonumber{
\begin{aligned}
u^{II}_1 &= \sqrt{\frac{r}{2 \pi }} \frac{1}{2\mu}  \sin \left(\frac{\vartheta }{2}\right) (2 + \kappa + \cos (\vartheta )), \\
u^{II}_2 &=  \sqrt{\frac{r}{2 \pi }} \frac{1}{2\mu}  \cos \left(\frac{\vartheta }{2}\right) (2 - \kappa -\cos (\vartheta )).
\end{aligned}
}

The gradient of displacements are given by 
\begin{equation*}
\begin{aligned}
\grad u_{11}^I(r,\vartheta)  &= \dfrac{1}{4 \mu \sqrt{2 \pi r }  } \cos \left(\dfrac{\vartheta }{2}\right) ( -\cos (\vartheta )+\cos (2 \vartheta )+\kappa-1),\\
\grad u_{12}^I(r,\vartheta) &= \dfrac{1}{4 \mu \sqrt{2 \pi r}   } \sin \left(\dfrac{\vartheta }{2}\right) (  \cos (\vartheta )+\cos (2 \vartheta )+\kappa + 1 ), \\
\grad u_{21}^I(r,\vartheta) &= \dfrac{1}{4 \mu \sqrt{2 \pi r}   }\sin \left(\dfrac{\vartheta }{2}\right) (\cos (\vartheta )+\cos (2 \vartheta ) - \kappa -1 ),\\
\grad u_{22}^I(r,\vartheta) &= \dfrac{1}{4 \mu \sqrt{2 \pi r}   }\cos \left(\dfrac{\vartheta }{2}\right) (\cos (\vartheta )-\cos (2 \vartheta )+ \kappa -1 ),
\end{aligned}
\end{equation*}
\begin{equation*}
\begin{aligned}
\grad u_{11}^{II}(r,\vartheta)  &=-\dfrac{1}{4 \mu \sqrt{2 \pi r }  }  \sin \left(\frac{\vartheta }{2}\right) (\cos (\vartheta )+\cos (2 \vartheta )+\kappa + 1),\\
\grad u_{12}^{II}(r,\vartheta) &= \dfrac{1}{4 \mu \sqrt{2 \pi r }  }   \cos \left(\frac{\vartheta }{2}\right) (-\cos (\vartheta )+\cos (2 \vartheta )+\kappa + 3),\\
\grad u_{21}^{II}(r,\vartheta) &=\dfrac{1}{4 \mu \sqrt{2 \pi r }  }\cos \left(\frac{\vartheta }{2}\right) (-\cos (\vartheta )+\cos (2 \vartheta )- \kappa + 1 ),\\
\grad u_{22}^{II}(r,\vartheta) &= \dfrac{1}{4 \mu \sqrt{2 \pi r }  }  \sin \left(\frac{\vartheta }{2}\right) (\cos (\vartheta )+\cos (2 \vartheta )-\kappa + 3).
\end{aligned}
\end{equation*}

Lastly the stress components are given by
\begin{equation*}
\begin{aligned}
\sigma_{11}^I(r,\vartheta)  &= \frac{1 }{\sqrt{2 \pi  r}}\left[1-\sin \left(\frac{\vartheta }{2}\right) \sin \left(\frac{3 \vartheta }{2}\right)\right] \cos \left(\frac{\vartheta }{2}\right), \\
\sigma_{22}^I(r,\vartheta) &= \frac{1}{\sqrt{2 \pi  r}} \left[\sin \left(\frac{\vartheta }{2}\right) \sin \left(\frac{3 \vartheta }{2}\right)+1\right] \cos \left(\frac{\vartheta }{2}\right), \\ 
\sigma_{12}^I(r,\vartheta) &= \frac{1 }{\sqrt{2 \pi  r}}\sin \left(\frac{\vartheta }{2}\right) \cos \left(\frac{\vartheta }{2}\right) \cos \left(\frac{3 \vartheta }{2}\right),
\end{aligned}
\end{equation*}
\begin{equation*}
\begin{aligned}
\sigma_{11}^{II}(r,\vartheta)  &= -\frac{1}{\sqrt{2 \pi  r}} \sin \left(\frac{\vartheta }{2}\right) \left[\cos \left(\frac{\vartheta }{2}\right) \cos \left(\frac{3 \vartheta }{2}\right)+2\right],\\
\sigma_{22}^{II}(r,\vartheta) &=\frac{1 }{\sqrt{2 \pi  r}}\left[1-\sin \left(\frac{\vartheta }{2}\right) \sin \left(\frac{3 \vartheta }{2}\right)\right] \cos \left(\frac{\vartheta }{2}\right),\\
\sigma_{12}^{II}(r,\vartheta) &=\frac{1}{\sqrt{2 \pi  r}} \sin \left(\frac{\vartheta }{2}\right) \cos \left(\frac{\vartheta }{2}\right) \cos \left(\frac{3 \vartheta }{2}\right).
\end{aligned}
\end{equation*}
Note that the above fields satisfy 
\begin{equation*}
\begin{aligned}
K_I\replaced[id=mc]{[\grad \ub^I]}{\left( \grad \ub^I \right)} = 1,& & K_{II}\replaced[id=mc]{[\grad\ub^I]}{\left( \grad \ub^I \right)} = 0,\\
K_I\replaced[id=mc]{[\grad \ub^{II}]}{\left( \grad \ub^{II} \right)} = 0,& & K_{II}\replaced[id=mc]{[\grad \ub^{II}]}{\left( \grad \ub^{II} \right)} = 1.
\end{aligned}
\end{equation*}

\section{Convergence of a continuous affine functional}
\label{sxn:convergence_functional}




In this appendix we first state and prove a proposition about the
convergence of linear and continuous functionals of a convergent
family of solutions.  This proof is essentially adapted from similar
results in \cite{buscaglia2000sensitivity,GuMa1994}. We next apply
this result to the interaction integral functional \eqref{eq:interint_tang_trac}.
\begin{prop*}
Let $V$ be a Hilbert space and $V^h \subset V$ be its finite
dimensional approximation. Let $a: V \times V\to \Rbb$  be a bilinear,
continuous and coercive form with $a(u,v)\le C_1\|u\|_V \|v\|_V$ for
all $u, v\in V$. Let $F:V \to \Rbb$ be linear and continuous
functional, and let $G:V\to \mathbb R$ be an affine and continuous
functional. 
Take $u$ to be the solution to $a(u,v) = F(v), \, \forall v \in V$ and $u^h$ the solution to $a(u^h,v^h) = F(v^h), \, \forall v^h \in V^h$. Let $w\in V$ be the unique member of $V$ such that
$a(v, w) = G(v)-G(0), \, \forall v \in V$.
We further assume that there exist positive real numbers $C_2$ and $C_3$ independent of $h$ such that
$\| u - u^h\|_V \leq C_2 h^k$ and  $\inf_{w^h \in V^h } \| w - w^h\|_V \leq C_3 h^k$.
Then there exists $C$ independent of $h$ such that
 \[
\left|G(u) - G\left(u^h\right) \right| \leq C h^{2k}.
 \]
\end{prop*} 

\begin{proof}
From the definition of $w$, 
\[
G(u) - G\left(u^h\right) = a\left(u - u^h, w\right).
\]
Furthermore note that for any $w^h \in V^h$,
\[
a\left(u-u^h, w\right) = a\left(u - u^h, w - w^h\right) + a\left(u - u^h, w^h\right) = a\left(u - u^h, w - w^h\right),
\]
where we have taken advantage of  Galerkin orthogonality, i.e.,
\[
a\left(u - u^h, w^h\right) = a\left(u, w^h\right) - a\left(u^h, w^h\right) = F\left(w^h\right) - F\left(w^h\right) = 0.
\]
Therefore we have
\[
\left|G(u) - G\left(u^h\right)\right| = \left|a\left(u - u^h, w-w^h\right)\right| \le C_1\left\|u - u^h\right\|_V\left\|w-w^h\right\|_V.
\]
Since $w^h$ is arbitrary, we have
\[
\left|G(u) - G\left(u^h\right)\right| \le C_1 \left\|u - u^h\right\|_V\inf_{w^h \in V^h} \left\|w-w^h\right\|_V \le C_1C_2C_3 h^{2k}.
\]
Taking $C=C_1C_2C_3$ yields the conclusion.
\end{proof}

The application of this proposition to the interaction integral
functionals here requires some additional work to account for the
difference between domains in curvilinear cracks, and the use of
quadrature rules. 
However, disregarding these differences, and
assuming that $\Bc^h_\Csc = \Bc_\Csc$ and that exact quadrature is
adopted, we have $\Ic^h=\Ic$, and for the standard finite element
method in \S \ref{subsxn:numerical_elasticity_solution}, $\Vc^h\subset
H^1(\Bc_\Csc; \Rbb^2)$, $\forall h$. Now, in its first argument, $\Ic[\betab,
\betab^\text{TF}, \delta \gammab^\text{TAN}]$ is affine and continuous
in $L^2(\Bc_\Csc; \Rbb^{2\times 2})$, so we set
$G(\ub) = \Ic[\grad \ub,\betab^\text{TF},\delta\gammab^\text{TAN}]$
(c.f.  \eqref{eq:interint_tang_trac}) and can use the above proposition. Since for the standard finite
element method it is known that  there exists  $C>0$ independent of $h$
such that $\| \ub -  \ub^h\|_{H^1(\Bc_\Csc; \Rbb^{2})} \le C h^k$, $\forall h$, then 
\[
\left| G(\ub) - G\left( \ub^h\right) \right| \leq \overline{C} h^{2 k }
\]
for some $\overline{C} \in \Rbb^+$.

This proposition is not directly applicable to the discontinuous
Galerkin method in \S\ref{subsxn:numerical_elasticity_solution},
because in this case it is also necessary to account for the use of an
approximation space that does not conform to $H^1$. Finally,  the two
functionals $G(\ub)=\Ic[\grad \ub,\betab^\text{DFC},\delta\gammab]$ in
\eqref{eq:interint_uni_dive} and \eqref{eq:interint_tang_dive} are not
continuous in  $H^1(\Bc_\Csc; \Rbb^2)$, because of the evaluations of
$\grad \ub$ on the crack faces, so we cannot directly apply the above result.



\clearpage
\bibliography{references}
\bibliographystyle{wileyj}

\end{document}